\documentclass[12pt]{amsart}
\usepackage{amsmath}
\usepackage{amsthm}
\usepackage{amsfonts}
\usepackage{amscd}
\usepackage{amssymb}
\usepackage{amstext}
\usepackage{hyperref}
\usepackage{epsfig}

\numberwithin{equation}{section}

\allowdisplaybreaks

\newtheorem{theorem}{Theorem}[section]
\newtheorem{proposition}{Proposition}[section]
\newtheorem{lemma}{Lemma}[section]
\newtheorem{remark}{Remark}[section]
\newtheorem{definition}{Definition}[section]
\newtheorem{corollary}{Corollary}[section]

\renewcommand{\epsilon}{\varepsilon}

\newcommand{\abs}[1]{\left\vert #1\right\vert}
\newcommand{\1}[1]{{\mathbf 1}{\{#1\}}}
\newcommand{\R}{\mathbb{R}}
\newcommand{\N}{\mathbb{N}}
\newcommand{\Z}{\mathbb{Z}}
\newcommand{\T}{\mathcal{T}}

\newcommand{\reff}{R_{\text{eff}}}

\def\fff#1{&{{\pageref{#1}}}\cr}
\def\hfff#1{\label{#1}}

\title[]{Scaling limit for the ant in high-dimensional labyrinths}
\date{}
\author[B.~Ben Arous]{G\'erard Ben Arous}
\address{G\'erard Ben Arous, Courant Institute of Mathematical Sciences, 251 Mercer Street, New York University,
New York, 12012-1185, U.S.A.} \email{benarous@cims.nyu.edu}
\author[M.~Cabezas]{Manuel Cabezas}
\address{Manuel Cabezas\\ Pontificia Universidad Cat\'{o}lica de Chile, Facultad de Matem\'{a}ticas, Campus San Joaqu\'{i}n, Avenida Vicu\~{n}a Mackenna 4860, Santiago, Chile.} \email{mncabeza@mat.puc.cl}

\author[A.~Fribergh]{Alexander Fribergh}
\address{Alexander Fribergh\\Universit\'e de Montr\'eal, DMS\\
Pavillon Andr\'e-Aisenstadt\\     2920, chemin de la Tour Montréal (Qu\'ebec),  H3T 1J4} \email{fribergh@dms.umontreal.ca}


\begin{document}

\begin{abstract}
We study here a detailed conjecture regarding one of the most important cases of anomalous diffusion, i.e the behavior of the "ant in the labyrinth".
It is natural to conjecture (see~\cite{Croydon_arc} and \cite{BJKS}) that the scaling limit for random walks on large critical random graphs exists in high dimensions, and is universal. This scaling limit is simply the natural Brownian Motion on the Integrated Super-Brownian Excursion. 

We give here a set of four natural sufficient conditions on the critical graphs and prove that this set of assumptions ensures the validity of this conjecture. The remaining future task is to prove that these sufficient conditions hold for the various classical cases of critical random structures, like the usual Bernoulli bond percolation, oriented percolation, spread-out percolation in high enough dimension.

In the companion paper ~\cite{BCFp}, we do precisely that in a first case, the random walk on the trace of a large critical branching random walk. We verify the validity of these sufficient conditions and thus obtain the scaling limit mentioned above, in dimensions larger than 14.
\end{abstract}

\maketitle

\section{Introduction}

Arguably one of the most important models in the study of anomalous diffusion is known as the "ant in the labyrinth", a term first coined by Gilles de Gennes in~\cite{deGennes1976}. This model refers to the study of the simple random walk on large critical, or infinite, percolation clusters on $\Z^d$. The physics literature on this topic is very rich and too broad to be covered in this introduction (see nevertheless \cite{havlin1987diffusion} and \cite{bouchaud1990anomalous}). The mathematical understanding is much more limited. A very important advance was achieved in 2009 by Gady Kozma and Assaf Nachmias in \cite{AO} when they proved that the Alexander-Orbach conjecture holds in high enough dimension, i.e.~that the spectral dimension is $4/3$, for critical bond percolation. This result was previously known for critical trees~\cite{Barlow_Kumagai} and for critical oriented percolation~\cite{BJKS}. A detailed discussion of the Alexander-Orbach conjecture can be found in~\cite{kumagai}.

Beyond the understanding of the critical exponent provided by the Alexander-Obach conjecture, not much is known. We strive here to understand the asymptotic behavior of random walks on large high-dimensional critical clusters in much more depth, and obtain a full description of its scaling limit. We put forward the conjecture that this scaling limit exists, and is universal in high enough dimensions, i.e does not depend on the precise nature of the percolation clusters. This scaling limit turns out to be the natural Brownian motion on Integrated Super-Brownian Excursion (denoted $B^{\text{ISE}}$), an object constructed by David Croydon in~\cite{Croydon_arc}. This universality conjecture is natural and was already suggested by David Croydon in~\cite{Croydon_arc} and the limiting object $B^{\text{ISE}}$ was proposed also as universal in~\cite{BJKS}.

Our main result (Theorem~\ref{thm_abstract}) gives the scaling limit of the simple random walk on general critical random graphs under a set of four natural conditions, which we conjecture hold quite generally. Even though we cannot prove, at this point, that they hold for bond percolation, we conjecture that they do.
We illustrate the abstract theorem proved here in the companion paper ~\cite{BCFp}. There we prove that our four sufficient conditions hold in an interesting but simpler case, i.e for the random walk on the trace of a large critical branching random walk, in dimensions larger than 14. We thus obtain that the scaling limit is indeed the $B^{\text{ISE}}$ in this case. 

The study of critical percolation in  high dimensions (currently meaning $d\geq 11$, see~\cite{fvdh1} and~\cite{fvdh2}) saw significant progress through the use of techniques known as lace expansion (for a recent survey see~\cite{PIMS}). Those techniques allowed a deep understanding of critical clusters in high dimensions (see~\cite{HaraSlade} or~\cite{TakashiVDHSlade}) and in particular opened the door to the proof of the Alexander-Orbach conjecture mentioned above. 

The technique of lace expansion was developed to study critical random environments in high dimensions. It is expected that several models such as critical branching random walks, oriented percolation, percolation and lattice trees have, in some sense, similar universal large scale behavior as explained in Section 6 of~\cite{van2006infinite}. It is thus natural to expect that the simple random walk on all of those random environments should have similar limiting behaviors. The goal of our work is to put this on a firm basis, and provide a general tool and a map for proofs of this universal scaling limit (the $B^{\text{ISE}}$) in the various important models of critical random clusters.

Our aim in this paper is thus to understand the natural conditions under which we can prove convergence of a simple random walk $(X^{G_n}_m)_{m\in \N}$ on large critical graphs $G_n$ towards the Brownian motion on the ISE. We isolate four conditions which are sufficient to obtain the expected limiting behavior. We, essentially, need to prove that for large $n$ (where $n$ quantifies size)
\begin{enumerate}
\item small parts of $G_n$, in the sense of the volume, are also small with respect to the intrinsic and  $\Z^d$ distances,
\item the graphs $G_n$ equipped with their intrinsic distances converge, in the finite dimensional sense, to the ISE,
\item the volume of $G_n$ is roughly uniformly distributed over the graph,
\item the resistance distance in $G_n$ is proportional to the intrinsic distance at a macroscopic level.
\end{enumerate}

Those conditions will be referred later as condition $(S)$, condition $(G)$, condition $(V)$ and condition $(R)$, standing respectively for "Skeleton approximation", "Graph convergence", "Volume uniform distribution", and "Resistance is linear".

As we mentioned, our aim in this paper is to provide a flexible theorem that will be applicable (or adaptable) in several models. To achieve this level of generality required we will have to introduce a certain amount of notations. The presentation of the model is done in Section~\ref{sect_abstract_cvg_thm} where the main theorem (Theorem~\ref{thm_abstract}) is stated.

The companion paper ~\cite{BCFp} provides an example of application of this abstract convergence theorem. There we show that the properly normalized random walk on the range of critical branching random walks converges to the Brownian motion on the ISE. This illustrates that the main result of this paper is indeed useful and may serve as a foundational step towards the analysis of the model of the ant in the labyrinth in other percolation models.

\subsection{Presentation of the main theorem}

The level of generality we aim for requires an important amount of notations. The detailed presentation of the model is done in Section~\ref{sect_abstract_cvg_thm}.

Despite this we will present the main theorem (Theorem~\ref{thm_abstract}) using some notations that will only be specified later. This is made in an effort to make the result appear first before focusing on a delicate construction of an object we call {\it skeleton of the graph},

The main idea behind our theorem is that large critical graphs $G_n$ are known, by lace expansion, to be tree-like, in the sense that there are no macroscopic loops for $n$ large. Hence, if we choose a certain integer $K$ and span $K$ points $V_i^n$ in an i.i.d.~fashion in the graphs $G_n$, it is highly likely that, for $n$ large and $K$ large (but independent of $n$), the random graph $G_n$ will be close (in distribution) to a random graph $\T^{(n,K)}$ which is a tree embedded in $\Z^d$. We call this graph the $K$-skeleton of $G_n$. Typically, we want to use points $V_i^n$ which are close to uniformly distributed. 

The construction of $\T^{(n,K)}$ is depicted in Figure 1. The bold points are the points $V_i^n$ selected to make the construction and the small points are the vertices of $\T^{(n,K)}$, the majority of which are not different from the points $V_i^n$ we spanned in order to build the tree. As mentioned $\T^{(n,K)}$ comes with a graph structure, spatial locations but it also comes with a natural metric (induced by $G_n$) and a measure given by projecting the volume of $G_n$ onto $\T^{(n,K)}$.

  \begin{figure}
  \includegraphics[width=0.7\linewidth]{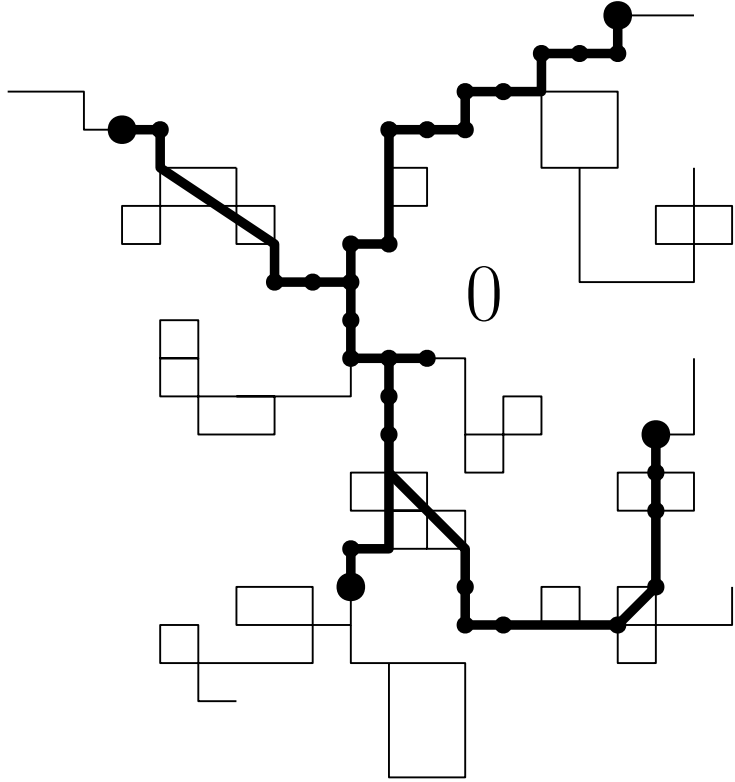}
  \caption{Construction of the $4$-skeleton of a graph where the points $V_1^n, \ldots, V_4^n$ are shown as thick circles.}
\end{figure}

\subsubsection{Condition $(S)$}

Condition $(S)$ (which is defined precisely in Section~\ref{sect_asympT_thin}) guarantees that the construction of the skeleton can be done and that it asymptotically is a good approximation, provided $K$ is large, of $G_n$.

\subsubsection{Condition $(G)$}

In $\T^{(n,K)}$, if we consider the points $(V_i^n)_{i\leq K}$ and the branching points of $\T^{(n,K)}$, we have a tree with at most $2K+1$ points. This tree is naturally equipped with a distance between those points, which means that $\T^{(n,K)}$ is a graph spatial tree (a term defined in Section~\ref{sect_graph_tree}). We rescale the distances on $\T^{(n,K)}$ by $n^{1/2}$.

By taking $K$ points uniformly at random on the ISE (see Section~\ref{def_ISE_sect}), we can construct an object called the $K$-ISE in the same way that we constructed $\T^{(n,K)}$.

Condition $(G)$ states that $\T^{(n,K)}$ converges weakly to the $K$-ISE as $n$ goes to infinity in a natural topology on graph spatial trees (see Definition~\ref{def_condG}).

\subsubsection{Condition $(V)$}

As mentioned previously $\T^{(n,K)}$ is equipped with a measure associated with the volume. This measure rescaled by $n$ is called $\mu^{(n,K)}$ (see Section~\ref{sect_mu}). Another natural measure on $\T^{(n,K)}$ is the Lebesgue measure rescaled to have mass 1 (recall that $\T^{(n,K)}$ is a finite tree with a distance). This measure is called $\lambda^{(n,K)}$ (see Section~\ref{sect_prop_skel}). 

Denote $\overrightarrow{\T^{(n,K)}_{x}}$ are the descendants of $x$ (including $x$ itself) in $\T^{(n,K)}$. Condition $(V)$ states that  there exists $\nu>0$ such that for $\epsilon>0$
\[
\limsup_{K\to \infty}\limsup_{n\to \infty} {\bf P}_n \Bigl[\sup_{x\in \T^{(n,K)}} \abs{\nu \lambda^{(n,K)}\Bigl(\overrightarrow{\T^{(n,K)}_{x}}\Bigr)-\mu^{(n,K)}\Bigl(\overrightarrow{\T^{(n,K)}_{x}}\Bigr)}>\epsilon\Bigr] =0.
\]

\subsubsection{Condition $(R)$}

Condition $(R)$ states that there exists $\rho>0$ such that for all $\epsilon>0$ and for all $i\in \N$
\[
\lim_{n\to \infty} {\bf P}_n\Bigl[ \abs{\frac{R^{G_n}(0,V_1^n)}{d^{G_n}(0,V_1^n)}-\rho}>\epsilon\Bigr]=0,
\]
where $R^{G_n}$ is the resistance distance in $G_n$ and $d^{G_n}$ is the intrinsic distance>

\subsection{Main result}

Even though we still need more details to defined properly our notation, we state our main theorem.
\begin{theorem}\label{thm_abstract}
Consider a sequence of  random graphs $(G_n,(V_i^n)_{i\in \N})_{n\in N}$ chosen under ${\bf P}_n$ which verifies conditions $(S)$, $(G)_{\sigma_d,\sigma_{\phi}}$, $(V)_{\nu}$ and $(R)_{\rho}$. Denoting $(X^{G_n}_m)_{m\in \N}$ the simple random walk on $G_n$ started at $0$, we have that 
\[
( n^{-1/4}X_{tn^{3/2}}^{G_n})_{t\geq 0} \to (\sigma_{\phi}\sqrt{\sigma_d} B^{ISE}_{(\rho \nu \sigma_d)^{-1}t})_{t\geq 0},
\]
the convergence is annealed and occurs in the topology of uniform convergence over compact sets.
\end{theorem}

\begin{remark} The reader will have noticed that constants appear in conditions $(G)_{\sigma_d,\sigma_{\phi}}$, $(V)_{\nu}$ and $(R)_{\rho}$. The precise definition of those conditions is postponed to Section~\ref{GRRR}.
\end{remark}

\subsection{Discussing the universality of the Brownian motion on the ISE}

\subsubsection{Relation with the Alexander-Orbach conjecture}
Our main result, Theorem~\ref{thm_abstract}, can be seen as a scaling limit counterpart of the Alexander-Orbach conjecture (which identifies the scaling exponent). In particular, from our result, the exponent predicted by the Alexander-Orbach conjecture can be deduced, which, to the best of our knowledge, is a new result in the context of simple random walk on critical branching random walks (see the companion article~\cite{BCFp}).

 In this light, our abstract convergence theorem, Theorem~\ref{thm_abstract}, can be seen as more delicate version of~\cite{KM} which gave sufficient conditions to prove the Alexander-Orbach conjecture which were subsequently applied to prove that conjecture in the case of critical percolation (see~\cite{KN} and~\cite{HHH} for a generalization)

\subsubsection{A universal scaling limit up to some caveats}\label{sect_caveats}
We conjecture that the Brownian motion on the ISE is the scaling limit of the simple random walk on large critical percolation (clusters conditioned on size) in $\Z^d$ for large $d$. We actually believe that this processes is, in some sense, universal up to two points that we will describe next.
\begin{itemize}
\item One issue stems from the fact that the environment can be chosen to be large in several manners. Conditioning the environment to have a large cardinality or to reach a large distance from the origin will result in different scaling limits. This will necessarily result in different scaling limits for the walk on those environments. Nevertheless, all those scaling limits are super-Brownian motions under certain conditionings, of which the ISE is a particular example (we refer the reader to~\cite{Dawson}, \cite{Legallsnake} and~\cite{Perkins} for surveys on the super-Brownian motion). For any classical conditionings for large clusters, we do believe that methods developed in this paper would be sufficient to obtain a variant of Theorem~\ref{thm_abstract}, with an alternate Condition (G) and a limiting process which would be a Brownian motion on a certain conditioned super-Brownian motion. 
\item Another potential problem comes from the fact that certain models, namely oriented percolation in $\Z^+\times \Z^d$, have a directed nature to them. In this setting, the Brownian motion on the ISE is not a natural candidate for the scaling limit since this process is inherently isotropic. However, in this case the canonical scaling limit in this context is still intimately related to the Brownian motion on the ISE. Neglecting the minor issues related to the precise signification of ``large", the natural scaling limit is Brownian motion in the oriented ISE defined by $((d_{\phi_{\mathfrak{T}}(\mathfrak{T})}(B^{CRT}_t), \phi_{\mathfrak{T}}(B^{CRT}_t)))_{t\geq 0}$ (in the notations of Section~\ref{section_bmsbm}). As one may see from Proposition~\ref{propdef_BISE}, this is strongly tied to the definition of the Brownian motion on the ISE.
\end{itemize}

Those examples show that the scaling limit behaviors are more diverse than the scaling exponents which are universal among a large class of models.

\subsubsection{Random walks on infinite critical structures}

In general, the critical graphs that we are interested in are large but finite. It is possible to define infinite critical structures, for example, as mentioned in the introduction, this has been done for critical percolation in dimensions $d=2$ and $d\geq 11$ (the \emph{Incipient Infinite Cluster}). It is natural to wonder what would happen when considering diffusions on those infinite structures. Obviously the Brownian motion on the ISE cannot be the scaling limit, since this process is restricted to the ISE which is a finite object. However, we believe that the ideas we developed in this paper will prove to be sufficient for the analysis of such models. The natural counter-part of the Brownian motion on the ISE on an infinite structure would be the Brownian motion on the infinite canonical super-Brownian motion, indeed the infinite canonical super-Brownian motion is the natural scaling limit for infinite critical structures (see~\cite{van2006infinite}).

We also want to emphasize the links between the Brownian motion on the infinite canonical super-Brownian motion and Spatially Subordinated Brownian Motions (SSBM) introduced in~\cite{arous2015randomly}. Indeed, if the former object were defined we could show that its projection onto the backbone (unique infinite simple path) would be an SSBM. This ties in with recent work~\cite{arous2016scaling} showing that the projection onto the backbone of a simple random walk on the infinite critical Galton-Watson tree converges to an SSBM.

\subsubsection{What are high dimensions?}

The notion of high dimension is highly dependent on the model that is considered, this is already known from lace expansion (see for example~\cite{van2006infinite}).

In our analysis of diffusions on large critical structures, we are limited by two factors
\begin{enumerate}
\item our limiting process is only defined in $\Z^d$ for $d\geq 8$,
\item our methods of proof requires the typically distance between two consecutive cut-points is microscopic for large $n$. We believe this to be one of the main limiting factors.
\end{enumerate}

\subsection{Notations}

Given a graph $G$, we will denote $V(G)$ the set of its vertices and $E(G)$ the set of its edges. For $x\in G$ and $k\in \R$, we will write $B_G(x,k)$ for the ball of $k$ centered at $x$ in the natural metric induced by $G$.

The constants in this paper will typically be denoted $c$ (for lower bounds) and $C$ (for upper bounds) and implicitly assumed to be positive and finite. Their value may change from line to line.

This paper contains a significant amount of notation, so we decided to include a glossary of notation at the end of the paper to help the reader.

\section{The Brownian motion on the ISE}\label{section_bmsbm}

\subsection{Real trees and spatial trees}

Before defining the Brownian motion on the  ISE it is necessary to actually define the ISE. For this, we choose to introduce the formalism of real trees and spatial trees of which the ISE is the canonical random example. For this we follow, almost to the word, notes from Le Gall (see~\cite{LG}).

\subsubsection{Real trees}\label{sect_real_tree}

\begin{definition}
A metric space $(T,d_{T})$ is a real tree if the following two
properties hold for every $\sigma_1,\sigma_2\in T$.

\begin{enumerate}
\item There is a unique
isometric map
$f_{\sigma_1,\sigma_2}$ from $[0,d_{T}(\sigma_1,\sigma_2)]$ into $T$ such
that $f_{\sigma_1,\sigma_2}(0)=\sigma_1$ and $f_{\sigma_1,\sigma_2}(
d_{T}(\sigma_1,\sigma_2))=\sigma_2$.
\item If $q$ is a continuous injective map from $[0,1]$ into
$T$, such that $q(0)=\sigma_1$ and $q(1)=\sigma_2$, we have
$q([0,1])=f_{\sigma_1,\sigma_2}([0,d_{T}(\sigma_1,\sigma_2)]).$
\end{enumerate}

A rooted real tree is a real tree $(T,d_{T})$
with a distinguished vertex called the root.
\end{definition}

Let us consider a rooted real tree $(T,d)$.
The range of the mapping $f_{\sigma_1,\sigma_2}$ in (1) is denoted by
$[ \sigma_1,\sigma_2]$ (this is the line segment between $\sigma_1$
and $\sigma_2$ in the tree). 
In particular, for every $\sigma\in T$, $[ \text{root},\sigma]$ is the path 
going from the root to $\sigma$, which we will interpret as the ancestral
line of $\sigma$. More precisely we can define a partial order on the
tree by setting $\sigma\preccurlyeq \sigma'$
($\sigma$ is an ancestor of $\sigma'$) if and only if $\sigma\in [ \text{root},\sigma']$, and, $\sigma \prec \sigma'$ if $\sigma\preccurlyeq \sigma'$ and $\sigma \neq \sigma'$.

If $\sigma,\sigma'\in T$, there is a unique $\eta\in T$ such that
$[ \text{root},\sigma ] \cap [\text{root} ,\sigma']=[\text{root},\eta ]$. We write $\eta=\sigma\wedge \sigma'$ and call $\eta$ the most recent
common ancestor to $\sigma$ and $\sigma'$.

Finally, let us observe that for any three points $\sigma_1,\sigma_2,\sigma_3$ of a real tree $T$ there exists a unique branch-point $b^{T}(\sigma_1,\sigma_2,\sigma_3)\in T$ that satisfies $b^{T}(\sigma_1,\sigma_2,\sigma_3)\in \T=[\sigma_1,\sigma_2]\cap [\sigma_2,\sigma_3]\cap [\sigma_3,\sigma_1].$

There are collections of real trees that cannot be distinguished as metric
spaces. For compact rooted real
trees (which are the only type of real trees we consider in this paper) two rooted real trees are equivalent if and only if
there exists a root preserving isometry between them. For our purposes, this subtlety will not be relevant and we will not make any distinction between a tree and its equivalence class. See~\cite{LG} for more details.

\vspace{0.5cm}

{\it A way to construct real trees }

\vspace{0.5cm}

There is a simple way of constructing compact real trees. We consider a	 (deterministic) continuous function
$g:[0,\infty)\longrightarrow[0,\infty)$ with compact support
and such that $g(0)=0$ and $g(x)=0$ for $x$ large but $g$ is not identically zero.

For every $s,t\geq 0$, we set
\[
m_g(s,t)=\inf_{r\in[s\wedge t,s\vee t]}g(r),
\]
and
\[
d_g(s,t)=g(s)+g(t)-2m_g(s,t).
\]


We then introduce the equivalence relation
$s\sim t$ iff $d_g(s,t)=0$ (or equivalently iff $g(s)=g(t)=m_g(s,t)$). Let
$T_g$ be the quotient space
\begin{equation}\label{def_equiv_tree}
T_g=[0,\infty)/ \sim.
\end{equation}

Obviously the function $d_g$ induces a distance on $T_g$, and we keep the
notation $d_g$ for this distance. Viewing the equivalence class of $0$ as the root, this means we have the following (see~\cite{DuLG})

\begin{theorem}
\label{tree-deterministic}
The metric space $(T_g,d_g)$ is a rooted real tree.
\end{theorem}

We will call $T_g$  the real tree coded by $g$.

\subsubsection{Spatial trees}\label{sect_spatial_tree}

\begin{definition} A ($d$-dimensional) spatial tree is a pair $(T,\phi_{T})$ where 
$T$ is a real tree and $\phi_T$ is a continuous mapping from 
$\phi_T$ into $\R^d$. 
\end{definition}

\begin{remark}
Two spatial trees $(T,\phi_T)$ and $(T', \phi_{T'})$ are said to be equivalent if and only if there exists a root preserving isometry $\pi$ from $T$ to $T'$ such that $\phi_{T}=\phi_{T'}\circ \pi$. In Section~\ref{sect_condG}, we will define a topology on spatial trees, which will actually be a topology on the equivalence classes of spatial trees with respect to the previous relation. Nevertheless, for our purposes it does not pose a problem to identify a tree with its equivalence class.
\end{remark}

Let $T$ be a compact rooted real tree  with a metric $d$. We may consider the $\R^d$-valued Gaussian process
$(\phi_{T}(\sigma),\sigma\in T)$ whose distribution is characterized by
\begin{align*}
&E[\phi_{T}(\sigma)]=0\;,\\
&{\rm cov}(\phi_{T}(\sigma),\phi_{T}(\sigma'))=d(\text{root},\sigma\wedge \sigma')\,{\rm Id}\;,
\end{align*}
where ${\rm Id}$ denotes the $d$-dimensional identity matrix. 

This corresponds to a Brownian embedding of $T$ into $\R^d$. The formula for the covariance is easy to understand if we recall that $\sigma\wedge \sigma'$ is the most recent common ancestor to $\sigma$
and $\sigma'$, and so the ancestors of $\sigma$ and $\sigma'$ are the same up to
level $d(\text{root},\sigma\wedge \sigma')$. 

Under certain assumptions, that will be verified in our context (see (8) in~\cite{LG} for details), the process $(\phi_{T}(\sigma),\sigma\in T)$ has a continuous modification. We keep the notation $\phi_T$ for this modification. 

Given a real tree $T$, we denote by $Q_T$ the law of the spatial tree $(T,(\phi_{T}(\sigma),\sigma\in T))$ (provided it exists).

\subsubsection{Graph spatial trees}\label{sect_graph_tree}

Let us now present a notion introduced by Croydon in~\cite{Croydon_arc}.

\begin{definition}
If a spatial tree $(T,(\phi_T(\sigma),\sigma\in T))$ is such that $T$ is a finite tree with finite edge length, we say that $(T,(\phi_T(\sigma),\sigma\in T))$ is a graph spatial tree.
\end{definition}

 Given a graph spatial tree $(T,(\phi_T(\sigma),\sigma\in T))$, we can assign a probability measure $\lambda_{T}$  defined as the renormalized Lebesgue measure \hfff{lambdaT}(so that the $\lambda_{T}$-measure of a line segment in $T$ is proportional to its length).

\vspace{0.5cm}

{\it A simple way to construct graph spatial  trees }

\vspace{0.5cm}

There is a simple way to construct a rooted graph spatial tree from a rooted spatial tree $(T,d_{T},\phi_T)$. For this we consider a sequence $(\sigma_i)_{i\in \N}$ of elements of a real tree $T$. Fix $K\in \N$. We define the reduced subtree $T(\sigma_1,\ldots,\sigma_K)$ to be the graph tree with vertex set 
\[
V(T(\sigma_1,\ldots, \sigma_K)):=\{b^{T}(\sigma,\sigma',\sigma''):\sigma,\sigma',\sigma'' \in \{\text{root},\sigma_1,\ldots,\sigma_K\}\},
\]
and graph tree structure induced by the arcs of $T$, so that two elements $\sigma$ and $\sigma'$ of $V(T(\sigma_1,\ldots,\sigma_K))$ are connected by an edge if and only if $\sigma\neq \sigma'$ and also $[\sigma,\sigma']\cap V(T(\sigma_1,\ldots,\sigma_K))=\{\sigma,\sigma'\}$. We set the length of an edge $\{\sigma,\sigma'\}$ to be equal to $d_{T}(\sigma,\sigma')$ and we extend the distance linearly on that edge. This allows us to view $T(\sigma_1,\ldots,\sigma_K)$ as a graph spatial tree.

This spatial graph tree will be denoted $(T^{K,(\sigma_i)},d_{T^{K,(\sigma_i)}}, \phi_{T^{K,(\sigma_i)}})$. The associated normalized probability measure is denoted $\lambda_{\phi_{T^{K,(\sigma_i)}}(T^{K,(\sigma_i)})}$. The dependence on $\sigma$ will often be dropped in the notation when the context is clear.

\subsection{Definition of the $CRT$, the $ISE$ and the $B^{ISE}$} \label{def_ISE_sect}

In this section our goal is to introduce the canonical random object associated to real trees, spatial trees, graph spatial trees and motions of spatial trees.

\subsubsection{The continuum random tree (CRT)}

Denote by $({\bf e}_t)_{0\leq t\leq 1}$ a normalized Brownian excursion. Informally, $({\bf e}_t)_{0\leq
t\leq 1}$ is just a Brownian path started at the origin and conditioned to stay positive
over the time interval $(0,1)$, and to come back to $0$ at time $1$ (see e.g.~Sections 2.9 
and 2.12 of It\^o and McKean \cite{IM} for a discussion of the normalized excursion). We
extend 
the definition of ${\bf e}_t$ by setting ${\bf e}_t=0$ if $t>1$. Then the (random) function
${\bf e}$ satisfies the assumptions of Section~\ref{sect_real_tree} and we can thus consider the
random real tree $T_{\bf e}$.

\begin{definition}
\label{CRTdef}
The random real tree $T_{\bf e}$ is called the Continuum Random Tree (CRT) and will be  denoted $(\mathfrak{T},d_{\mathfrak{T}})$ \hfff{crt}. We write $\Xi$ to denote its law.
\end{definition}

The CRT was initially defined by Aldous \cite{Al1} with a different formalism,
but the preceding definition corresponds to Corollary 22 in \cite{Al3}, up to an unimportant
scaling factor $2$.

We can define a natural volume measure on $\mathfrak{T}$ by projecting the Lebesgue measure on $[0,1]$, i.e., for any open $A\subseteq \mathfrak{T}$, we set
\[
\lambda^{\mathfrak{T}}(A)=\text{Leb}\{t\in [0,1], [t]\in A\},
\]
where $[t]$ denotes the equivalence class of $t$ with respect to the relation defined at~\eqref{def_equiv_tree}.

One major motivation for studying the CRT is the fact that it occurs as the
scaling limit of large critical Galton-Watson trees. In particular, recalling the notations of the introduction, we have the following (see Theorem~3.1 in~\cite{LG} which is a simple consequence of Theorem 23 in~\cite{Al3})
\begin{theorem} 
Assume that we are given a critical offspring distribution $Z$ that has finite variance $\sigma^2>0$ and which is aperiodic. Denoting $\text{GW}_n$ a Galton-Watson tree conditioned to have cardinal $n$, we have that the rescaled real tree $(\text{GW}_n,\frac{\sigma}{2\sqrt n} d_{\text{GW}_n})$ converges to the CRT, where  the convergence occurs in distribution with the Gromov-Hausdorff topology. 
\end{theorem}

\subsubsection{The integrated super-Brownian excursion (ISE)}

We will combine the CRT with $d$-dimensional Brownian motions started from $x=0$, in the
way explained in Section~\ref{sect_spatial_tree}. Precisely this means that we are
considering the probability measure on spatial trees defined by
\[
M=\int \Xi(d\mathfrak{T})\,Q_{\mathfrak{T}}.
\]

Recall the notation $\lambda_{\mathfrak{T}}$ for the uniform measure on $\mathfrak{T}$
(this makes sense $\Xi(d\mathfrak{T})$ a.s.).

\begin{definition}
The random probability measure $\lambda^{\phi_{\mathfrak{T}}(\mathfrak{T})}$ on $\R^d$ defined under $M$ by $\lambda^{\phi_{\mathfrak{T}}(\mathfrak{T})}:=\lambda^{\mathfrak{T}} \circ \phi_{\mathfrak{T}}^{-1}$ is called $d$-dimensional ISE (for Integrated Super-Brownian Excursion).
\end{definition}

Note that the topological support of ISE is the range 
of the spatial tree, and that ISE should be interpreted as the
uniform measure on this set. We will often abuse the terminology and write ISE to mean its topological support. We will write $\phi_{\mathfrak{T}}(\mathfrak{T})$ \hfff{ise} to designate this set.

The random measure
ISE was first discussed by Aldous~\cite{Al4}. It occurs in various 
asymptotics for models of statistical mechanics (see in particular~\cite{Sl1} and~\cite{Sl2}).

\subsubsection{The Brownian motion on the ISE: $B^{ISE}$}

We are now going to define a canonical dynamic on the ISE. For this we will start by discussing the Brownian motion on the CRT.

Let $(T,d_T)$ be any real tree. It was suggested by Aldous~\cite{Al2} that a Brownian motion on $(T,d_{T},\nu)$ should be a strong Markov process with continuous sample paths that is reversible with respect to its invariant measure $\nu$ and satisfies the following properties,
\begin{enumerate}
\item For $\sigma_1,\sigma_2 \in T$ with $\sigma_1\neq \sigma_2$, we have 
\[
P_{\sigma}^{T,\nu}(T_{\sigma_1}<T_{\sigma_2})=\frac{d_{T}(b^{T}(\sigma,\sigma_1,\sigma_2),\sigma_2)}{d_{T}(\sigma_1,\sigma_2)}, \qquad \forall \sigma \in T,
\]
where $T_{\sigma}:=\inf\{t>0,X_t^{T}=\sigma\}$ is the hitting time of $\sigma \in T$.
\item For $\sigma_1,\sigma_2 \in T$, the mean occupation measure for the process started at $\sigma_1$ and killed on hitting $\sigma_2$ has density
\[
2d_{T}(b^{T}(\sigma,\sigma_1,\sigma_2),\sigma_2)\nu(d\sigma) \qquad \forall \sigma \in T.
\]
\end{enumerate}

These properties guarantee the uniqueness of the Brownian motion on $(T,d_{T},\nu)$. 

The existence of such a process follows from techniques of resistance forms (see~\cite{Kigami_Harm} for an introduction on resistance forms). More specifically, it was proved in Section 6 of~\cite{Croydon_arc} that
\begin{proposition}\label{prop_def_process}
Let $(T,d_{T})$ be a compact real tree, $\nu$ be a finite Borel measure on $\T$ that satisfies $\nu(A)>0$ for every non-empty open set $A\subseteq T$, and $(\mathcal{E}_{T},\mathcal{F}_{T})$ be the resistance form associated with $(T,d_{T})$. Then $(\frac 12 \mathcal{E}_{T},\mathcal{F}_{T})$ is a local, regular Dirichlet form on $L^2(T, \nu)$, and the corresponding Markov process $B^{T,\nu}$ is the Brownian motion on $(T,d_{T},\nu)$.
\end{proposition}

For $d\geq 8$, it can be proved that $\phi_{\mathfrak{T}}$ is injective from $\mathfrak{T}$ (the CRT)  to $\phi_{\mathfrak{T}}(\mathfrak{T})$ (the range of the ISE), see Proposition 3.5.~in~\cite{Croydon_arc}. This means $\phi_{\mathfrak{T}}$ is actually an isometry between $(\mathfrak{T},d_{\mathfrak{T}})$ and $(\phi_{\mathfrak{T}}(\mathfrak{T}),d_{\phi_{\mathfrak{T}}(\mathfrak{T})})$ which sends  $\lambda_{\mathfrak{T}}$ to $\lambda_{\phi_{\mathfrak{T}}(\mathfrak{T})}$.  Hence, $(\phi_{\mathfrak{T}}(\mathfrak{T}),d_{\phi_{\mathfrak{T}}(\mathfrak{T})})$ is a real tree and this allows us to define easily a process, which in the sense defined by Aldous, is the Brownian motion on the ISE in $\Z^d$ for $d\geq 8$.

\begin{proposition}\hfff{bise}\label{propdef_BISE} For $\Xi$-a.e.~$\mathfrak{T}$, the Brownian motion $B^{CRT}$ on $(\mathfrak{T},d_{\mathfrak{T}},\mu^{\mathfrak{T}})$ exists. Furthermore if $d\geq 8$, for $M$-a.e.~ $(\mathfrak{T},\phi)$, the Brownian motion $B^{ISE}$ on $(\phi_{\mathfrak{T}}(\mathfrak{T}),d_{\phi_{\mathfrak{T}}(\mathfrak{T})},\mu^{\phi(\mathfrak{T})})$ exists and, moreover, $B^{ISE}=\phi_{\mathfrak{T}}(B^{CRT})$.
\end{proposition}

Since the map $\phi_{\mathfrak{T}}:\mathfrak{T}\to \R^d$ is continuous for $M$-a.e.~spatial trees $(\mathfrak{T},\phi_{\mathfrak{T}})$ the law of $B^{ISE}$ is, $M$-a.s., a well-defined probability measure on $C(\R_+,\R^d)$.

\subsubsection{Approximating the ISE and the $B^{ISE}$ using graph spatial trees}\label{sect_KISE}

It will be useful for us to approximate the ISE (and the $B^{ISE}$) by a graph spatial tree (and a process on this graph spatial tree). Indeed, the topological structure of the graph spatial tree is much simpler which makes it easier to study. This is an idea that was already used by Croydon in Section 8 of~\cite{Croydon_arc} (building on ideas he developed in~\cite{Croydon_crt})

Consider $\mathfrak{T}$ a realization of the CRT and $(U_i)_{i\in \N}$ chosen according to $(\lambda^{\mathfrak{T}})^{\otimes \N}$.
Fix $K\in \N$ . We can use the construction described in Section~\ref{sect_graph_tree} to define a graph spatial tree, which we call $K$-ISE and denote it  $(\mathfrak{T}^{(K)},d_{\mathfrak{T}^{(K)}},\phi_{\mathfrak{T}^{(K)}})$, where $\mathfrak{T}^{(K)}$ is called $K$-CRT \hfff{kcrt}. We recall that this object comes with a probability measure $\lambda_{\phi_{\mathfrak{T}^{(K)}}(\mathfrak{T}^{(K)})}$. For the sake of simplicity we will denote $\lambda_{\phi_{\mathfrak{T}^{(K)}}(\mathfrak{T}^{(K)})}$ as $\lambda_{\frak{T}}^{(K)}$.

 It is also interesting to note that $\mathfrak{T}^{(K)}$ has no point of degree more than 3, indeed, by Theorem 4.6 in~\cite{DuLG}, it is known that $\Xi$-a.s.~for any $x\in \mathfrak{T}$ the set $\mathfrak{T}\setminus \{x\}$ has at most three connected components.

Once again, Proposition~\ref{prop_def_process} allows us to define a Brownian motion \hfff{bcrtk}
$B^{(K)}$ in $(\frak{T}^{(K)},d_{\frak{T}},\lambda_{\frak{T}}^{(K)})$, where $\lambda_{\frak{T}^{(K)}}$ is the Lebesgue measure in $\frak{T}^{(K)}$ normalized to be a probability measure. Also, we define the Brownian motion $B^{K-ISE}$ \hfff{bisek} on the $K$-ISE $(\phi^{(K)}(\mathfrak{T}^{(K)}),d_{\phi^{(K)}(\mathfrak{T}^{(K)})},\lambda_{\phi^{(K)}(\mathfrak{T}^{(K)})})$. It can be shown (in essence equation (8.3) of~\cite{Croydon_arc}) that

\begin{proposition} \label{prop_approx_bisek}
We have that $B^{K-ISE}$ converges to $B^{ISE}$ as $K\to \infty$, in distribution  in the topology of uniform convergence in $C(\R_+,\R^d)$ for $M \otimes (\lambda^{\mathfrak{T}})^{\otimes \N}$-a.e.~realization of $(\mathfrak{T},d_{\mathfrak{T}},\phi_{\mathfrak{T}},(U_i)_{i\in \N})$.
\end{proposition}

\subsubsection{Local times of the $B^{ISE}$ and the $B^{K-ISE}$}

We recall the definition of local times for processes taking values on real trees.
\begin{definition}\label{def:localtime}
Let $T$ be an $\R$-tree equipped with a Borel measure $\mu$. Consider a process $(X_t)_{t\geq0}$ taking values in $T$. We say that a random process $(L_t(x))_{x\in T, t\geq0}$ is a local time for $X$ if
 \[\text{Leb}\{s\leq t:X_s\in A\}=\int_AL_t(x)\mu(dx), \]
 for any $A$, Borelian of $T$.
\end{definition}


    
 Let $(L^{(K)}_t(x))_{x\in\frak{T}^{(K)},t\geq0}$ be the local time of $B^{(K)}$ with respect to the measure $\lambda_{\frak{T}}^{(K)}$ and $(L_t(x))_{x\in\frak{T},t\geq0}$ be the local time of $B^{CRT}$ with respect to $\mu^{\frak{T}}$. It is a known fact that $L^{(n,K)}$ and $L^{(K)}$  \hfff{ltk} exists for almost every realization of $\frak{T}$, $\frak{T}^{(K)}$ and moreover, they can be chosen to be jointly continuous in $t$ and $x$ (see Lemma 3.3 in \cite{Croydon_crt}). 
 
Since $\Phi_{\frak{T}}$ is a isometry, it is elementary to see that the corresponding local times of $B^{K-ISE}$ and $B^{ISE}$ are just $(L^{(K)}_t(\Phi_{\frak{T}}^{-1}(x)))_{x\in\Phi_{\frak{T}}(\frak{T}^{(K)}),t\geq0}$ and $(L_t(\Phi^{-1}_{\frak{T}}(x)))_{x\in\Phi_{\frak{T}}(\frak{T}),t\geq0}$.

\section{Abstract convergence theorem}\label{sect_abstract_cvg_thm}

\subsection{Construction of the skeleton of a graph}\label{sect_constr_tnk}

\subsubsection{Decomposing the graph along cut-points}

Let $G$ be a rooted finite graph that is connected. 

\begin{definition}\label{def_cut_point}
We call cut-bond any edge $e\in E(G)$ whose removal disconnects G. By definition only one of the endpoints of a cut-bond is connected to the root and any such point is called a cut-point. \end{definition}
We denote $V_{\text{cut}}(G)$ the set of cut-points of $G$, which we assume  to be non-empty.

Let us now consider a sequence  $(x_i)_{i\in \N}$ of points $V_{\text{cut}}(G)$. Fix $K\in \N$, we construct the graph $G(K)$ in the following manner
\begin{enumerate}
\item the vertices of $G(K)$ are the set of all cut-points that lie on a path the root to an $x_i$ for $i\leq K$,
\item  two vertices of $G(K)$ are adjacent if there exists a path connecting them which does not use any cut-point.
\end{enumerate}

The new graph $G(K)$ will be rooted at $\text{root}^*$ \hfff{rootstar} which is the first cut-point on the path from the root to $x_0$.

It is elementary to notice that this graph is composed of complete graphs glued together, indeed the removal of all cut-bonds in $G$ results in a graph several connected components. Those connected components with the cut-bonds that link them to the root are called bubbles. All cut-points corresponding to cut-bonds with at least one end-point in the same bubble are inter-connected.

\begin{figure}
  \includegraphics[width=\linewidth]{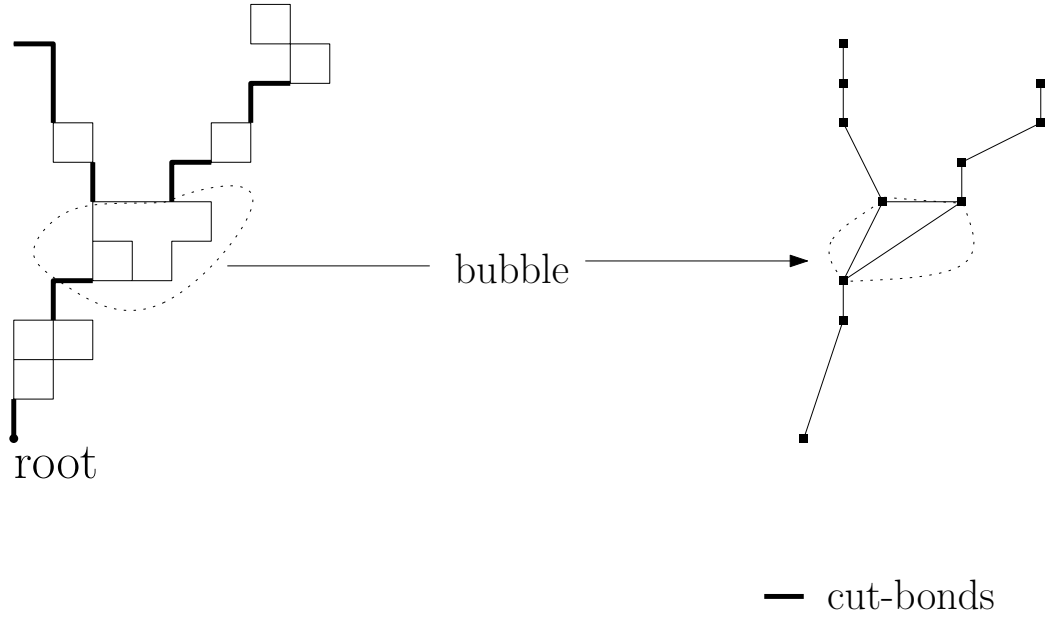}
  \caption{Construction of $G(K)$, obtained using  a sequence $(x_i)_{i\in \N}$ the covers the whole graph}
\end{figure}

\begin{definition}\label{def_thin}
We will say that a graph $G(K)$ is asymptotically tree-like if it does not contain any subgraph that is a complete graph apart from segments and triangles. 
\end{definition}

We would typically expect $G$ to be asymptotically tree-like if $G$ has many cut-bonds and the random variables $V_i$ are uniformly distributed. 

\subsubsection{Approximating a asymptotically tree-like graph by a graph spatial tree}

Let us assume that $G(K)$ is asymptotically tree-like. We are now going to perform a technical operation, that will be helpful to complete our proofs. In essence we are trying to build a graph spatial tree that will approximate $G(K)$ well.

We want to turn the triangles  present in $G(K)$ into stars in order to turn out asymptotically tree-like graph into a tree, this procedure will add one point for every triangle present in the graph $G(K)$.
 
\vspace{0.5cm}

{\it Step 1: Turning $G(K)$ into a tree $\T^{(G,K)}$}
 
\vspace{0.5cm}

 For every triangle $(x,y),(y,z),(z,x)\in E(G(K))$, we remove the edges $(x,y),(y,z),(z,x)$ and we introduce a new vertex $v_{x,y,z}$ and new edges $(x,v_{x,y,z})$, $(y,v_{x,y,z})$, $(z,v_{x,y,z})$.   We denote $\T^{(G,K)}$, the tree obtained by this construction. 
 
 We denote $V(\T^{(G,K)})$ the vertices of $\T^{(G,K)}$ and $V^*(\T^{(G,K)})$ the vertices which are not of the form $v_{x,y,z}$ (which are actually the vertices of $G(K)$). 
 
 Similarly, we denote $E(\T^{(G,K)})$ the edges of $\T^{(G,K)}$ and $E^*(\T^{(G,K)})$ \hfff{estar} the edges which are not of the form $(x,v_{x,y,z}),(y,v_{x,y,z}),(z,v_{x,y,z})$.

Finally, for $x,y\in V(\T^{(G,K)})$, we write $x\sim^*y$ if there exists no $z\in  V^*(\T^{(G,K)})$ which lies on the path from $x$ to $y$. This means that $x$ and $y$ were neighbours before the star-triangle transformation, or equivalently that they are connected by a bubble (see for example $x$, $y$ and $z$ in Figure 3).

\begin{figure}
  \includegraphics[width=\linewidth]{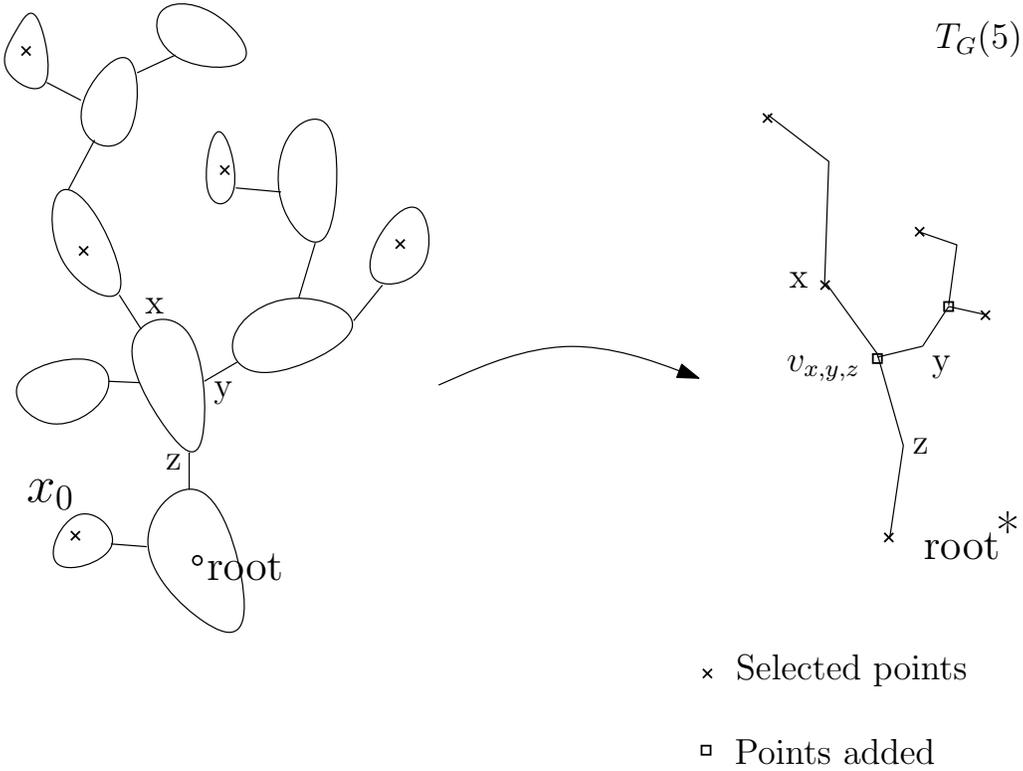}
  \caption{Construction of $T^{(G,5)}$ from a graph $G$. Even though the bubble of $x$, $y$ and $z$ has four neighbouring bubbles, the graph is still asymptotically-tree like because the bubble on the left does not lead to a selected point of the form $x_i$}
\end{figure}

 Since the tree $\T^{(G,K)}$ is rooted (at $\text{root}^*$) it comes with a natural notion of ancestry. For $x\in \T^{(G,K)}$, we denote $\overrightarrow{\T^{(G,K)}_{x}}$, the set of points of $\T^{(G,K)}$ which are descendants of $x$, including $x$.

\vspace{0.5cm}

{\it Step 2: Turning $\T^{(G,K)}$ into a real tree by adding a metric}

\vspace{0.5cm}

The tree $\T(G)$ comes with a natural metric by setting
\begin{enumerate}
\item for $(x,y)\in E^*(\T^{(G,K)})$, we set $d_{\T^{(G,K)}}(x,y)=d_{G}(x,y)$,
\item for any triple of edges $(x,v_{x,y,z}),(y,v_{x,y,z}),(z,v_{x,y,z})$, where $x$ is the ancestor of $y$ and $z$, we set \[d_{\T^{(G,K)}}(x,v_{x,y,z})=\frac{d_G(x,y)+d_G(x,z)-d_G(z,y)}{2},\] \[d_{\T^{(G,K)}}(y,v_{x,y,z})=\frac{d_G(x,y)+d_G(y,z)-d_G(x,z)}{2}\] and \[d_{\T^{(G,K)}}(z,v_{x,y,z})=\frac{d_G(x,z)+d_G(y,z)-d_G(x,y)}{2}.\]
Note that this assignment of distances keeps consistency in the sense that \[d_G(x,y)=d_{\T^{(G,K)}}(x,v_{x,y,z})+d_{\T^{(G,K)}}(y,v_{x,y,z}),\] \[d_G(x,z)=d_{\T^{(G,K)}}(x,v_{x,y,z})+d_{\T^{(G,K)}}(z,v_{x,y,z})\] and \[d_G(y,z)=d_{\T^{(G,K)}}(y,v_{x,y,z})+d_{\T^{(G,K)}}(z,v_{x,y,z}).\]
\item the distance grows linearly along an edge. 
\end{enumerate}

Our choice for the distances in the second part is arbitrary but it will not have an significant impact on our proof. It can be noted that this distance conserves the distance from $\text{root}^*$ to any point in $V^*(\T^{(G,K)})$.

\vspace{0.5cm}

{\it Step 3: Assigning a spatial location to the points in $\T^{(G,K)}$}

\vspace{0.5cm}

Finally we want to view our tree as a spatial tree embedded in $\R^d$, i.e.~we want to find an embedding of the edges into $\R^d$.

 Any vertex of $V^*(\T^{(G,K)})$ is assigned its original location in $G$. Moreover the vertices $v_{x,y,z}$ are mapped to the barycenter of  $x$, $y$ and $z$. We write $\phi^{G(K)}$ this map.
 
  If $(x,y)\in E(\T^{(G,K)})$, then the point $z$ which is at a $d_{\T^{(G,K)}}$-distance $\alpha d_{\T^{(G,K)}}(x,y)$ along the edge $(x,y)$ is mapped to the point which is at distance $\alpha d_{\Z^d}(\phi^{G(K)}(x),\phi^{G(K)}(y))$ along the $\R^d$-geodesic between $\phi^{G(K)}(x)$ and $\phi^{G(K)}(y)$. This extends  $\phi^{G(K)}$ to a map from $\T^{(G,K)}$ to $\R^d$. 
  
  In particular the notation $\phi^{G(K)}(e)$, for $e\in E(\T^{(G,K)})$, corresponds to a segment of $\R^d$.

\vspace{0.5cm}

\subsubsection{A  natural resistance metric on the skeleton}

Let us now endow $\T^{(G,K)}$ with a resistance metric. We refer the reader to~\cite{lyons2005probability} for a background on resistances, time reversibility and electrical network theory which are central notions for the remainder of the paper. 

First, for all $(x,y)\in E^*(\T^{(G,K)})$, we set $\reff^{\T^{(G,K)}}(x,y)$ as the effective resistance between $x$ and $y$ in the graph $G$ (where edges in $G$ have resistance $1$).

Let us consider a triangle $(x,y),(y,z),(z,x)\in E(G(K))$. We denote $P_x^G$ the law of a simple random walk on the graph $G$ started at $x$. For any set $A\in V(G)$, let 
\begin{equation}\label{eq:defofTT+}
T_A=\inf\{l\geq0: X_l\in A\}\quad \text{ and } \quad T^+_A=\inf\{l>0: X_l\in A \}.
\end{equation}
If $A=\{x\}$ we write $T_x,T^+_x$ instead of $T_{\{x\}}$, $T^+_{\{x\}}$.
 We set
\[\reff^{G(K)}(x,y)^{-1}:=\pi(x)P_{x}^G[T(y)<T(z)\wedge T^+(x)]\]
where $\pi$ is number of neighbors of $x$ in $G$, which is the invariant measure associated to unit resistances on $G$. Note that this procedure defines the resistances of all three edges corresponding to a triangle.
\[\reff^{G(K)}(x,z)^{-1}:=\pi(x)P_{x}[T_z<T_y\wedge T^+_x]\]
and
\[\reff^{G(K)}(y,z)^{-1}:=\pi(y)P_{y}[T_z<T_x\wedge T^+_y].\]

We could have defined the corresponding quantities $\reff^{G(K)}(y,x)$, $\reff^{G(K)}(z,x)$ and $\reff^{G(K)}(z,x)$ in an analogous way and, by time reversibility we would have obtained that $\reff^{G(K)}(x,y)=\reff^{G(K)}(y,x),\reff^{\omega_n}(x,z)=\reff^{G(K)}(z,x)$ and $\reff^{G(K)}(y,z)=\reff^{G(K)}(z,y)$.

Next, it only remains to define the values of $\reff$ for the edges of $\T^{(G,K)}$ containing the artificial vertex $v_{x,y,z}$ using the star-triangle transformation. That is
\[\reff^{\T^{(G,K)}}(x,v_{x,y,z})=\frac{\reff^{G(K)}(x,y)\reff^{G(K)}(x,z)}{\reff^{G(K)}(x,y)+\reff^{G(K)}(y,z)+\reff^{G(K)}(z,x)},\]
\[\reff^{\T^{(G,K)}}(y,v_{x,y,z})=\frac{\reff^{G(K)}(x,y)\reff^{G(K)}(y,z)}{\reff^{G(K)}(x,y)+\reff^{G(K)}(y,z)+\reff^{G(K)}(z,x)}\]
and
\[\reff^{\T^{(G,K)}}(z,v_{x,y,z})=\frac{\reff^{G(K)}(x,z)\reff^{G(K)}(y,z)}{\reff^{G(K)}(x,y)+\reff^{G(K)}(y,z)+\reff^{G(K)}(z,x)}.\]

Taking $x,y \in V(\T^{(G,K)})$, we know that there is a unique simple path $x_0,\ldots, x_l$ in $V(\T^{(G,K)})$ from $x$ to $y$ and this allows us to  set $\reff^{\T^{(G,K)}}(x,y)=\sum_{i=0}^{l-1} \reff^{\T^{(G,K)}}(x_i,x_{i+1})$. This is the natural definition in view of the law of resistance in series and the fact that points in $V^*(\T^{(G,K)})$ are cut-points.

The resistance defines a metric on $V(\T^{(G,K)})$ that we denote $d^{\text{res}}_{\T^{(G,K)}}$. Those definitions were chosen so as to have the following property.

\begin{lemma}\label{res_equiv_tnk}
For all $x,y \in V^*(\T^{(G,K)})$, we have
\[
\reff^{G}(x, y)=\reff^{\T^{(G,K)}}(x, y).
\]
\end{lemma}

\begin{proof}
Since all vertices of $V^*(\T^{(G,K)})$ are cut-points, our additive definition of $\reff^{\T^{(n,K)}}$ (corresponding to the law of resistances in series) ensures that the previous relation will be verified if we simply check the previous relation for  any $x,y \in V^*(\T^{(G,K)})$ which are part of a triangle. We know that the star-triangle transformation conserves resistances (see e.g.~\cite{lyons2005probability}) so all we need to check is that the resistances we defined for triangles respect the transition probabilities for the walk. This can be checked since for any $(x,y,z)$ which is a triangle in $G^{(K)}$ we have
\[
P_x^{G}[T_y<T_z\wedge T_x^+]=\pi(x)\reff^{G(K)}(x,y)^{-1},
\]
by the very definition of $\reff^{G(K)}$.

\end{proof}


\subsubsection{Adding a measure associated to the volume of the graph}\label{sect_mu}

We are going to add a measure to our graph $\T^{(G,K)}$.

For any $x\in G$, let $\pi^{(G,K)}(x)$ be the unique $v\in V^*(\T^{(G,K)})$ separating $x$ from the origin and such that for any $v'\in V^*(\T^{(G,K)})$ with $v'$ separating $x$ from the origin and $v'\neq v$ we have that $v' \prec v$. That is, when going from $\text{root}^*$ to $x$, the point $\pi^{(G,K)}(x)$ is the last cut-point crossed before reaching $x$. In the case where $x$ is not separated from the origin by a cut-bond, i.e.~$x$ is in the bubble of the origin, then we set $\pi^{(G,K)}(x)=\text{root}^*$ by convention.

Now for $x\in V^*(\T^{(G,K)})$, let $v_{\T^{(G,K)}}(x):=\#\{(y,z)\in E(G): \pi^{(G,K)}(y)=x \text{ and } y\neq x\}$ and use this to define a measure on $V^*(\T^{(G,K)})$.
\[
\mu^{(G,K)}:=\sum_{x\in V^*(\T^{(G,K)})} v_{\T^{(G,K)}}(x)\delta_x.
\]

\subsubsection{Another way of viewing $\T^{(G,K)}$ as graph spatial tree}\label{sect_mathfrak}

For our future purpose it will be convenient to be able to introduce a reduced version of $\T^{(G,K)}$ where we view it as a graph spatial tree with a number of vertices between $K+1$ and $2K+1$ (whereas $\T^{(G,K)}$ typically has a high number of points if $G$ is large). This distinction will be important for the Definition~\ref{def_condG}.

It will be a graph spatial tree  denoted $(\mathfrak{T}^{(G,K)},d_{\mathfrak{T}^{(G,K)}},\phi_{\mathfrak{T}^{(G,K)}})$ which is obtained by a  procedure similar to Section~\ref{sect_graph_tree}. In the notations of that section this spatial graph is $((\T^{(G,K)})^{K,(x_i)},d_{(\T^{(G,K)})^{K,(x_i)}}, \phi_{(\T^{(G,K)})^{K,(x_i)}})$.

In words this simply means we restrict the tree structure $\T^{(G,K)}$ to the subgraph obtained from the points $\text{root}^*$, $x_0,\ldots,x_K$ and the branching points that these points created. Hence, the vertices of $\mathfrak{T}^{(G,K)}$ are 
\[
V(\mathfrak{T}^{(G,K)}):=\{\text{root}^*,x_0,\ldots,x_K\} \cup (V(\T^{(G,K)})\setminus V^*(\T^{(G,K)})),
\]
 (the set on the right-hand side of the union being the branching points) and the edges $E(\mathfrak{T}^{(G,K)})$ are the ones naturally inherited from the tree structure of $\T^{(G,K)}$.

\begin{remark}\label{rem_not_depend} It is important to note that the distance, the resistance distance and the embedding we assign to $\mathfrak{T}^{(G,K)}$ coincide with those assigned to $\T^{(G,K)}$. This will allow us to use, e.g., $d_{\T^{(G,K)}}$ to signify $d_{\mathfrak{T}^{(G,K)}}$.\end{remark}

\subsection{Setting for the abstract theorem} \label{sect_model_abstract}


In this section, we will consider a sequence of random graphs $(G_n)_{n\in \N}$ chosen under a measure ${\bf P}_n$, which in practice will be  large critical structures. The $n$ will quantify the order of the volume of $G_n$, which in turn means (in the universality class we are interested in) that the intrinsic distances between points in $G_n$ is of the order of $n^{1/2}$ and the extrinsic distances are of the order $n^{1/4}$. Our eventual goal is to study $(X^{G_n}_m)_{m\in \N}$ which is the simple random walk on $G_n$.

Our construction will rely on a sequence of i.i.d.~random variables $(V_i^n)_{i\in \N}$ supported on cut-points. In practice this sequence should be asymptotically close to uniform random variables. 

\begin{definition}\label{def_augmented} \hfff{gv} For $n\in \N$, we say that $(G_n,(V_i^n)_{i\in \N})_{n\in \N}$ is a sequence of random augmented graphs under the measure ${\bf P}_n$. \end{definition}

Fix $K\in \N$. If the graph $G_n(K)$ constructed from $(G_n,(V_i^n)_{i\in \N})_{n\in N}$ is asymptotically tree-like, then the construction of the skeleton of the previous section can be carried out. In order to lighten the notations, we will write \hfff{mass} $G^{(n,K)}$, $\T^{(n,K)}$, $\mathfrak{T}^{(n,K)}$, $V^*(\T^{(n,K)})$, $\phi^{(n,K)}$, $v^{(n,K)}$, $\reff^{(n,K)}$ and $\pi^{(n,K)}$ for $G_n(K)$, $\T^{(G_n,K)}$, $\mathfrak{T}^{(G_n,K)}$, $V^*(\T^{(G_n,K)})$, $\phi^{G_n(K)}$,  $v_{\T_{G_n(K)}}$, $\reff^{\T^{(G_n,K)}}$ and $\pi^{(G_n,K)}$ and we also introduce the rescaled quantities $d^{(n,K)}(\cdot,\cdot)$, $d^{(n,K)}_{\text{res}}(\cdot,\cdot)$ and $\mu^{(n,K)}$ for $n^{-1/2}d_{\T^{(G_n,K)}}(\cdot,\cdot)$, $n^{-1/2}d_{\T^{(G_n,K)}}^{\text{res}}(\cdot,\cdot)$ and $n^{-1}\mu^{(G_n,K)}$. All those quantities were defined in Section~\ref{sect_constr_tnk}.

We recall that $d^{(n,K)}(\cdot,\cdot)$ (resp.~$\phi^{(n,K)}$) is a distance on (resp.~embedding of) $\mathfrak{T}^{(n,K)}$ because of Remark~\ref{rem_not_depend}.


\subsubsection{Condition $(S)$}\label{sect_asympT_thin}

For any $x\in V^*(T_n(K))$, we call $K$-sausage of $x$ the set $\{y\in G_n,\ \pi^{(n,K)}(y)=x\}$.

  \begin{figure}
  \includegraphics[width=\linewidth]{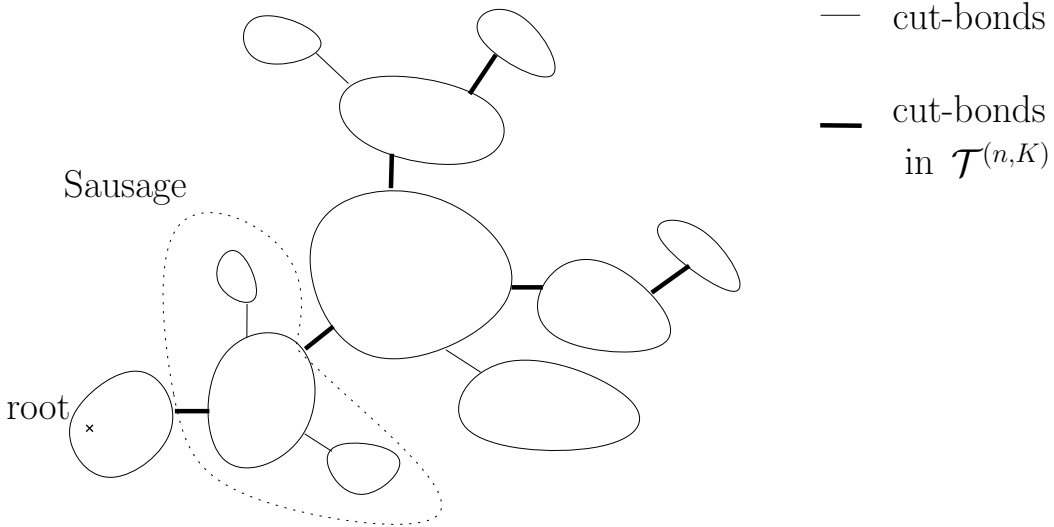}
  \caption{The sausage corresponding to a specific cut-bond.}
\end{figure}

 Note that a sausage is typically much large than the corresponding bubble because it also contains bubbles of $G_n$ which are not in $\T^{(n,K)}$ see Figure 4.
 We introduce\hfff{delta1}
\begin{equation}\label{eq:defofdelta}	
\Delta^{(n,K)}_{\Z^d}:=\max_{x\in V^*(\T^{(G_n,K)})} \text{Diam}_{\Z^d}(\{y\in V(G_n),\ \pi^{(n,K)}(y)=x\}),
\end{equation}
where $\text{Diam}_{\Z^d}(A):=\max\{d_{\Z^d}(x,y):x,y\in A\}$, for any $A\subset \Z^d$.
We also introduce\hfff{delta2}
\begin{equation}\label{eq:defofdeltaintr} 
\Delta^{(n,K)}_{G_n}:=\max_{x\in V^*(\T^{(G_n,K)})} \text{Diam}_{G_n}(\{y\in V(G_n),\ \pi^{(n,K)}(y)=x\}),
\end{equation}
where $\text{Diam}_{G_n}(A):=\min\{d_{G_n}(x,y):x,y\in A\}$ for any $A\subset \Z^d$, and $d_{G_n}$ is the graph distance in $G_n$. 
In our context, we want to extend the definition of asymptotically tree-like graphs (see Definition~\ref{def_thin}).
\begin{definition}\label{def_athin}
We say that a sequence of random augmented graphs $(G_n,(V_i^n)_{i\in \N})_{n\in N}$ verifies condition $(S)$ if  
\begin{enumerate}
\item for all $K\in \N$, we have
\[
\lim_{n\to \infty} {\bf P}_n[G^{(n,K)}\text{ is asymptotically tree-like}]=1.
\]
\item for all $\epsilon>0$, we have
\[ 
\lim_{K\to \infty} \sup_{n\in \N} {\bf P}_n\Bigl[n^{-1/4}\Delta^{(n,K)}_{\Z^d} >\epsilon\Bigr]=0
\]
and
\[ 
\lim_{K\to \infty} \sup_{n\in \N} {\bf P}_n\Bigl[n^{-1/2}\Delta^{(n,K)}_{G_n} >\epsilon\Bigr]=0.
\]

\end{enumerate}
\end{definition}

The first part of the definition states that it is unlikely for four larges branches to emanate from the same bubble. It should be noted that in the ISE this property is verified (see Section~\ref{sect_KISE}) The second part of the condition is related to the fact that there are no parts of $G_n$ that have macroscopic ($\Z^d$ and intrinsic) length but where there are no cut-points.

\begin{remark}\label{rem_abuse_tnk} If a sequence $(G_n,(V_i^n)_{i\in \N})_{n\in N}$ is asymptotically tree-like, then the notation $\T^{(n,K)}$ and $\mathfrak{T}^{(n,K)}$ make sense with probability going to $1$ since these objects can be constructed with the methods of Section~\ref{sect_constr_tnk}. The conditions which will involve $\T^{(n,K)}$ and $\mathfrak{T}^{(n,K)}$ (condition $(G)$ of definition~\ref{def_condG} and condition $(V)$ of definition~\ref{def_condV}) are all asymptotical in $n$. Hence, they are not affected by the fact that $\T^{(n,K)}$ is not defined on an event of small probability. We will thus allow ourselves a slight abuse of notation in the statement of these conditions. \end{remark}

\subsubsection{Key properties for the skeleton}\label{sect_prop_skel}

We will now explain why the skeleton of a graph is useful.

It is elementary to see that for all $K\in \N$, 
\begin{equation}\label{skeleton_good_approx}
\max_{\substack{x,y\in V^*(\T^{(n,K)}),\\  x\prec y, x\sim^*y}} \max_{\substack{z\in V(G_n(K)),\ \pi^{(n,K)}(z)=x \\ z'\in \phi^{(n,K)}([x,y])}} d_{\Z^d}(z,z') \leq 2\Delta^{(n,K)}_{\Z^d},
\end{equation}
which, for asymptotically tree-like graphs, means that the left side is non-macroscopic. Intuitively this means that every $K$-sausage is close to the corresponding edge in the $K$-skeleton of $G_n$. This proves the first interesting property of the skeleton: it is a good spatial approximation of asymptotically tree-like graphs.

Let us present a second interesting property of the skeleton. For this, we introduce $\lambda^{(n,K)}_{\text{res}}$  \hfff{lambdares} the Lebesgue measure on $(\T^{(n,K)},d^{(n,K)}_{\text{res}})$, normalized to have total mass $1$. This measure is well defined because $\T^{(n,K)}$ is  a graph spatial tree which can always be equipped with a Lebesgue probability measure associated with the resistance distance (see Section~
\ref{sect_graph_tree}). This naturally induces a Lebesgue probability measure on $\mathfrak{T}^{(n,K)}$ (see Section~\ref{sect_mathfrak} and Remark~\ref{rem_not_depend}).

Now, we can define $B^{(n,K)}$ be the Brownian motion on $(\mathfrak{T}^{(n,K)},d^{(n,K)}_{\text{res}},\lambda^{(n,K)}_{\text{res}})$ (which can be defined using Proposition~\ref{prop_def_process}). \hfff{bnk}

\begin{remark}\label{belowbnk}
Let us point out that the line segments of $\phi^{(n,K)}(\T^{(n,K)})$ may intersect once embedded into $\R^d$, however the Brownian motion $B^{(n,K)}$ is not impacted by those cycles. Indeed, because of the way the resistance metric is defined, this Brownian motion actually corresponds to a Brownian motion on a tree which has then been embedded.
\end{remark}

If $G^{(n,K)}$ is asymptotically tree-like, let $h^{(n,K)}_0:=0$ and define the successive times when $B^{(n,K)}$ moves to a different vertex of $V^*(\T^{(n,K)})$ \hfff{h}
\[
h^{(n,K)}_m:=\inf\left\{t\geq h^{(n,K)}_{m-1}:B^{(n,K)}_t\in V^*(\T^{(n,K)})\setminus \left\{ B^{(n,K)}_{h^{(n,K)}_{m-1}}\right\}\right\},
\]
for all $m\in\N$.
Furthermore set $A^{(n,K)}(0)=0$ and \hfff{ank}
\[
A^{(n,K)}(m):=\min \left\{l>A^{(n,K)}(m-1): X^{G_n}_l \in V^*(\T^{(n,K)})\setminus \left\{ X^{G_n}_{A^{(n,K)}(m-1)}\right\} \right\},
\] 
for all $m\in\N$,
where we recall that $(X^{G_n}_m)_{m\in \N}$ is the simple random walk on $G_n$.

The next lemma states that $X^{G_n}$ is a time change of a version of a $B^{(n,K)}$. Hence, it will be sufficient to understand the convergence of $B^{(n,K)}$ as $n$ goes to infinity and the time change linking $X^{G_n}$ and $B^{(n,K)}$ to prove a scaling limit result on $X_n$.
\begin{lemma}\label{lem:JtildeJ}
If $G^{(n,K)}$ is asymptotically tree-like, the processes $(B^{(n,K)}_{h^{(n,K)}_m})_{m\geq 0}$ and $(X^{G_n}_{A^{(n,K)}(m)})_{m\geq 0}$ on $V^*(\T^{(n,K)})$ have the same transition probabilities.
\end{lemma}

\begin{proof}
Considering that the transition probabilities, on vertices of $V^*(\T^{(n,K)})$, of the random walk $X^{G_n}$ and the Brownian motion $B^{(n,K)}$ are defined in the same manner in terms of the resistances on $G_n$ and $\T^{(n,K)}$ respectively. The result then follows from Lemma~\ref{res_equiv_tnk}.
\end{proof}

\subsection{Conditions for convergence towards $B^{ISE}$}\label{GRRR}

We are now going to introduce three other conditions that are the central hypotheses for our abstract theorem (Theorem~\ref{thm_abstract})

\subsubsection{Condition (G): asymptotic shape of the graph}\label{sect_condG}

Let us define a distance $D$ on graph spatial trees (defined in Section~\ref{sect_graph_tree}). Here, we follow Section 7 of~\cite{Croydon_arc}.

For $(T,d,\phi)$ a graph spatial tree, write $T^*$ for the rooted shape of the tree (the ordered graph tree without edge lengths) and $\abs{e_1},\ldots,\abs{e_l}$ for the lengths of the edges (according to the lexicographical order of $T^*$).

Take two graph spatial trees $(T,d,\psi)$ and $(T',d',\psi ')$. If $T^*\neq T'^*$, then  we set $d_1(T,T')=\infty$ and otherwise we set
\begin{equation}\label{eq:defofd1}
d_1(T,T'):=\sup_{i} \abs{\abs{e_i}-\abs{e_i'}}.
\end{equation}

Now if $T^*=T'^*$, we have a homeomorphism $\Upsilon_{T,T'}:\psi(T)\to \psi(T')$ such that  if $x\in \psi(T)$ is at a distance $\alpha \abs{e}$ along the edge $e$ is mapped to the point $x'\in \psi'(T')$ which is at distance $\alpha \abs{e'}$ along the corresponding edge $d'$. We then set 
\[
d_2(T,T'):=\sup_{x\in \psi(T)} d_{\R^d}(\psi(x), \psi'(\Upsilon_{T,T'}(x))).
\]
This yields a metric 
\begin{equation}\label{eq:defofD}
D((T,d,\psi),(T',d',\psi ')):=(d_1(T,T')+d_2(T,T'))\wedge 1
\end{equation}
 on graph spatial trees. The importance of this metric stems from the following result, see Lemma 7.1.~in~\cite{Croydon_arc} (recall that $\lambda_\T$ is the normalized Lebesgue measure on the tree).

\begin{proposition}\label{prop_metrictoBrownian}
Suppose $((T_n^{'},d_n^{'},\psi_n^{'}))_{n\in \N}$ is a sequence of graph spatial trees that converge with respect to the metric $D$ to a graph spatial tree $(T^{'},d^{'},\psi^{'})$. For each $n$, let $B^n$ be the Brownian motion on the spatial tree $(T_n^{'},d_n^{'},\lambda_{\T_n^{'}})$ and let $B$ be the Brownian motion on the spatial tree $(T^{'},d^{'},\lambda_{\T^{'}})$ (defined with Proposition~\ref{prop_def_process}). Then $(\psi_n(B^n))_{n\in \N}$ converges to $(\psi(B))$ in distribution for the topology of uniform convergence on compact sets.
\end{proposition}

This distance allows us to define our first condition (relevant definitions can be found at Definition~\ref{def_augmented}, Definition~\ref{def_athin}, Remark~\ref{rem_abuse_tnk} and Section~\ref{sect_KISE}).

\begin{definition}\label{def_condG} Condition (G): We say that a sequence of asymptotically tree-like random augmented graphs $(G_n,(V_i^n)_{i\in \N})_{n\in \N}$ satisfies condition $(G)_{\sigma_d,\sigma_{\phi}}$ if there exists $\sigma_d, \sigma_{\phi}>0$ such that for all $K\in \N$, the sequence of graph spatial trees $((\mathfrak{T}^{(n,K)},d^{(n,K)}(\cdot,\cdot),n^{-1/4}\phi^{(n,K)}))_{n\in \N}$ converges weakly to a $K$-ISE $(\mathfrak{T}^{(K)},\sigma_{d} d_{\mathfrak{T}^{(K)}},\sigma_{\phi} \sqrt{\sigma_{d} } \phi_{\mathfrak{T}^{(K)}})$ in the topology induced by $D$.
\end{definition}

This condition states that for all $K\in \N$, the $K$-skeleton of our sequence of random graphs resembles the $K$-skeleton of the ISE.
The constant $\sigma_d$ is the ratio between the graph distance in $G_n$ and the canonical distance in the CRT. The constant $\sigma_{\phi}$ is the diffusivity of the embedding per unit of graph-distance. 

\begin{remark} It is very important to stress that in condition (G), the topology induced by $D$ imposes a condition on the convergence of the distances of only a finite number of points (between $1+K$ and $1+2K$). This is where the distinction between $\mathfrak{T}^{(n,K)}$ and $\T^{(n,K)}$ makes a big difference. \end{remark}

\begin{remark} The requirement that $T$ and $T'$ have the same shape seems to be very strong, but we believe that in practice this condition will always be verified. We recall that we only require that the skeleton trees, as in Figure 3, (which a finite number of points between $1+K$ and $1+2K$) have the same shape when $n$ is large. \end{remark}

\subsubsection{Condition (V): Uniform distribution of the volume}

 Recalling that $\lambda^{(n,K)}_{\text{res}}$ is the natural Lebesgue probability measure on the graph spatial tree $\T^{(n,K)}$ (see beginning of Section~\ref{sect_prop_skel}), that $\overrightarrow{\T^{(n,K)}_{x}}$ are the descendants of $x$ (including $x$ itself) in $\T^{(n,K)}$ \hfff{arrowt} and the definition of $\mu^{(n,K)}$ in Section~\ref{sect_mu}. Our second condition states that

\begin{definition}\label{def_condV} Condition (V): We say that a sequence of asymptotically tree like random augmented graphs  $(G_n,(V_i^n)_{i\in \N})_{n\in N}$ satisfies condition $(V)_{\nu}$ if  there exists $\nu>0$ such that for $\epsilon>0$
\[
\limsup_{K\to \infty}\limsup_{n\to \infty} {\bf P}_n \Bigl[\sup_{x\in \T^{(n,K)}} \abs{\nu \lambda^{(n,K)}\Bigl(\overrightarrow{\T^{(n,K)}_{x}}\Bigr)-\mu^{(n,K)}\Bigl(\overrightarrow{\T^{(n,K)}_{x}}\Bigr)}>\epsilon\Bigr] =0.
\]
\end{definition}

Intuitively this condition says that the volume of the graph is asymptotically uniformly distributed over the graph and that the total volume is of order $\nu n$. 

\subsubsection{Condition (R): the linearity of the resistance}

\begin{definition} Condition (R): We say that a sequence of random augmented graphs $(G_n,(V_i^n)_{i\in \N})_{n\in N}$ verifying condition $(S)$  satisfies condition $(R)_{\rho}$ if there exists $\rho>0$ such that for all $\epsilon>0$ and for all $i\in \N$
\[
\lim_{n\to \infty} {\bf P}_n\Bigl[ \abs{\frac{R^{G_n}(0,V_i^n)}{d^{G_n}(0,V_i^n)}-\rho}>\epsilon\Bigr]=0.
\]
\end{definition}

Intuitively this condition says that the resistance distance is asymptotically almost proportional to the graph distance.

\subsection{Sketch of proof}\label{sect_abstract_thm}

Let us discuss informally how the proof of the abstract theorem works.

Firstly, we can notice that, since the graphs verify condition $(S)$, the graph $G_n$ is uniformly close to its $K$-skeleton $\T^{(n,K)}$ for $K$ large by~\eqref{skeleton_good_approx}. Hence, in order to understand the simple random walk on $(X^{G_n}_m)_{m\geq 0}$ it is enough to look at the random walk at times when it visits $V^*(\T^{(n,K)})$, which we write $(X^{G_n}_{A^{(n,K)}(m)})_{m\geq 0}$.

Then, we can use Lemma~\ref{lem:JtildeJ} to see that $(X^{G_n}_{A^{(n,K)}(m)})_{m\geq 0}$ has the same law as $(B^{(n,K)}_{h^{(n,K)}_m})_{m\geq 0}$. Using this, we can couple $X^{G_n}$ and $B^{(n,K)}$ in such a way that $B^{(n,K)}_t$ is close to $n^{-1/2}{X}^{G_n}_{A^{(n,K)}(m(t))}$ for a certain function $m(t)$.

Now, $B^{(n,K)}$ is the Brownian motion on $(\T^{(n,K)},d^{(n,K)}_{\text{res}},\lambda^{(n,K)}_{\text{res}})$. By condition $(R)$, we have that $d^{(n,K)}_{\text{res}} \approx \rho d^{(n,K)}$. Hence, by Proposition~\ref{prop_metrictoBrownian} and condition $(G)$, it can be shown that $\phi_{\T_n}(B^{(n,K)}_t)$ becomes close, for $n$ large, to $\sigma_{\phi} \sqrt{\sigma_d}B^{K-ISE}_{(\rho\sigma_d)^{-1}t}$. That, together with Proposition~\ref{prop_approx_bisek}, we get that, for $K$ and $n$ large, $\phi_{\T_n}(B^{(n,K)}_t)$ is close to $\sigma_{\phi} \sqrt{\sigma_d}B^{ISE}_{(\rho\sigma_d)^{-1}t}$

The only remaining task is to estimate the time-change $A^{(n,K)}(m(t))$ between $X^{G_n}$ and $B^{(n,K)}$. If we prove that $n^{-3/2}A^{(n,K)}(m(t))\approx  \nu t$, then a simple computation shows that the theorem follows. This time-change accounts for the time spent by the random walk  crossing edges of $E^*(\T^{(n,K)})$ until the coupled Brownian motion $B^{(n,K)}$ has moved for a time $t$.

For each edge $e\in E^*(\T^{(n,K)})$, by time $t$ the random walk $X^{G_n}$ will cross $l^{(n,K)}_t(e)$ times the bubble corresponding to the edge $e$ where $l^{(n,K)}_{t}(e)$ is the edge local time of  $B_t^{(n,K)}$. Since $l^{(n,K)}_t(e)$ is very large, we would like to say that the time spent crossing $e$ should be close to its expected value. Because of the tree structure of $\T^{(n,K)}$, a crossing in one direction is following by a crossing in the other, and this allows us to apply the commute time formula (see~\cite{chandra1996electrical}) which gives us an exact value for the expected time of a back and forth crossing. This means that
\[
A^{(n,K)}(m(t))\approx \sum_{e\in E(\T^{(n,K)})}l^{(n,K)}_{t}(e)\reff^n(e)\mu^{(G_n,K)}(e),
\]
where one recognizes the quantities appearing in the commute time formule (the factor $2$ is cancelled by the fact that the number of back and forth crossings of $e$ is only, roughly, half of the local time at $e$).

After showing that for $n$ large, 
\begin{enumerate}
\item $n^{-1/2}l^{(n,K)}_{t}(e)\reff^n(e)$ is close to the local time $\tilde{L}^{(K)}_t$ of $\sigma_{\phi} \sqrt{\sigma_d}B^{ISE}_{(\rho\sigma_d)^{-1}t}$, 
\item $ n^{-1} \mu^{(G_n,K)}$ is close to $\nu \lambda^{(K)}_{\frak{T}}$ by $(V)_{\nu}$,
\end{enumerate}
we can show that 
\[
n^{-3/2}A^{(n,K)}(m(t))\approx  \nu \int_{\frak{T}^{(K)}}  \tilde{L}_t^{(K)}(x) \lambda^{(K)}_{\frak{T}}(dx) =\nu t,
\]
where the last equality is a simple consequence of the definition of local time, see Definition~\ref{def:localtime}.

\section{Proof of the abstract theorem}

This first section is devoted to the proof of the main theorem. The proof will rely on resistance and local time estimates that are deferred to later sections in this paper (see Section~\ref{sect_resistance_estimate} and Section~\ref{s:AnKislinear} respectively).

\subsection{Convergence of the image of $B^{(n,K)}$ towards $B^{K-ISE}$}

By Assumption $(G)_{\sigma_d,\sigma_{\phi}}$ and the Skorohod representation Theorem, we can assume that, for all $K\in\N$, the augmented random graphs $(G_n, (V^n_i)_{i = 1,\cdots,K})_{n\in\N}$ and $\mathfrak{T}^{(K)}$ are defined in a common probability space $(\Omega^{(K)},\mathcal{F}^{(K)},{\mathbf{P}}^{(K)})$ and that  
\begin{equation}\label{eq:couplinggeometry}
(\mathfrak{T}^{(n,K)},d^{(n,K)},n^{-1/4}\phi_{\T_n})\stackrel{n\to\infty}{\to}(\mathfrak{T}^{(K)}, \sigma_d d_{\mathfrak{T}^{(K)}},\sigma_{\phi}\sqrt{\sigma_d}\phi_{\mathfrak{T}^{(K)}}) \qquad {\mathbf{P}}^{(K)}\text{-a.s.},
\end{equation}
in the topology induced by $D$, where $D$ is the natural topology on graph spatial trees, (see \eqref{eq:defofD}).

Recall that $B^{(n,K)}$ is the Brownian motion in $(\frak{T}^{(n,K)},d^{(n,K)}_{\text{res}},\lambda^{(n,K)}_{\text{res}})$ as in Proposition \ref{prop_def_process}. Our main objective in this subsection is to show that $n^{-1/4}\phi_{\T_n}(B^{(n,K)})$ scales to $B^{K-ISE}$ as $n\to\infty$.  More precisely
\begin{lemma} \label{lem:BnktoBk}
 Under \eqref{eq:couplinggeometry} and condition $(R)_{\rho}$, for all $K\in\N$
\begin{equation}
(n^{-1/4}\phi_{\T_n}(B^{(n,K)}_t))_{t\geq0}\stackrel{n\to\infty}{\to} ( \sigma_{\phi} \sqrt{\sigma_d}B^{K-ISE}_{(\rho\sigma_d)^{-1}t})_{t\geq0},
\end{equation}
in distribution in the space $C(\R_+,\R^d)$ endowed with the topology of uniform convergence over compact subsets, $\mathbf{P}^{(K)}$-almost surely.
\end{lemma}
The proof of the Lemma above relies on the following result.
 \begin{lemma}\label{lem:graphspatialresistancetreeconvergence} Under \eqref{eq:couplinggeometry} and condition $(R)_{\rho}$, for all $K\in\N$
   \[(\frak{T}^{(n,K)},d_{\text{res}}^{(n,K)},n^{-1/4}\phi^{(n,K)})\stackrel{D}{\to}(\frak{T}^{(K)},\sigma_d \rho d_{\frak{T}^{(K)}}, \sigma_{\phi} \sqrt{\sigma_d}\phi^{\frak{T}^{(K)}}),\]
  as $n\to\infty$, $\mathbf{P}^{(K)}$-almost surely.
 \end{lemma}
 \begin{remark}
 By Lemma \ref{lem:graphspatialresistancetreeconvergence}, we know that for $n$ large enough, we have that $\frak{T}^{(K)}$ and $\frak{T}^{(n,K)}$ are composed of a finite number of edges  $(e(i))_{i=1,\dots,s}$ and $(e^n(i))_{i,\dots,s}$ respectively, where $s$ is the number of edges of $\frak{T}^{(K)}$ and $\frak{T}^{(n,K)}$. Then by a simple union bound, we have that, for all $K\in\N$ and $\epsilon>0$,
 \begin{equation}\label{eq:srt2}
 \lim_{n\to\infty}\mathbf{P}^{(K)}\left[\max_{i=1,\dots,s}\abs{\frac{d^{(n,K)}_{\text{res}}(e^n(i))}{ d_{\frak{T}}(e(i))}-\sigma_d\rho}\geq\epsilon\right]=0.
 \end{equation}
 In other words, the random variable $\max_{i=1,\dots,s}\abs{\frac{d^{(n,K)}_{\text{res}}(e^n(i))}{ d_{\frak{T}}(e(i))}-\sigma_d\rho}$ converges to $0$ in $\mathbf P^{(K)}$-probability. Therefore, by the Skorohod representation theorem we can assume that 
 \begin{equation}
 \max_{i=1,\dots,s}\abs{\frac{d^{(n,K)}_{\text{res}}(e^n(i))}{ d_{\frak{T}}(e(i))}-\sigma_d\rho}\stackrel{n\to\infty}{\to}0, \quad \mathbf{P}^{(K)}\text{-almost surely}.
 \end{equation}
 \end{remark}
 \begin{proof}[Proof of Lemma \ref{lem:graphspatialresistancetreeconvergence}]
  First, note that by \eqref{eq:couplinggeometry}, $\frak{T}^{(n,K)}$ is homeomorphic to $\frak{T}^{(K)}$ for $n$ large enough. Moreover, the homeomorphism can be chosen to be lexicographical-order preserving and it induces a correspondence between the edges $(e_i^{n,K})_{i}$ of $\frak{T}^{(n,K)}$ and the edges $(e_i)_{i}$ of $\frak{T}^{(K)}$ as in Section \ref{sect_condG}. By \eqref{eq:couplinggeometry} and Lemma \ref{res_2a} we get 
\[
  \lim_{n\to\infty}\sup_i\abs{\abs{e_i^{n,K}}_{d^{(n,K)}_{\text{res}}}-\rho\sigma_d\abs{e_i}_{d_{\frak{T}^{(K)}}}}=0,\quad\mathbf{P}^{(K)}\text{-a.s},
\]
  where $\abs{\cdot}_{d^{(n,K)}_{\text{res}}},\abs{\cdot}_{d_{\frak{T}^{(K)}}}$ denote the length with according  to $d^{(n,K)}_{\text{res}},d_{\frak{T}^{(K)}}$ respectively.  
  That is, $(\frak{T}^{(n,K)},d^{(n,K)}_{\text{res}})$ converges under $d_1$ to $(\frak{T}^{(K)},\sigma_d\rho d_{\frak{T}^{(K)}})$, $\mathbf{P}^{(K)}$- almost surely, where $d_1$ is as in \eqref{eq:defofd1}. Let $\hat{\Upsilon}_{n,K}$ be the lexicographical-order preserving homeomorphism between $(\frak{T}^{(n,K)},d^{(n,K)}_{\text{res}})$ and $(\frak{T}^{(K)},\rho\sigma_dd_{\frak{T}^{(K)}})$ which is linear along the edges.
  It remains to show that
  \begin{equation}\label{eq:Gfordres}
  \lim_{n\to\infty }\sup_{x\in\frak{T}^{(n,K)}}d_{\R^d}(n^{-1/4}\phi_{\T_n}(x),\sigma_{\phi}\sqrt{\sigma_d}\phi(\hat{\Upsilon}_{n,K}(x)))=0,\quad\mathbf{P}^{(K)}\text{-a.s.}
  \end{equation} 
  Let  $\Upsilon_{n,K}$ the lexicographical-order preserving homeomorphism between $(\frak{T}^{(n,K)},d^{(n,K)})$ and $(\frak{T}^{(K)},\sigma_dd_{\frak{T}^{(K)}})$ which is linear along the edges. By  \eqref{eq:couplinggeometry} we have
  \begin{equation}\label{eq:Gford}
  \lim_{n\to\infty }\sup_{x\in\frak{T}^{(n,K)}}d_{\R^d}(n^{-1/4}\phi_{\T_n}(x),\sigma_{\phi}\sqrt{\sigma_d}\phi(\Upsilon_{n,K}(x)))=0,\quad\mathbf{P}^{(K)}\text{-a.s.}
  \end{equation} 
  It can be deduced from Lemma  \ref{res_2a} and \eqref{eq:couplinggeometry} that \begin{equation}
  \lim_{n\to\infty} \sup_{x\in\frak{T}^{(n,K)}} d_{\frak{T}^{(K)}}(\Upsilon_{n,K}(x),\hat{\Upsilon}_{n,K}(x))=0.
  \end{equation}
  That, plus \eqref{eq:Gford} and the continuity of $\phi_{\frak{T}^{(n,K)}}$ yields \eqref{eq:Gfordres}.
 
 \end{proof}
  \begin{proof}[Proof of Lemma \ref{lem:BnktoBk}]
  Recall from Section \ref{sect_KISE} that $B^{(K)}$ is the Brownian motion in $(\frak{T}^{(K)},d_{\frak{T}},\lambda^{(K)}_{\frak{T}})$, where $\lambda^{(K)}_{\frak{T}}$ is the Lebesgue measure according to $d_{\frak{T}}$ normalized to become a probability measure.
  Let $\bar{B}^{(K)}_t:=B^{(K)}_{(\rho\sigma_d)^{-1}t}$. We will start showing that 
  \begin{equation}\label{eq:barBequivalentdefinition}
  \bar{B}^{(K)}\text{ is the Brownian motion in }(\frak{T}^{(K)}, \rho \sigma_d d_{\frak{T}}, \lambda^{(K)}_{\frak{T}}).
  \end{equation} For this is enough to check the properties given above Proposition \ref{prop_def_process}.
  Among them the only one which is not trivial is property $(2)$, which we will check now.
 
  Let $E_x,\bar{E}_x$ denote expectation with respect to the law of $B^{(K)}, \bar{B}^{(K)}$ started at $x$. Let $T_x:=\inf\{s:B^{(K)}_s=x\}$ 
  and $\bar{T}_x:=\inf\{s:\bar{B}^{(K)}_s=x\}$.
  It follows that, for any Borelian $A$ of $\frak{T}^{(K)}$
  \begin{align*}
  &\bar{E}_x[\text{Leb}\{s\leq \bar{T}_y: \bar{B}^{(K)}_s\in A \}]
  =E_x [\rho \sigma_d\text{Leb} \{ s\leq T_y: B^{(K)}_s\in A\}]\\
  = &\rho \sigma_d\int_A 2 d_{\frak{T}}(b^{\frak{T}}(x,y,z),y) \lambda^{(K)}_{\frak{T}}(dz)=\int_A 2 \rho \sigma_dd_{\frak{T}}(b^{\frak{T}}(x,y,z),y) \lambda^{(K)}_{\frak{T}}(dz)
  \end{align*} 
  where the first equality follows from the definition of $\bar{B}^{(K)}$ and the second equality follows from the fact that $B^{(K)}$ is the Brownian motion in $(\frak{T}^{(K)},d_{\frak{T}},\lambda^{(K)}_{\frak{T}})$. This shows property $(2)$ in the definition of Brownian motion. Therefore \eqref{eq:barBequivalentdefinition} holds.

 Therefore, by \eqref{eq:barBequivalentdefinition}, we can deduce from  Proposition \ref{prop_metrictoBrownian} and Lemma \ref{lem:graphspatialresistancetreeconvergence} that
 \begin{equation}\mathbf
 (\phi_{\T_n}(B^{(n,K)}_t))_{t\geq0}\text{ converges to }(\sqrt{\sigma_d}\sigma_{\phi}\phi(\bar{B}^{(K)}_{t}))_{t\geq0}.
 \end{equation}
   Hence, the lemma follows from the definition of $\bar{B}^{(K)}$ and the fact that $B^{ISE-K}_t=\phi(B^{(K)}_t)$ (which holds by definition).
   \end{proof}

\subsection{Coupling between $X^{G_n}$ and $B^{(n,K)}$} \label{sect:couplingrwbm}

In this section we will introduce a coupling between $X^{G_n}$ and $B^{(n,K)}$ (the latter being defined above Remark~\ref{belowbnk}).

 The construction is performed in two steps. The first one is to use Lemma \ref{lem:JtildeJ} and $B^{(n,K)}$ to construct a process $J^{(n,K)}$ taking values in $V^*(\T^{(n,K)})$, which has the law of trace of $X^{G_n}$ on $V^*(\T^{(n,K)})$.  Given $B^{(n,K)}$, let 
\begin{equation}\label{eq:defofJnK}
J^{(n,K)}_m:=B^{(n,K)}_{h^{(n,K)}_m} \qquad \text{for all }m\in\N,
\end{equation}
where
$h^{(n,K)}_0:=0$ and 
\begin{equation}\label{eq:defofhnk}
h^{(n,K)}_{m+1}:=\min\left\{t>h^{(n,K)}_m:B^{(n,K)}_t\in V(\T^{(n,K)})-\left\{B^{(n,K)}_{h^{(n,K)}_m}\right\} \right\}.
\end{equation}
By Lemma \ref{lem:JtildeJ}, $(J^{(n,K)}_m)_{m\in\N}$ has the same transition probabilities as $(B^{(n,K)}_{h^{(n,K)}_m})_{m\in\N}$.
Before continuing with the construction of the coupling, let us state here the following result which explains the relationship between $J^{(n,K)}$ and $B^{(n,K)}$.
Let 
\begin{equation}\label{eq:defofmt}
m(t):=\min\{m\in\N: h^{(n,K)}_m\geq t\}.
\end{equation}
\begin{lemma}\label{lem:jnktobnk}
 If the random augmented graphs $(G_n,(V^i_n)_{i\in\N})_{n\in\N}$ verifies condition $(S)$, then
\[\lim_{K\to\infty}\limsup_{n\to\infty} \mathbf{P}^{(K)}\left[\sup_{t\in\R} d_{\R^d}(n^{-1/4}\phi_{\T_n}(J^{(n,K)}_{m(t)}),n^{-1/4}\phi_{\T_n}(B^{(n,K)}_t))\geq \epsilon \right] =0,\]
for all $\epsilon>0$,
\end{lemma}
\begin{proof}[Proof of Lemma \ref{lem:jnktobnk}]
It follows from the coupling between $J^{(n,K)}$ and $B^{(n,K)}$ that, for all $t\geq0$, $\phi_{\T_n}(J^{(n,K)}_{m(t)})$ and $\phi_{\T_n}(B^{(n,K)}_t)$ are either in the same bubble or in two different, but contiguous bubbles. Therefore 
\[d_{\R^d}(\phi_{\T_n}(J^{(n,K)}_{m(t)}),\phi_{\T_n}(B^{(n,K)}_t))\leq 2 \Delta^{(n,K)}_{\Z^d}.\]
Hence, the result follows immediately from the assumption that the sequence of graphs verifies condition $(S)$.
\end{proof}
 The second step of the coupling consists in completing the path of $\phi_{\T_n}(J^{(n,K)})$ between successive visits of vertices of $\T^{(n,K)}$ to obtain a version of $X^{G_n}$. Starting from $X^{G_n}_0=0$, construct $X^{G_n}$ up to the first hitting time of $V^\ast(\T^{(n,K)})$ by sampling a simple random walk on $G_n$ conditioned on $X^{G_n}_{A^{(n,K)}(0)}=\phi_{\T_n}(J_0^{n,K})=\phi_{\T_n}(\text{root}^\ast)$, where $A^{(n,K)}(0):=\min\{X^{G_n}_l\in \phi_{\T_n}(V^\ast(\T^{(n,K)}))\}$ and otherwise independent of $B^{(n,K)}$.
 It is clear that we can repeat this procedure in such a way that
\begin{equation}\label{eq:coupling}
X^{G_n}_{A^{(n,K)}(m)}=\phi_{\T_n}(J^{(n,K)}_m)\qquad \text{for all }m\text{ such that }J^{(n,K)}_m\in V^\ast(\T^{(n,K)}),
\end{equation}
where
\[A^{(n,K)}(m+1):=\min \left\{l>0: X^{G_n}_l \in \phi_{\T_n}(V^\ast(\T^{(n,K)}))-\left\{ Y^{n,K}_{A^{(n,K)}(m)}\right\} \right\}.\]  
It follows from the construction that $X^{G_n}$ has the law of a simple random walk on $G_n$. The reader should note that the actual definition of $X^{G_n}$ depends on $K$, but we are mainly interested in its law, which is independent of $K$, so we will not emphasize the $K$ dependence in the notation.
Let 
\begin{equation}\label{eq:defofTnK}
S^{(n,K)}(m):=\min\{k\geq 0: A^{(n,K)}(k)\geq m\} \quad m\in\N,
\end{equation}
be the (generalized) inverse of $A^{(n,K)}$. We extend the domain of $S^{(n,K)}$ to $\mathbb{R}_+$ by linear interpolation.
 Under the coupling constructed above, and using the same reasoning as in the proof of Lemma \ref{lem:jnktobnk} we obtain:
\begin{lemma}\label{lem:XtoJ}  If the random augmented graphs $(G_n,(V^i_n)_{i\in\N})_{n\in\N}$ verify condition $(S)$, then for all $\epsilon>0$
\[\lim_{K\to\infty}\limsup_{n\to\infty}\mathbf{P}^{(K)}\left[\sup_{m\in \N}d_{\R^d}( n^{-1/4}X^{G_n}_{m}, n^{-1/4}\phi_{\T_n}(J^{(n,K)}_{S^{(n,K)}(m)})) \geq\epsilon\right]=0.\]
\end{lemma}

\subsection{Asymptotic linearity of the time change}\label{sect_timechange}

The following important lemma shows that $t\mapsto A^{(n,K)}(m(t))$ can be rescaled to converge to a linear function.
\begin{proposition}\label{l:Ankislinear} Consider a sequence of random augmented graphs $(G_n,(V_i^n)_{i\in \N})_{n\in \N}$ which verifies condition $(S)$  and  condition $(G)_{\sigma_d,\sigma_{\phi}}$, $(V)_\nu$ and $(R)_{\rho}$. Then for each $t\geq0,\epsilon>0$,
\[\lim_{K\to\infty}\limsup_{n\to\infty}\mathbf{P}^{(K)}\left[|n^{-3/2}A^{(n,K)}(m(t))-\nu  t|\geq\epsilon\right]=0.\]
\end{proposition}
The proof of this lemma is one of the main tasks of  this paper and is postponed to Section~\ref{s:AnKislinear}.
For the next pages we will assume Lemma \ref{l:Ankislinear} and deduce the main theorem (Theorem \ref{thm_abstract}).
Recall the definition of $T^{(n,K)}$ from \eqref{eq:defofTnK}.
\begin{corollary}\label{cor:TnK}
 Assume the hypotheses of Lemma \ref{l:Ankislinear}. For all $\eta,T>0$, let
\[\frak{B}_{n,K}^c:=\{m(\nu^{-1}t-\eta)\leq S^{(n,K)}(tn^{3/2})\leq m(\nu^{-1}t+\eta)\text{ for all }t\leq T\}.\]
Then
\[\lim_{K\to\infty}\limsup_{n\to\infty}\mathbf{P}^{(K)}\left[\frak{B}_{n,K}\right]=0.
\]
\end{corollary}
\begin{proof}[Proof of Corollary \ref{cor:TnK}]
We have that $\frak{B}_{n,K}=\frak{B}_{n,K}^1\cup\frak{B}_{n,K}^2$, where
\[\frak{B}_{n,K}^1:=\{S^{(n,K)}(tn^{3/2})>m(\nu^{-1}t +\eta)\text{ for some } t\leq T\}\]
 and
  \[\frak{B}_{n,K}^2:=\{S^{(n,K)}(tn^{3/2})<m(\nu^{-1}t -\eta)\text{ for some }t\leq T\}.\]
It follows directly from the definition of $S^{(n,K)}$ that on $\frak{B}^1_{n,K}$
 \[
\min\{k: A^{(n,K)}(k)\geq tn^{3/2}\}>m(\nu^{-1}t+\eta),
\]
for some $t\leq T$.
Therefore
\[
A^{(n,K)}(m(\nu^{-1}t+\eta))<n^{3/2}t\]
for some $t\leq T$.
That is
\begin{equation}\label{eq:coipos}
n^{-3/2} A^{(n,K)}(m(\nu^{-1}t+\eta))<t.
\end{equation}
for some $t\leq T$.
On the other hand, 
since, for all $n,K\in\N$, $t\mapsto A^{n,K}(m(t))$ is a monotone function and $t\mapsto \nu \rho t$ is continuous, it follows from Lemma \ref{l:Ankislinear} that
\begin{equation}\label{eq:AnKislinearsup}
\lim_{K\to\infty}\limsup_{n\to\infty}\mathbf{P}^{(K)}\left[\sup_{t\leq R}\abs{n^{-3/2}A^{(n,K)}(m(t))-\nu  t}\geq\epsilon\right]=0,
\end{equation}
for all $\epsilon,R>0$. Choosing $\epsilon<\nu\eta$ in the display above, it follows that the probability of the event in display \eqref{eq:coipos} goes to $0$ as $n\to\infty$ and then $K\to\infty$. Therefore
\[\lim_{K\to\infty}\lim_{n\to\infty}\mathbf{P}^{(K)}[\frak{B}_{n,K}^1]=0.\]
Analogous arguments give that
\[\lim_{K\to\infty}\lim_{n\to\infty}\mathbf{P}^{(K)}[\frak{B}_{n,K}^2]=0.\]
That finishes the proof
\end{proof}
\subsection{Proof of the Abstract Theorem}



\begin{proof}[Proof of Theorem \ref{thm_abstract}]
Let $T\geq0$ be fixed.
By Proposition \ref{prop_approx_bisek} we have that
\begin{equation}\label{eq:bisetobisek}
\mathbf{P}^{(K)}\left[\sup_{t\leq T}d_{\R^d}(B^{ISE}_{t},B^{K-ISE}_{t})\geq \epsilon\right]\to 0 \quad \text{as }K\to\infty,
\end{equation}
 for all $\epsilon>0$.
Also, by Lemma \ref{lem:BnktoBk}
\begin{equation}\label{bnktobisek}
\mathbf{P}^{(K)}\left[\sup_{t\leq T}d_{\R^d}(n^{-1/4}\phi_{\T_n}(B^{(n,K)}_{t}),\sigma_{\phi} \sqrt{\sigma_d} B^{K-ISE}_{(\rho\sigma_d)^{-1}t}))\geq\epsilon\right]\to0,
\end{equation}
as $n\to\infty$, for all $\epsilon>0$.
On the other hand, by Lemma \ref{lem:jnktobnk}
\begin{equation}\label{eq:bnktojnk}
\lim_{K\to\infty}\limsup_{n\to\infty}\mathbf{P}^{(K)}\left[\sup_{t\leq T} d_{\R^d}( n^{-1/4}\phi_{\T_n}(B^{(n,K)}_t),n^{-1/4}\phi_{\T_n}(J^{(n,K)}_{m(t)}))\geq \epsilon \right] =0,
\end{equation}
for all $\epsilon>0$.
Now, since by construction $\phi_{\T_n}(J^{(n,K)}_{m(t)})=X^{G_n}_{A^{(n,K)}(m(t))}$, by the three displays above we get that
\begin{equation}\label{eq:preconvergence}
\lim_{K\to\infty}\limsup_{n\to\infty} \mathbf{P}^{(K)}\left[\sup_{t\leq T} d_{\R^d}(n^{-1/4}X^{G_n}_{A^{(n,K)}(m(t))},\sigma_{\phi}\sqrt{\sigma_d}B^{ISE}_{(\rho\sigma_d)^{-1}t}) \geq \epsilon\right]=0,
\end{equation}
for all $\epsilon>0$.

Corollary \ref{cor:TnK}, together with the continuity of $B^{ISE}$ and \eqref{eq:preconvergence}
imply that
\begin{equation}\label{eq:precvg}
\lim_{K\to\infty}\limsup_{n\to\infty}\mathbf{P}^{(K)}\left[\sup_{t\leq T} d_{\Z^d}(n^{-1/4}X^{G_n}_{A^{(n,K)}(S^{(n,K)}(n^{3/2}t))},\sigma_{\phi}\sqrt{\sigma_d}B^{ISE}_{(\rho\nu\sigma_d)^{-1}t} )\geq \epsilon\right]=0,
\end{equation}
for all $\epsilon>0$.

On the other hand, by \eqref{eq:coupling} and Lemma \ref{lem:XtoJ} we have that
\begin{equation}\label{eq:XtoJ}
\lim_{K\to\infty}\limsup_{n\to\infty}\mathbf{P}^{(K)}\left[ \sup_{t\geq 0}d_{\Z^d}(n^{-1/4}X^{G_n}_{n^{3/2}t},n^{-1/4} X^{G_n}_{A^{(n,K)}(S^{(n,K)}(n^{3/2}t))} ) \geq \epsilon\right]=0,
\end{equation} 
for all $\epsilon >0$.
This, together with \eqref{eq:precvg}, proves the theorem. 

 \end{proof}

\section{Strengthened resistance estimates}\label{sect_resistance_estimate}

Our goal in this section is to turn the point to point resistance estimate from condition $(R)$ into resistances estimates which are essentially uniform over $\T^{(n,K)}$.

\subsection{Creating $\delta$-dense sets by spanning uniform points}

Fix $\delta>0$ and $K,K'\in \N$ and an augmented graph $(G,(V_i)_{i\in \N})$ which is asymptotically tree-like. Recalling the definition of $\pi^{(G,K)}$ in Section~\ref{sect_mu}, we say that $V_0,\ldots V_{K'}$ is $\delta$-dense in $\T^{(G,K)}$ if 
\begin{enumerate}
\item the set $\{\pi^{(G,K)}(V_{l})\text{  with }l \leq K'\}$ has at least one point on every edge of $\mathfrak{T}^{(G,K)}$.
\item If $x$, $y \in \{\pi^{(G,K)}(V_{l})\text{  with } l \leq K'\}$  are neighbours (in  the sense that on the unique  path between them there is no other point in $\{\pi^{(G,K)}(V_{l}) \text{ with }l \leq K'\}$) then there exists $i_x,i_y \leq K'$ such that $\pi^{(G,K)}(V_{i_y})=x$, $\pi^{(G,K)}(V_{i_y})=y$ and $d_{\T^{(G,K')}}(V_{i_x},V_{i_y})\leq \delta$.
\end{enumerate}

\begin{figure}
  \includegraphics[width=0.75\linewidth]{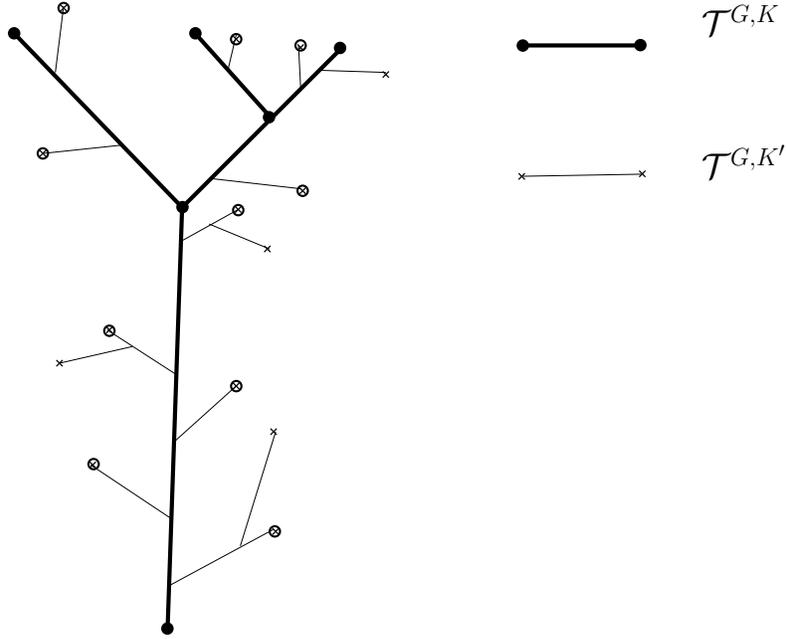}
  \caption{In a $\delta$-dense setting, the maximal distance between neighbouring points among the leaves  of $\T^{(G,K)}$ and the circled ones of $\T^{(G,K')}$ is less than $\delta$.}
\end{figure}

First let us show that spanning uniform points on the $K$-CRT creates $\delta$-dense sets with high probability.
\begin{lemma}\label{epsnet0}
Consider  $(\mathfrak{T},(U_i)_{i\in \N})$ a CRT with a sequence of uniformly chosen points (i.e.~chosen according to $(\lambda^{\mathfrak{T}})^{\otimes \N}$). Fix $\epsilon'>0$ and $\delta>0$, there exists $K'$ such that we have 
\[
M \otimes (\lambda^{\mathfrak{T}})^{\otimes \N}[U_0\ldots U_{K'} \text{ is not $\delta$-dense in }\mathfrak{T}^{(K)}] \leq \epsilon'.
\]
\end{lemma}
\begin{proof}
For $x\in\frak{T}, r>0$ we let $B(x,r):=\{y\in\frak{T}: d_\frak{T}(x,y)<\delta\}$. We claim that 
\begin{equation}\label{eq:covering}
(B(U_i,\delta))_{i\in\N} \text{ is almost surely a covering of } \mathfrak{T}.
\end{equation}
 First notice that $\mathfrak{T}$ is compact, because it is defined as a  quotient of $[0,1]$, which is itself compact. Therefore $\mathfrak{T}$ is separable, i.e., there exists a dense sequence $(x_i)_{i\in\N}$. 
 Hence, to show \eqref{eq:covering}, it is enough to show that 
 \begin{equation}\label{eq:density}
 P[(x_i)_{i\in\N}\subset\cup_{i\in\N}B(U_i,\delta)]=1.
 \end{equation}
 But the display above holds because $P[x_i\in B(U_j,\delta)]=\lambda_{\mathfrak{T}}(B(x_i,\delta))>0$ (This follows from the fact that $\lambda^{\mathfrak{T}}(A)>0$ for any open set $A$, which, in turn, follows directly from the definition of $\lambda^{\mathfrak{T}}$). Since the $(U_j)_{j\in\N}$ are independent we get that $P[x_i\in \cup_{j\in\N} B(U_j,\delta)]=1$. Moreover, since the $(x_i)_{i\in\N}$ are countable, the claim at \eqref{eq:density} holds true. This shows \eqref{eq:covering}.
 
 From display \eqref{eq:covering} and the fact that $\mathfrak{T}$ is compact  we  get
 \[P[\exists K :\mathfrak{T} \subset \cup_{i=1}^KB(U_i,\delta)]=1\]
 and therefore
 \[\lim_{K\to\infty}P[\mathfrak{T}\subset(B(U_i,\delta))_{i\leq K}]=1.\]
 Hence
 \[\lim_{K\to\infty}P[(U_i)_{i\leq K} \text{ is not a }\delta \text{ covering of }\mathfrak{T}]=0\]
 and this proves the lemma. 
 
\end{proof}

This result translates into the existence of dense sets for certain augmented random graphs.
\begin{lemma}\label{epsnet}
Consider a sequence of random augmented graphs $(G_n,(V_i^n)_{i\in \N})_{n\in N}$ which verifies condition $(S)$ and condition $(G)_{\sigma_d,\sigma_{\phi}}$.

Fix $\epsilon'>0$ and $\delta>0$, there exists $K'$ such that we have 
\[
  \limsup_{n\to \infty} {\bf P}_n[V_0^n\ldots V_{K'}^n \text{ is not $\delta$-dense in  }\T^{(n,K)}] \leq \epsilon'.
\]
\end{lemma}

Let us emphasize the fact that there is no factor $n^{1/2}$ because the distance on $\T^{(n,K)}$ is already rescaled.

\begin{proof}
We can notice that the event $\{V_0^n\ldots V_{K'}^n \text{ is not $\delta$-dense in  }\T^{(n,K)}\}$ is measurable with respect to a finite set of conditions on the shape of $\mathfrak{T}^{(n,K')}$ and the distances within this graph, since for $i\leq K$ the points $\pi^{(G,K')}(V_i)$ are branch-points of $\mathfrak{T}^{(n,K')}$. Hence by condition $(G)$, we know that 
\begin{align*}
&  \lim_{n\to \infty} {\bf P}_n[V_0^n\ldots V_{K'}^n \text{ is not $\delta$-dense in  }\T^{(n,K)}] \\
=&M \otimes (\lambda^{\mathfrak{T}})^{\otimes \N}[U_0\ldots U_{K'} \text{ is not $\sigma_d\delta$-dense in }\mathfrak{T}^{(K)}] ,
 \end{align*}
 and hence the result follows from Lemma~\ref{epsnet0}.
\end{proof}

\subsection{Strengthened resistance estimate}

Let us start by a technical result stating, in essence, that the bubble containing $0$ is small.
\begin{lemma}\label{bubble_zero}
Consider a sequence of random augmented graphs $(G_n,(V_i^n)_{i\in \N})_{n\in \N}$ which verifies condition $(S)$.

For any $\epsilon>0$, we have that
\[
\limsup_{n\to\infty} {\bf P}_n\Bigl[\frac{d_{G_n}(0,\text{root}^*)}{n^{1/2}}>\epsilon\Bigr]=0.
\]
\end{lemma}

\begin{proof}
Fix $\delta>0$. One can see that, by definition, $\pi^{(n,K)}(0)=\text{root}^*$ for all $K\geq 1$. Thus we can see that $d_{G_n}(0,\text{root}^*)\leq \Delta^{(n,K)}_{G_n}$ for all $K\in \N$. Hence, for all $K\in \N$
\[
\limsup_{n\to\infty}{\bf P}_n\Bigl[\frac{d_{G_n}(0,\text{root}^*)}{n^{1/2}}>\epsilon\Bigr] \leq \limsup_{n\to\infty}{\bf P}_n\Bigl[\frac{\Delta^{(n,K)}_{G_n}}{n^{1/2}}>\epsilon\Bigr],
\]
and the result follows by taking $K$ to infinity and using the second property of condition $(S)$. 
\end{proof}

We can also prove that
\begin{lemma}\label{bounded_dist}
Consider a sequence of random augmented graphs $(G_n,(V_i^n)_{i\in \N})_{n\in N}$ which verifies condition $(S)$  and condition $(G)_{\sigma_d,\sigma_{\phi}}$.

For any $\epsilon>0$ and $K\in \N$, there exists $C_{\epsilon,K}<\infty$ such that
\[
\limsup_{n\to \infty} {\bf P}_n\Bigl[\max_{i=1\ldots K}  n^{-1/2}d_{G_n}(0,V_i^{n})>C_{\epsilon,K}\Bigr]\leq \epsilon.
\]
\end{lemma}

\begin{proof}
We claim that, for any $\epsilon>0$ and $K\in \N$,  there exists $C_{\epsilon,K}<\infty$ such that 
\begin{equation}\label{raaaaa}
M \otimes (\lambda^\T)^{\otimes \N}\Bigl[\max_{i=1\ldots K}  d_{\T^{(K)}}(0,U_i)>C_{\epsilon,K}\Bigr]\leq \epsilon.
\end{equation}
Indeed, 
\begin{equation}
M \otimes (\lambda^\T)^{\otimes \N}\Bigl[\max_{i=1\ldots K}  d_{\T^{(K)}}(0,U_i)>C\Bigr]\leq M\Bigl[\max_{x\in \T^{(K)}}  d_{\T^{(K)}}(0,x)>C\Bigr],
\end{equation}
and, since $\max_{x\in \T^{(K)}}  d_{\T^{(K)}}(0,x)<\infty$, we have that $M \Bigl[\max_{x\in \T^{(K)}}  d_{\T^{(K)}}(0,x)>C\Bigr]\to 0$ as $C\to\infty$, and this shows \eqref{raaaaa}.
Then we can see that 
\begin{align*}
 & {\bf P}_n\Bigl[\max_{i=1\ldots K}  n^{-1/2}d_{G_n}(0,V_i^{n})>2C_{\epsilon,K}\sigma_d \Bigr]\\
 \leq & {\bf P}_n\Bigl[\max_{i=1\ldots K}  n^{-1/2}d_{G_n}(0,\text{root}^*)>C_{\epsilon,K}\sigma_d\Bigr]+  {\bf P}_n\Bigl[\max_{i=1\ldots K}  d^{(n,K)}(\text{root}^*,V_i^{n})>C_{\epsilon,K}\sigma_d\Bigr],
 \end{align*}
 where the first term goes to $0$ by Lemma~\ref{bubble_zero}. The second term is an event defined in terms of a finite set of condition on distances between vertices in $\mathfrak{T}^{(n,K)}$ so we can use condition $(G)$  and~\eqref{raaaaa} to see that 
\begin{align*}
& \limsup_{n\to\infty}  {\bf P}_n\Bigl[\max_{i=1\ldots K}  d^{(n,K)}(\text{root}^*,V_i^{n})>C_{\epsilon,K}\sigma_d\Bigr]  \\
=& M \otimes (\lambda^\T)^{\otimes \N}\Bigl[\max_{i=1\ldots K}  d_{\T^{(K)}}(0,U_i)>C_{\epsilon,K}\Bigr] \\
\leq & \epsilon,
\end{align*}
for all $\epsilon>0$, which finishes the proof.
\end{proof}

We are now able to prove that the resistance-distance and the graph distance are equivalent when looking at the vertices $(V_i)_{i\leq K}$.
\begin{lemma}\label{dnetresis}
Consider a sequence of random augmented graphs $(G_n,(V_i^n)_{i\in \N})_{n\in N}$ which  verifies condition $(S)$  and condition $(G)_{\sigma_d,\sigma_{\phi}}$ and $(R)_{\rho}$.

Fix $\epsilon>0$ and $K\in \N$. We have
\[
\limsup_{n\to \infty} {\bf P}_n\Bigl[\max_{i=1\ldots K} \abs{ d_{\text{res}}^{(n,K)}(\text{root}^*,V_i^{n}) - \rho d^{(n,K)}(\text{root}^*,V_i^{n})}>\epsilon\Bigr]=0.
\]
\end{lemma}

\begin{proof}
Fix $\epsilon_1>0$ and $K\in \N$, by Lemma~\ref{bounded_dist} there exists $C_{\epsilon_1,K}<\infty$ such that
\[
\limsup_{n\to \infty} {\bf P}_n\Bigl[\max_{i=1\ldots K}  n^{-1/2}d_{G_n}(0,V_i^{n})>C_{\epsilon_1,K}\Bigr]\leq \epsilon_1,
\]
where we denote $A_n^c(\epsilon_1)$ the event inside the probability. Then, see that for all $\epsilon_1>0$ and $\epsilon>0$, we have
\begin{align*}
& \limsup_{n\to \infty} {\bf P}_n\Bigl[\max_{i=1\ldots K} n^{-1/2}\abs{ d^{\text{res}}_{G_n}(0,V_i^{n}) - \rho d_{G_n}(0,V_i^{n})}>\epsilon\Bigr] \\ 
\leq & \limsup_{n\to \infty} \Bigl( {\bf P}_n\Bigl[\max_{i=1\ldots K} \abs{\frac{d^{\text{res}}_{G_n}(0,V_i^{n})}{d_{G_n}(0,V_i^{n})}  - \rho  }>\frac{\epsilon}{n^{-1/2}\max_{i=1\ldots K} d_{G_n}(0,V_i^{n})},A_n(\epsilon_1)\Bigr] \\
& \qquad \qquad \qquad \qquad \qquad +{\bf P}_n[A_n^c(\epsilon_1)] \Bigl)
\\ \leq &  \limsup_{n\to \infty} {\bf P}_n\Bigl[\max_{i=1\ldots K}\abs{\frac{d^{\text{res}}_{G_n}(0,V_i^{n})}{d_{G_n}(0,V_i^{n})}  - \rho  }>\frac{\epsilon}{C_{\epsilon_1,K}}\Bigr]+\epsilon_1,
\end{align*}
by the previous equation. But from condition $(R)$ and a simple union bound, we see that for any $\epsilon_2>0$ we have
\[
\limsup_{n\to \infty} {\bf P}_n\Bigl[\max_{i=1\ldots K} \abs{ \frac{d^{\text{res}}_{G_n}(0,V_i^{n})}{ d_{G_n}(0,V_i^{n})} - \rho}>\epsilon_2\Bigr]=0,
\]
which combined with the previous equation, used with $\epsilon_2=\epsilon/ C_{\epsilon_1,K}$. yields that for any $\epsilon>0$
\[
\limsup_{n\to \infty} {\bf P}_n\Bigl[\max_{i=1\ldots K} n^{-1/2}\abs{ d^{\text{res}}_{G_n}(0,V_i^{n}) - \rho d_{G_n}(0,V_i^{n})}>\epsilon\Bigr]=0.
\]
by letting $\epsilon_1$ go to $0$ after taking $n$ to infinity. Using this and Lemma~\ref{bubble_zero} (and the fact that $d_{\text{res}}^{(n,K)}(\cdot,\cdot)\leq d^{(n,K)}(\cdot,\cdot)$ which follows from Rayleigh's monotonicity principle) we see that for any $\epsilon>0$ we have
\[
\limsup_{n\to \infty} {\bf P}_n\Bigl[\max_{i=1\ldots K} n^{-1/2}\abs{ d^{\text{res}}_{G_n}(\text{root}^*,V_i^{n}) - \rho d_{G_n}(\text{root}^*,V_i^{n})}>\epsilon\Bigr]=0.
\]

The result then follows from  the fact  that, since $\text{root}^*$ and $V_i^{n}$ are in $V^*(\T^{(n,K)})$ and thus correspond to cut-points in $G_n$, we have $n^{-1/2}d^{\text{res}}_{G_n}(\text{root}^*,V_i^{n})=d_{\text{res}}^{(n,K)}(\text{root}^*,V_i^{n})$ and  $n^{-1/2}d_{G_n}(\text{root}^*,V_i^{n})=d^{(n,K)}(\text{root}^*,V_i^{n})$ for all $i\leq K$.
\end{proof}

It allows us to prove our final result
\begin{lemma}\label{res_2a}
Consider a sequence of random augmented graphs $(G_n,(V_i^n)_{i\in \N})_{n\in N}$ which verifies condition $(S)$ and condition $(G)_{\sigma_d,\sigma_{\phi}}$ and $(R)_{\rho}$.

Fix $\epsilon>0$ and $K\in \N$. We have
\[
\limsup_{n\to \infty} {\bf P}_n\Bigl[\max_{x\in \T^{(n,K)}} \abs{ d_{\text{res}}^{(n,K)}(\text{root}^*,x) - \rho d^{(n,K)}(\text{root}^*,x)}>\epsilon\Bigr]=0.
\]
\end{lemma}

\begin{proof}
Fix $\epsilon>0$ 

 By Lemma~\ref{epsnet}, there exists $K'$  such that 
  any vertex in $V(\T^{(n,K)})$ is within a small $d^{(n,K)}$-distance of at least one $V^n_i$ with $i\leq K'$. More precisely, for any $v\in V(\T^{(n,K)})$, there exists $i(v)\leq K'$ with 
 \begin{equation}\label{eq:eleven}
 \bold{P}\left[\max_{v\in V(\T^{(n,K)})} d^{(n,K')}(v, V_{i(v)}^n)\geq \epsilon\right]\leq \epsilon.
 \end{equation} 
 

On the other hand
\begin{align*}
 &\abs{\rho d^{(n,K)}(\text{root}^\ast,v)- d^{(n,K)}_{\text{res}}(\text{root}^\ast,v)}\\
\leq &\abs{\rho d^{(n,K')}(\text{root}^\ast,v)-\rho d^{(n,K')}(\text{root}^\ast,V_{i(v)}^n)}\\
& \qquad + \abs{\rho d^{(n,K')}(\text{root}^\ast,V_{i(v)}^n)-d_{\text{res}}^{(n,K')}(\text{root}^\ast ,V_{i(v)}^n)}\\
&\qquad +\abs{d_{\text{res}}^{(n,K')}(\text{root}^\ast,V_{i(v)}^n)- d_{\text{res}}^{(n,K')}(\text{root}^\ast,v)} \\
\leq & \rho d^{(n,K')}(V_{i(v)}^n,v)+\abs{\rho d^{(n,K')}(\text{root}^\ast,V_{i(v)}^n)-d_{\text{res}}^{(n,K')}(\text{root}^\ast ,V_{i(v)}^n)}\\
& \qquad+d_{\text{res}}^{(n,K')}(v,V_{i(v)}^n).
\end{align*}

The second summand of the right hand side above can be controlled (uniformly on $v$) by Lemma \ref{dnetresis}, while the first and third summand can be controlled (also uniformly on $v$) recalling \eqref{eq:eleven} and the fact that  $d^{(n,K)}_{\text{res}}(\cdot,\cdot)\leq d^{(n,K)}(\cdot,\cdot)$.
\end{proof}

\subsection{No edge has macroscopic resistance}
The previous result and the existence of $\delta$-dense sets allow us to prove
\begin{lemma}\label{lem_no_macro_res}
Consider a sequence of random augmented graphs $(G_n,(V_i^n)_{i\in \N})_{n\in N}$ which verifies condition $(S)$ and condition $(G)_{\sigma_d,\sigma_{\phi}}$ and $(R)_{\rho}$.
 For $K \in \N$ and $\epsilon>0$ we have that,
\[
\limsup_{n\to \infty} {\bf P}_n\Bigl[\max_{e\in E(\T^{(n,K)})} R^{\T^{(n,K)}}(e) n^{-1/2} >\epsilon\Bigr]=0.
\]
\end{lemma}

\begin{proof}
We can notice that
\begin{align*}
\max_{e\in E(\T^{(n,K)})} R^{\T^{(n,K)}}(e) n^{-1/2} & = \max_{\substack{x,y\in V(\T^{(n,K)}) \\ x\sim y}} R^{\T^{(n,K)}}(x, y)n^{-1/2} \\
 &  = \max_{\substack{x,y\in V(\T^{(n,K)}) \\ x\sim y}} \abs{d_{\text{res}}^{(n,K)}(0,x)-d_{\text{res}}^{(n,K)}(0,y)},
 \end{align*}
 by the law of resistances in series applied  in the tree $\T^{(n,K)}$, which is composed of cut-points. Hence, using that $\rho\leq 1$,
 \begin{align*}
 & \limsup_{n\to \infty} {\bf P}_n\Bigl[\max_{e\in E(\T^{(n,K)})} R^{\T^{(n,K)}}(e) n^{-1/2} >\epsilon\Bigr] \\ 
 \leq &  2\limsup_{n\to \infty} {\bf P}_n\Bigl[\max_{x\in \T^{(n,K)}} \abs{ d_{\text{res}}^{(n,K)}(\text{root}^*,x) - \rho d_{\T^{(n,K)}}(\text{root}^*,x)}>\epsilon/3\Bigr] \\
  & \qquad \qquad + \limsup_{n\to\infty} {\bf P}_n\Bigl[\max_{\substack{x,y\in V(\T^{(n,K)}) \\ x\sim y}} \abs{d_{\T^{(n,K)}}(\text{root}^*,x)-d_{\T^{(n,K)}}(\text{root}^*,y)} >\epsilon/3 \Bigr] \\ 
  =& \limsup_{n\to\infty} {\bf P}_n\Bigl[\max_{\substack{x,y\in V(\T^{(n,K)}) \\ x\sim y}} \abs{d_{\T^{(n,K)}}(\text{root}^*,x)-d_{\T^{(n,K)}}(\text{root}^*,y)}>\epsilon/3 \Bigr],
  \end{align*}
  where we used Lemma~\ref{res_2a}.
  
  Now, since any edge in $\T^{(n,K)}$ is contained in an edge in $\mathfrak{T}^{(n,K')}$ (recall the definitions in Section~\ref{sect_mathfrak}) we can notice that for all $K'\geq K$ we have 
  \begin{align*}
&  \max_{\substack{x,y\in V(\T^{(n,K)}) \\ x\sim y}} \abs{d_{\T^{(n,K)}}(\text{root}^*,x)-d_{\T^{(n,K)}}(\text{root}^*,y)} \\
 \leq & \max_{\substack{i,j \leq K'\\ V_i^n \sim V_j^n\text{ in } \mathfrak{T}^{(n,K')}}}d_{\mathfrak{T}^{(n,K)}}(V_i^n,V_j^n),
  \end{align*}
  so 
   \begin{align}\label{fuuu}
 & \limsup_{n\to \infty} {\bf P}_n\Bigl[\max_{e\in E(\T^{(n,K)})} R^{\T^{(n,K)}}(e) n^{-1/2} >\epsilon\Bigr] \\ \nonumber
 \leq & \liminf_{K'\to \infty} {\bf P}_n\Bigl[\max_{\substack{i,j \leq K'\\ V_i^n \sim V_j^n\text{ in } \mathfrak{T}^{(n,K')}}} n^{-1/2}d_{\mathfrak{T}^{(n,K)}}(V_i^n,V_j^n)\leq \epsilon/3\Bigr].
 \end{align}

However, if $ \Bigl\{\displaystyle{\max_{\substack{i,j \leq K'\\ V_i^n \sim V_j^n\text{ in } \mathfrak{T}^{(n,K')}}}} n^{-1/2}d_{\mathfrak{T}^{(n,K)}}(V_i^n,V_j^n)>\epsilon/3\Bigr\}$ then the set $V_0^n,\ldots,V_{K'}^n$ is not $\epsilon/3$-dense. Hence, using Lemma~\ref{epsnet} we can see that right hand side of~\eqref{fuuu} is $0$ which proves the lemma
                             
\end{proof}

 \section{Convergence of local times}\label{section_cvg_local_time}

In this section we obtain estimates on local times that will be used to prove Proposition~\ref{l:Ankislinear}.

\subsection{Statement of results and organization of the section}
Recall the definition of $J^{(n,K)}$ from \eqref{eq:defofJnK}, $B^{(n,K)}$ above Remark~\ref{belowbnk} and that of $B^{(K)}$ from Section \ref{sect_KISE}. The aim of this subsection is to show that the discrete local times of $J^{(n,K)}$ are close to those of $B^{(n,K)}$. 

In all this section, we will assume that $(G_n,(V_i^n)_{i\in \N})_{n\in \N}$   verifies condition $(S)$  and condition $(G)_{\sigma_d,\sigma_{\phi}}$, $(V)_\nu$ and $(R)_{\rho}$.   

 Let $(L^{(n,K)}_t(x))_{x\in\frak{T}^{(n,K)},t\geq 0}$ be the a jointly continuous local time of $B^{(n,K)}$ with respect to the measure $\lambda^{(n,K)}_{\text{res}}$ according to Definition \ref{def:localtime} (for the existence of this process we refer to Lemma 3.3 in \cite{Croydon_crt}). 
 
 For all $e\in E(\T^{(n,K)})$, let $l^{(n,K)}_t(e)$ be the number of times that $J^{(n,K)}$ has crossed the vertex $e$ up to time $m(t)$ (defined in display \eqref{eq:defofmt}). Note that, by the coupling between $J^{(n,K)}$ and $B^{(n,K)}$, $l^{(n,K)}_t(e)$ coincides with the number of crossings of $e$ by $B^{(n,K)}$ up to time $t$. 
 
Similarly, for $x\in V^\ast(\T^{(n,K)})$, let 
\begin{equation}\label{eq:defoflvert}
l^{(n,K,\text{vert})}_m(x):=\sum_{i=0}^{m}1_{\{J^{(n,K)}_i=x\}}.
\end{equation}

 Recall from Section \ref{sect_condG} that $\Upsilon_{\frak{T}^{(K)},\frak{T}^{(n,K)}}$ is the homeomorphism between $(\frak{T}^{(K)},d_{\frak{T}})$ and $(\frak{T}^{(n,K)},d_{\text{res}}^{(n,K)})$ which is linear along the edges (and whose existence, for $n$ large enough, is guaranteed $\mathbf{P}^{(K)}$-a.s.~by display \eqref{eq:couplinggeometry} ). For the sake of simplicity we will denote 
 \begin{equation}\label{eq:defofupsilon}
\Upsilon_{n,K}:= \Upsilon_{\frak{T}^{(K)},\frak{T}^{(n,K)}}.
\end{equation}
Let $V^{\circ}(\T^{(n,K)})$ be the set of vertices which are not leaves\footnote{a leaf is a vertex of degree $1$}  nor branching points and whose two neighbors are in $V^\ast(\T^{(n,K)})$.

For each $v\in V^\circ(\T^{(n,K)})$ \hfff{vcirc}, let $e^-(v)\in E(\T^{(n,K)})$ be the only edge of the form $(v^-,v)$ for some $v^-\in V(\T^{(n,K)})$  and satisfying that $v^-$ is closer to the root than $v$. That is, $e^-(v)$ is the edge preceding $v$. Also, for any $v\in V^\circ(\T^{(n,K)})$, let $e^+(v)\in E(\T^{(n,K)})$ be the only edge of the form $(v,v^+)$, for some $v^+\in V(\T^{(n,K)})$ satisfying that $v^+$ is farther from the root than $v$. We emphasize that the uniqueness of this edge is guaranteed by the fact that $v$ is not a branching point.
Let 
\begin{equation}\label{eq:sopracontrariar}
\reff^{(n,K,\text{vert})}(v):=\frac{ d^{(n,K)}_{\text{res}}(e^-(v))d^{(n,K)}_{\text{res}}(e^+(v))}{d^{(n,K)}_{\text{res}}(e^+(v))+d^{(n,K)}_{\text{res}}(e^-(v))}.
\end{equation}
Recall that $(L^{(K)}_t(x))_{x\in\frak{T},t\geq0}$ is the local time of $B^{(K)}$ with respect to the measure $\lambda_{\frak{T}}^{(K)}$.   
For each $x\in \T^{(n,K)}$, let $e(x)$ be the edge containing $x$. 
The main result of this subsection is the following:

\begin{proposition}\label{l:convergenceoflocaltimes} Under condition $(R)_\rho$, for all $K\in\N, t\geq0$ and $\epsilon>0$. 
\begin{equation}\label{eq:convergenceofvertexlocaltimes}
\lim_{n\to\infty}\mathbf{P}^{(K)}\left[\sup_{v\in V^\circ(\T^{(n,K)})}\abs{ 2 \reff^{(n,K,\text{vert})}(v)l^{(n,K,\text{vert})}_{m(t)}(v)-L^{(n,K)}_{t}(v)}\geq \epsilon \right]=0
\end{equation}
and
\begin{equation}\label{eq:convergenceofedgelocaltimes}
\lim_{n\to\infty}\mathbf{P}^{(K)}\left[\sup_{x\in \T^{(n,K)}}\abs{  d^{(n,K)}_{\text{res}}(e(x))l^{(n,K)}_{t}(e(x))-L^{(n,K)}_{t}(x)}\geq \epsilon \right]=0.
\end{equation}
If, in addition, we take into account \eqref{eq:couplinggeometry}, there exists a coupling between $l^{(n,K)}$ and $L^{(K)}$ such that
\begin{equation}\label{eq:conv3}
\lim_{n\to\infty}\mathbf{P}^{(K)}\left[\sup_{x\in \T^{(n,K)}}\abs{  d^{(n,K)}_\text{res}(e(x))l^{(n,K)}_{t}(e(x))-\rho\sigma_d L^{(K)}_{(\rho\sigma_d)^{-1}t}(\Upsilon^{-1}_{n,K}(x))} \geq \epsilon \right]=0.\end{equation}
\end{proposition}

In order to prove this proposition we will require three ingredients
\begin{enumerate}
\item estimates on the time between excursions between  two edges which are at macroscopic (but small) distance,
\item estimates on the number excursions between two edges which are at macroscopic (but small) distance,
\item a regularity estimate on local times in the space variable.
\end{enumerate}

The former two points will be used to obtain a convergence in the finite-dimensional distributional sense in the previous proposition and the later ingredient will act as a tightness estimate.

The organization of this section is as follows, the three estimates will be proved in Section~\ref{sect:time_excursion}, Section~\ref{sect:number_excursion} and Section~\ref{sect:regularity_local_time} respectively. Finally Proposition~\ref{l:convergenceoflocaltimes} will be proved in Section~\ref{proof_prop61a} and Section~\ref{proof_prop61b}.


\subsection{Time of excursions between edges} \label{sect:time_excursion}

Our goal in this section is to control the time of excursions between edges, which will give us a tail estimate on local times.

\subsubsection{Control on excursion time}

 For any oriented edge\footnote{For oriented edge we mean a pair $(e_-,e_+)\in V(\T^{(n,K)})$ with $e_-$ closer to the root than $e_+$.} $e=(e_-,e_+)\in E(\T^{(n,K)})$, let $l^{(n,K,\rightarrow)}_t(e)$ be the number of rightward crossings of $e$ (i.e., from $e_-,e_+$) by $B^{(n,K)}$ before time $t$.
More precisely, let $\theta^{\text{in}}_0(e):=0$,
  \begin{equation}\label{eq:inandoutanedge}
 \begin{aligned} 
 &\theta_i^{\text{out}}(e):=\inf\{s\geq \theta^{\text{in}}_{i-1}(e): B^{(n,K)}_s=e_+ \} \quad i\in\N,\\
   &\theta^{\text{in}}_i(e):=\inf\{s\geq \theta^{\text{out}}_i(e): B^{(n,K)}_s=e_-
   \} \quad i\in\N
   \end{aligned}
   \end{equation}
   and
   \begin{equation}\label{eq:defofupwardsedgelocaltime}
   l^{(n,K,\rightarrow)}_t(e):= \max\{i\in\N: \theta^{\text{out}}_i(e)\leq t \}.
   \end{equation}
   
Let
\begin{equation}\label{eq:defofsigma}
\sigma^{n,K}_i(e):=\theta_{i+1}^{\text{out}}(e)-\theta^{\text{out}}_{i}(e).
\end{equation}
That is, $\sigma^{n,K}_i(e)$ is the time between the $i$-th and the $(i+1)$-th rightcrossing of $e$ by $B^{(n,K)}$. Observe that $(\sigma^{n,K}_i(e))_{i\geq1}$ is an independent and identically distributed sequence of random variables.

Analogously to $l^{(n,K,\rightarrow)}_t(e)$, we define the number of leftward crossings of an edge, denoted by $l^{(n,K,\leftarrow)}_t(e)$.  
Let $P^{G_n}$ denote the probability $\mathbf{P}^{(K)}$ conditioned on a fixed realization of the environment $G_n$. Let $E^{G_n}$ denote the expectation with respect to $P^{G_n}$. 

\begin{lemma}\label{l:pretailoflocaltime}
For each $K\in\N$, there exists $c>0$ such that $\mathbf{P}^{(K)}$-almost surely, for all $n\in\N$, $e\in E(\T^{(n,K)})$
\[P^{G_n}[\sigma^{n,K}_1(e)\geq \epsilon]\geq \frac{c\reff^{G_n}(e)}{n^{1/2}}.\]
\end{lemma}

\begin{proof}[Proof of Lemma \ref{l:pretailoflocaltime}]
By Lemma \ref{lem:graphspatialresistancetreeconvergence} we know that $(\T^{(n,K)},\rho^{-1}d^{(n,K)}_{\text{res}})$ converges to $(\frak{T}^{(K)},\sigma_dd_{\frak{T}^{(K)}})$ $\mathbf{P}^{(K)}$-almost surely. Since $\frak{T}^{(K)}$ is composed of a finite number of linear segments, let us denote $2r$ the length of the shortest edge of $\frak{T}^{(K)}$. Therefore, for $n$ large enough, each edge $e=(e_-,e_+)\in E(\T^{(n,K)})$ is contained in a linear segment $[\text{root}^\ast,l]\subset \T^{(n,K)}$ of $d_{\text{res}}^{(n,K)}$-length greater or equal than $2r$, where $l$ is a leaf\footnote{A leaf is a vertex of $V(\T^{n,K})$ which has degree $1$.} of $V(\T^{(n,K)})$. Therefore, either $d_{\text{res}}^{(n,K)}(\text{root}^\ast,e_+)\geq r$ or $d_{\text{res}}^{(n,K)}(e_-,l)\geq r$, for some leaf $l$ of $V(\T^{(n,K)})$. 

We will do the proof only for the first case, the second being analogous. By the strong Markov property of $B^{(n,K)}$
\begin{equation}\label{eq:lastlastlast}
\begin{aligned}
&P^{G_n}[\sigma_1^{n,K}(e)\geq 1]\\\geq&  P_{e_-}^{G_n}[B^{(n,K)}\text{ hits }\text{root}^\ast\text{ before }e_+]\times P_{\text{root}^\ast}^{G_n}[B^{(n,K)}\text{ hits }e_+\text{ after time }1].
\end{aligned}
\end{equation}
Also
\begin{equation}\label{eq:hittingprobabilitiesfortildeB}
P_{e_-}^{G_n}[B^{(n,K)}\text{ hits }\text{root}^\ast\text{ before }e_+]=\frac{d^{(n,K)}_{\text{res}}(e_-,e_+)}{d^{(n,K)}_{\text{res}}(e_+,\text{root}^\ast)} \geq \reff^{G_n}(e)n^{-1/2}(2\rho\sigma_d\Lambda^{(K)})^{-1},
\end{equation}
where $\Lambda^{(K)}$ denotes the total $d_{\frak{T}}$-length of $\frak{T}^{(K)}$ and
 we have used that $d^{(n,K)}_{\text{res}}(e_+,\text{root}^\ast)\leq 2\rho\sigma_d\Lambda^{(K)}$ for $n$ large enough, which follows from the convergence of $(\T^{(n,K)},\rho^{-1}d^{(n,K)}_{\text{res}})$ to $(\frak{T}^{(K)},\sigma_dd_{\frak{T}})$ (see Lemma \ref{lem:graphspatialresistancetreeconvergence}).
Moreover, since we have $d^{(n,K)}_{\text{res}}(\text{root}^\ast,e_+)\geq r$, there exists $c\geq0$ independent of $n$ such that 
\[P_{\text{root}^\ast}^{G_n}\left[B^{(n,K)}\text{ hits }e_+\text{ after time }1\right]\geq c.\]
The lemma follows from the display above, \eqref{eq:hittingprobabilitiesfortildeB} and \eqref{eq:lastlastlast}. 
\end{proof}

\subsubsection{Tail estimate on local times}

Our tail estimate on local time is the following

\begin{lemma}\label{l:tailoflocaltime}
For any $T>0$ and $K\in\N$, $\mathbf{P}^{(K)}$-almost surely, there exists constants $c_1,c_2>0$ such that
\[\sup_{e\in E^\ast(\T^{(n,K)}),n\in\N}P^{G_n}\left[n^{-1/2}\reff^{G_n}(e)l^{(n,K,\rightarrow)}_T(e)\geq m \right]\leq c_1 \exp{(-c_2m)}\]
and
\[\sup_{e\in E^\ast(\T^{(n,K)}),n\in\N}P^{G_n}\left[n^{-1/2}\reff^{G_n}(e)l^{(n,K,\leftarrow)}_T(e)\geq m \right]\leq c_1 \exp{(-c_2m)}.\]
\end{lemma}

\begin{proof}
We will only do the proof of the first display, the second being analogous.
The display \eqref{eq:defofupwardsedgelocaltime} allows us to relate $l^{(n,K,\rightarrow)}_T(e)$ and $\theta_i^{\text{out}}$ through an inversion argument. We can then use \eqref{eq:defofsigma} to see that 
\begin{align*}
&P^{G_n}\left[n^{-1/2}\reff^{G_n}(e)l^{(n,K,\rightarrow)}_T(e)\geq  m \right]\leq P^{G_n}\left[\sum_{i=1}^{\lfloor mn^{1/2}\reff^{G_n}(e)^{-1}\rfloor}\sigma^{n,K}_i(e)\leq T\right]\\
\leq &P^{G_n} \left [\sum_{i=1}^{\lfloor mn^{1/2}\reff^{G_n}(e)^{-1}\rfloor}1_{\{\sigma^{n,K}_i(e)\geq 1  \}}\leq T\right]\\
 \leq &P^{G_n}[\text{Bin}(\lfloor \reff^{G_n}(e)^{-1}mn^{1/2} \rfloor, C_1^{-1}n^{-1/2}\reff^{G_n}(e))\leq T\epsilon^{-1}],
\end{align*}
where $\text{Bin}(n,p)$ denotes a binomial random variable of parameters $n,p$ and for the last inequality we have used Lemma \ref{l:pretailoflocaltime}.
Finally, we can use  the convergence of binomial to Poisson to upper-bound the display above by
\[
C_2 P^{G_n}[\text{Poiss}(C_1^{-1}m)  \leq T  ]= C_2\exp(-C_1^{-1}m)\sum_{i=1}^{\lfloor T \rfloor} \frac{(C_1^{-1}m)^{i}}{i!},
\]
where $\text{Poiss}(\lambda)$ denotes a Poisson random variable of parameter $\lambda$ and $C_2$ is a large constant.

The right hand side in the display above, seen as a function of $m$, is the product between exponential and a polynomial. Therefore, there exists $c_1,c_2$ such that the display above is bounded by $c_1\exp(-c_2 m)$.
\end{proof}

\subsubsection{Tail estimate of local times for branching  edges}
The next result extends Lemma \ref{l:tailoflocaltime} to edges in $E(G^{(n,K)})\backslash E^\ast(\T^{(n,K)})$, i.e., edges of $G^{(n,K)}$ which are removed when applying the star triangle transformation. This result will be used in Section \ref{s:AnKislinear} to neglect the time that $X^{G_n}$ spends on bubbles corresponding to branchings of $\T^{(n,K)}$.

 Let $e=(x,y)\in E(G^{(n,K)})\backslash E^\ast(\T^{(n,K)})$. We want to define a local time for $e$ which will correspond to the number of times that $B^{(n,K)}$ (or equivalently $J^{(n,K)}$) moves from one end to the other of $e$. Since we want to define directed local times ($l^{(n,K,\rightarrow)}_t(e)$ and $l^{(n,K,\leftarrow)}_t(e)$), we will need to decide if $l^{n,K,\rightarrow}_t(e)$ will count transitions from $x$ to $y$ or viceversa. If $e$ is such that $x,y$ are comparable in the genealogical order, we can assume without loss of generality that $x\prec y$, in this case we relabel $e_-=x$ and $e_+=y$ ans define $l^{(n,K,\rightarrow)}_t(e)$ as in \eqref{eq:defofupwardsedgelocaltime}. Analogously, we define $l_t^{(n,K,\leftarrow)}(e)$. It might be that $x,y$ are not comparable in the genealogical partial order on $V^\ast(\T^{(n,K)})$, in this case we might use the lexicographical total order and assume without loss of generality that $x<y$ and relabel $e_-=x$ and $e_+=y$. Again, we define $l^{(n,K,\rightarrow)}_t(e)$ as in \eqref{eq:defofupwardsedgelocaltime}. Analogously, we define $l^{(n,K\leftarrow)}_t(e)$.
 
For $e=(e_-,e_+)\in E(G^{(n,K)})\setminus E(\T^{(n,K)})$, define $\reff^{G_n}(e)$ as the effective resistance between $e_-$ and $e_+$ in $G_n$. We claim that:
\begin{lemma}\label{l:tailoflocaltimetrifurcation}
For any $t>0$ and $K\in\N$, $\mathbf{P}^{(K)}$-almost surely, there exists constants $c_1,c_2>0$ such that
\[\sup_{e\in E(G^{(n,K)})\backslash E^\ast(\T^{(n,K)}),n\in\N}P^{G_n}\left[n^{-1/2}\reff^{G_n}(e)l^{(n,K,\rightarrow)}_t(e)\geq m \right]\leq c_1 \exp{(-c_2m)}\]
and
\[\sup_{e\in E(G^{(n,K)})\backslash E^\ast(\T^{(n,K)}),n\in\N}P^{G_n}\left[n^{-1/2}\reff^{G_n}(e)l^{(n,K,\leftarrow)}_t(e)\geq m \right]\leq c_1 \exp{(-c_2m)}.\]
\end{lemma}
\begin{proof}[Proof of Lemma \ref{l:tailoflocaltimetrifurcation}]
It is clear that $e=(e_-,e_+)\in E(G^{(n,K)})\setminus E^\ast(\T^{(n,K)})$ corresponds to two adjacent  edges $e_1,e_2\in E(\T^{(n,K)}) \setminus E^\ast(\T^{(n,K)})$.
By virtue of the star-triangle transformation,
\[d^{(n,K)}_{\text{res}}(e_1)+d^{(n,K)}_{\text{res}}(e_2)=n^{-1/2}\reff^{G_n}(e).\] 
Recall that $l^{(n,K,\rightarrow)}_t(e)$ (resp. $l^{(n,K,\leftarrow)}_t(e)$) is the number of oriented crossings of the segment corresponding to $e_1,e_2$ by $B^{(n,K)}$. Therefore, using the display above we can repeat the arguments leading to display \eqref{eq:hittingprobabilitiesfortildeB}. Therefore, we can mimic the proof of Lemmas \ref{l:pretailoflocaltime} and \ref{l:tailoflocaltime} in a straightforward way. This finishes the proof
\end{proof}
\subsection{Number of excursion between edges} \label{sect:number_excursion}

 Fix two edges $e=(e_-,e_+),e'=(e'_-,e'_+)\in E(\T^{(n,K)})$.  Let us define the resistance-distance between between those edges by
  \[d^{(n,K)}_{\text{res}}(e,e')=\min\{d_{\text{res}}^{(n,K)}(e_-,e'_-),d^{(n,K)}_{\text{res}}(e_-,e'_+),d^{(n,K)}_{\text{res}}(e_+,e'_-),d^{(n,K)}_{\text{res}}(e_+,e'_+)\}.\]
 
 We also introduce $N_i(e,e')$ the number of rightward crossings of $e'$ (from $e_-' $ to $e_+'$) by $B^{(n,K)}$ between the $i$-th and the $(i+1)$-th rightward crossing of $e$ (from $e_- $ to $e_+$) by $B^{(n,K)}$. We set
\begin{equation}\label{eq:defeta}
\eta_i(e,e')=N_i(e,e')\reff^{G_n}(e')-\reff^{G_n}(e).
\end{equation}

We will control the number of crossings between two edges using this random variable. The main result of this section is

\begin{lemma}\label{lem:premaximalinequality}
Let $t>0$. There exist constants $C_1,C_2$ independent of $n$ such that
\begin{equation}\label{eq:nonuniform1}
\mathbf{P}^{(K)}\left[n^{-1/2}\abs{\sum_{i=1}^{l^{(n,K,\rightarrow)}_t(e)}\eta_i(e,e') }\geq\epsilon\right]\leq C_1\epsilon^{-4}(d^{(n,K)}_{\text{res}}(e,e'))^2,
\end{equation}
and
\begin{equation}\label{eq:nonuniform2}
\mathbf{P}^{(K)}\left[ n^{-1/2} \abs{\eta_{l^{(n,K,\rightarrow)}_t(e)}(e,e')}\geq\epsilon\right]\leq C_2\epsilon^{-4}(d^{(n,K)}_{\text{res}}(e,e'))^2.
\end{equation}
\end{lemma}

In order to prove this lemma, we will need to obtain a moment estimate on $\eta_i$ and an estimate allowing us to treat the random variables $\eta_i(e,e')$ and $l_t^{(n,K,\rightarrow)}(e)$ as if they were independent. The latter point will be proved through an optional stopping time argument.

\subsubsection{Moment estimate}

Fix $i\in \N$ and $e=(e_-,e_+),e'=(e'_-,e'_+)$ two edges of $\T^{(n,K)}$, where $e$ lies in the path from the root to $e'$.
The random variables defined on \eqref{eq:defeta} verify the following property. 
\begin{lemma}\label{l:boundoneta}
Fix $i\in \N$. We have
\[
E^{G_n}[\eta_i(e,e')]=0,
\]
and for any $r\in\N$,
\begin{equation}\label{eq:piccolo}
E^{G_n}[\abs{\eta_i(e,e')}^r]\leq C(r) \reff^{G_n}(e) \reff^{G_n}(e_-,e_+')^{r-1},
\end{equation}
where $C(r)$ is a constant depending only on $r$.
\end{lemma}

\begin{proof}
Since any path from $e_+'$ to $e_-$ has to go through $e_-'$, we can see that for $k\geq 1$ we have
\begin{equation}\label{eq:probabilitiesofN}
P^{G_n}[N_i(e,e')\geq k]=P^{G_n}_{e_+}[T_{e_+'}<T_{e_-}]P^{G_n}_{e_-'}[T_{e_-}>T_{e_+'}]^{k-1},
\end{equation}
where $P^{G_n}_x$ denotes the law associated to $P^{G_n}$ conditioned of $J^{(n,K)}_0=x$ and thus
\begin{equation}\label{etapart1}
E^{G_n}[N_i(e,e')]=\frac{P^{G_n}_{e_+}[T_{e_+'}<T_{e_-}]}{P^{G_n}_{e_-'}[T_{e_-}<T_{e_+'}]}.
\end{equation}

Using the terminology of electrical networks (see~\cite{lyons2005probability}), we denote $u(\cdot)$ the potential associated to the resistances of $(\T^{(n,K)},d^{(n,K)}_{\text{res}})$ with boundary conditions given by $u(e_-)=1$ and $u(e_+')=0$. It is then a standard fact that
\[
P_{e_+}^{G_n}[T_{e_+'}<T_{e_-}]=u(e_-)-u(e_+) \text{ and } P_{e_-'}^{G_n}[T_{e_-}<T_{e_+'}]=u(e_-')-u(e_+').
\]

Since any path from $e_-$ to $e_+'$ has to go through $e_+$ and $e_-'$, we know that $i(e)=i(e')=C_{\text{eff}}(e_-, e_+')$. This leads to
\begin{equation}\label{etapart2}
\reff^{G_n}(e)^{-1}P_{e_+}^{G_n}[T_{e_+'}<T_{e_-}]=i(e)=C_{\text{eff}}(e_-, e_+')=i(e')=\reff^{G_n}(e')^{-1}P^{G_n}_{e_-'}[T_{e_-}<T_{e_+'}],
\end{equation}
and, using~\eqref{etapart1},
\begin{equation}\label{eq:expectationofN}
E^{G_n}[N_i(e,e')]=\frac{\reff^{G_n}(e)}{\reff^{G_n}(e')},
\end{equation}
which proves the first part of the lemma.

Now we turn our attention to the proof of the bound on the $r$-th moment of $\eta_i(e,e')$. Let us start by noticing that 
\begin{equation}\label{wwww}
\begin{aligned}
E^{G_n}[\abs{\eta_i(e,e')}^r]&=E^{G_n}[\abs{\eta_i(e,e')}^r\1{\eta_i(e,e') <0}]\\&+E^{G_n}[\abs{\eta_i(e,e')}^r\1{\eta_i(e,e') \geq 0}] \\ 
         &\leq \reff^{G_n}(e)^r +E^{G_n}[\eta_i(e,e')^r\1{\eta_i(e,e') \geq 0}] \\ 
         & \leq  \reff^{G_n}(e) \reff^{G_n}(e_-,e_+')^{r-1}+E^{G_n}[\eta_i(e,e')^r\1{\eta_i(e,e') \geq 0}] ,
\end{aligned}
\end{equation}
where we used that $\eta_i(e,e')\geq -\reff^{G_n}(e)$ and $\reff^{G_n}(e)\leq \reff^{G_n}(e_-,e_+')$. Let us simply observe that $\eta_i(e,e')\leq \reff^{G_n}(e')N_i(e,e')$. For an integer-valued random variable, we denote $E[(X)_r]:=E[X(X-1)\cdots (X-(r-1))]$ the $r$-th factorial moment of $X$. An elementary computation shows that
\[
E^{G_n}[(N_i(e,e'))_r]=r! P_{e_+}^{G_n}[T_{e_+'}<T_{e_-}] \frac{(1-P_{e_-'}^{G_n}[T_{e_-}<T_{e_+'}])^{r-1}}{P_{e_-'}^{G_n}[T_{e_-}<T_{e_+'}]^r},
\]
from the first $r$-th factorial moments, it is clear that we can recover the $r$-th moment. In particular, we can obtain that
\[
E^{G_n}[N_i(e,e')^r] \leq C(r) \frac{P_{e_+}^{G_n}[T_{e_+'}<T_{e_-}] }{P_{e_-'}^{G_n}[T_{e_-}<T_{e_+'}]^r}.
\]

Now, we can use that $\eta_i(e,e') \reff^{G_n}(e')^{-1}\1{\eta_i(e,e')\geq 0} \leq N_i(e,e')$ to see that
\[
E^{G_n}[(\eta_i(e,e') \reff^{G_n}(e')^{-1})^r\1{\eta_i(e,e')\geq 0}]\leq C(r) \frac{P_{e_+}^{G_n}[T_{e_+'}<T_{e_-}] }{P_{e_-'}^{G_n}[T_{e_-}<T_{e_+'}]^r}.
\]

Recalling~\eqref{etapart2}, we see that
\[
E^{G_n}[(\eta_i(e,e') \reff^{G_n}(e')^{-1})^r\1{\eta_i(e,e')\geq 0}]\leq C(r) \frac{\reff^{G_n}(e)}{\reff^{G_n}(e')^r} \reff^{G_n}(e_-,e_+')^{r-1},
\]
which finally leads to
\[
E^{G_n}[\eta_i(e,e')^r\1{\eta_i(e,e')\geq 0}]\leq C(r) \reff^{G_n}(e) \reff^{G_n}(e_-,e_+')^{r-1}.
\]

This concludes the proof of the lemma along with~\eqref{wwww}.
\end{proof}

\subsubsection{Optional stopping time theorem}
We want to apply Lemma \ref{lem:martingaleinequality} in the appendix to $(\eta_i(e,e'))_{i\in\N}$ and the stopping time $l_t^{(n,K,\rightarrow)}(e)$. Note that the estimate in Lemma~\ref{lem:martingaleinequality} would be similar to that obtained if $(\eta_i(e,e'))_{i\in\N}$ and $l_t^{(n,K,\rightarrow)}(e)$ were independent. The next result is to check the hypothesis of Lemma \ref{lem:martingaleinequality}.
Let $\Psi(z):=E[\exp(z\eta_1(e,e'))]$ and
\[
M_m(z):=\Psi(z)^{-m}\exp\left(z\sum_{i=1}^m \eta_i(e,e') \right).
\]
\begin{lemma}\label{l:mgfatstoppingtime}
For all $e,e'\in E(\T^{(n,K)})$ and  $t\geq0$ we have
\[
E^{G_n}\left[M_{l^{(n,K,\rightarrow)}_t(e)}(z)\right]=1,
\]
for all $z$ in a neighborhood of $0$, $\mathbf{P}^{(K)}$-almost surely.
\end{lemma}

\begin{proof}[Proof of Lemma \ref{l:mgfatstoppingtime}]
Observe that $(M_m(z))_{m\in\N}$ is a martingale adapted to $(\mathcal{F}_m)_{m\in\N}$,
 where $\mathcal{F}_m$ is the $\sigma$-algebra generated by $B^{(n,K)}$ up to time $\theta^{\text{out}}_{m+1}$. Indeed, since $(\eta_i(e,e'))_{i\in\N}$ is i.i.d.~ and adapted to $(\mathcal{F}_i)_{i\in\N}$, we have that
\[
E^{G_n}[M_{m+1}(z)|\mathcal{F}_m]=M_m (z).
\]

Therefore, to show that $(M_m(z))_{m\in\N}$ is a martingale it suffices to show that $E^{G_n}[M_m(z)]<\infty$ for all $m\in\N$. 
 We have 
\begin{equation}\label{eq:definitionofpsi}
\Psi(z)=\exp(-z\reff^{G_n}(e)) \sum_{i=0}^{\infty} P^{G_n}[N_1(e,e')=i]\exp(zi\reff^{G_n}(e')).
\end{equation}
Using \eqref{eq:probabilitiesofN} we get that the display above equals
\begin{align*}
&\exp(-z\reff^{G_n}(e))\times
\Big[ P_{e_+}^{G_n}[T_{e_-}<T_{e'_+}]\\+&P^{G_n}_{e_+}[T_{e_-}>T_{e'_+}]\frac{P^{G_n}_{e'_-}[T_{e_-}<T_{e'_+}]}{P^{G_n}_{e'_-}[T_{e_-}>T_{e'_+}]}\sum_{k=1}^{\infty} (\exp(z\reff^{G_n}(e'))P^{G_n}_{e'_-}[T_{e_-}>T_{e'_+}])^k\Big] \end{align*}
For $z$ small enough the geometric series converges and moreover
\begin{equation}\label{eq:psitoone}
\Psi(z)\to1 \quad \text{as }z\to 0^+. 
\end{equation} 
Finally, since $\Psi(z)<\infty$ for $z$ small enough, we get that 
\[E^{G_n}[M_k(z)]=\Psi(z)^{-k}E^{G_n}\left[\exp\left(z\sum_{i=1}^k \eta_i(e,e')\right)\right]=\frac{\Psi(z)^k}{\Psi(z)^k}=1\] for $z$ small enough, where the last equality follows from the independence of the $(\eta_i(e,e'))_{i\in\N}$.
Therefore, we have showed that $(M_m(z))_{m\in \N}$ is a martingale for $z$ small enough.

For any $t\geq0$, $l^{(n,K,\rightarrow)}_t(e)$ is a stopping time relative to $(\mathcal{F}_i)_{i\in\N}$. We will apply the optional stopping theorem to $M_k(z)$ at $l^{(n,K,\rightarrow)}_t(z)$. 
It is not hard to see that the optional stopping theorem (see e.g., \cite[Theorem 2.2, \S 7]{bookdoob1953stochastic}) holds provided
\begin{enumerate}
\item $l^{(n,K,\rightarrow)}_t(e) < \infty$, $P^{G_n}$-almost surely,
\item $E^{G_n}[M_{l^{(n,K,\rightarrow)}_t(e)}(z)]<\infty$,
\item $E^{G_n}[M_m(z)1_{\{l^{(n,K,\rightarrow)}_t(e)>m\}}]\to0$ as $m\to\infty$.
\end{enumerate} 

Condition $(1)$ is clear.
We turn our attention to the proof of displays $(2)$ and $(3)$. 
Using Cauchy-Schwarz inequality,
\[E^{G_n}\left[M_m(z)1_{\{l^{(n,K,\rightarrow)}_t(e)\geq m\}}\right] \leq E^{G_n}[M_m(z)^2]^{1/2}P^{G_n}[l^{(n,K,\rightarrow)}_t(e)\geq m]^{1/2}.\]
  The right hand side of the display above equals
  \[ \left[\frac{\Psi(2z)^{1/2}}{\Psi(z)}\right]^m P^{G_n}[l^{(n,K,\rightarrow)}_t(e)\geq m]^{1/2}.\]
Moreover, Lemma \ref{l:tailoflocaltime} implies that the display above is upper bounded by
\[ \left[\frac{\Psi(2z)^{1/2}}{\Psi(z)}\right]^m C\exp(-cm), \]
for some positive constants $C,c$. Moreover, display \eqref{eq:psitoone} implies that $\frac{\Psi(2z)^{1/2}}{\Psi(z)}\to1$ as $z\to0^+$, therefore, for $z$ small enough, we have
\begin{equation}\label{eq:cauchy}
E^{G_n}\left[M_m(z)1_{\{l^{(n,K,\rightarrow)}_t(e)\geq m\}}\right] \leq C\exp(-c'm),
\end{equation}
 for some positive constants $C,c'$. 
Condition $(3)$ follows immediately from the display above.

It just remains to check condition $(2)$. Observe that
\[E^{G_n}\left[M_{l_t^{(n,K,\rightarrow)}(e)}(z)\right]=\sum_{m=0}^{\infty}E^{G_n}[M_m(z) 1_{\{l_t^{(n,K,\rightarrow)}(e)=m\}}].\]
Using display \eqref{eq:cauchy} we get that the display above is bounded above by
\[\sum_{m=0}^{\infty}C \exp(c'm)<\infty.\]
This shows condition $(2)$.
Therefore, the optional stopping theorem holds for the martingale $(M_k(z))_{k\in\N}$ at the stopping time $l^{(n,K,\rightarrow)}_t(e)$ when $z$ is sufficiently small. 
Hence 
\[E^{G_n}\left[ M_{l_t^{(n,K,\rightarrow)}(e)}(z)\right]= E^{G_n}\left[ M_0(z) \right]=1. \]
This yields the claim of the lemma.
\end{proof}

\subsubsection{Proof of Lemma \ref{lem:premaximalinequality}}

 We will do the proof of Lemma \ref{lem:premaximalinequality} using Lemma 
 \ref{l:boundoneta}, Lemma~\ref{l:tailoflocaltime} and Lemma~\ref{l:mgfatstoppingtime}. 

\begin{proof}[Proof of Lemma \ref{lem:premaximalinequality}]
 We first do the proof of \eqref{eq:nonuniform1} and then we will prove \eqref{eq:nonuniform2}.
 
{\it First step:  Proof of display \eqref{eq:nonuniform1}.}

\noindent We have
\begin{equation}\label{eq:fourthmomentdominates}
P^{G_n}\left[ n^{-1/2}\left|\sum_{i=1}^{l_t^{(n,K,\rightarrow)}(e)}\eta_i(e,e')\right| \geq \epsilon/3 \right] \leq \epsilon^{-4} 3^4 n^{-2}E\left[\left(\sum_{i=1}^{l_t^{(n,K,\rightarrow)}(e)}\eta_i(e,e')\right)^4 \right]. 
\end{equation}
To control the expectation in the right hand side above, we follow an argument from \cite[Lemma 1.6]{borodin1982asymptotic} which consist in the use of Lemma \ref{lem:martingaleinequality} in the appendix.

 Lemma \ref{l:mgfatstoppingtime} and Lemma~\ref{lem:martingaleinequality} give that there exists $C$ such that
\begin{equation}\label{eq:fourthmoment}
\begin{aligned}
&E^{G_n}\left[\left(\sum_{i=1}^{l_t^{(n,K,\rightarrow)}(e)}\eta_i(e,e')\right)^4 \right] \\
\leq &C\left( E^{G_n}[\eta_1(e,e')^4]E^{G_n}[l^{(n,K,\rightarrow)}_t(e)] + E^{G_n}[\eta_1(e,e')^2]^2E^{G_n}[l^{(n,K,\rightarrow)}_t(e)^2] \right)
\end{aligned}
\end{equation}
 Note that Lemma \ref{l:tailoflocaltime} implies that there exists constants $C_1,C_2$ such that 
 \begin{equation}\label{eq:localtimemoments}
 E^{G_n}[l^{(n,K,\rightarrow)}_t(e)]\leq C_1 \frac{n^{1/2}}{\reff^{G_n}(e)}\quad \text{and} \quad E^{G_n}[l^{(n,K,\rightarrow)}_t(e)^2]\leq C_2\frac{n}{ \reff^{G_n}(e)^{2}}.
 \end{equation}
 
Using the bounds above together with the bounds on $E^{G_n}[\eta_i(e,e')^k]$ of Lemma \ref{l:boundoneta} in display \eqref{eq:fourthmoment}, all the instances of $\reff^{G_n}(e)$ cancel and, recalling that $d^{(n,K)}_{\text{res}}(x,y):=n^{-1/2}\reff^{G_n}(x,y)$ we find that there exists $C$ such that 
\[
 E^{G_n}\left[\left(\sum_{i=1}^{l^{(n,K,\rightarrow)}_t(e)}\eta_i(e,e')\right)^4\right] 
 \leq  Cn^2(d^{(n,K)}_{\text{res}}(e_-,e'_+)^2+d^{(n,K)}_{\text{res}}(e_-,e'_+)^{3}).
 \]
 
Therefore, under our assumption that $d^{(n,K)}_{\text{res}}(e_-,e'_+) \leq 1$ (and therefore, $d^{(n,K)}_{\text{res}}(e_-,e'_+)^3<d^{(n,K)}_{\text{res}}(e_-,e'_+)^2$)
 \begin{equation}\label{eq:fourthmomentisdominated}
 E^{G_n}\left[\left(\sum_{i=1}^{l^{(n,K,\rightarrow)}_t(e)}\eta_i(e,e')\right)^4\right]\leq Cn^2 d^{(n,K)}_{\text{res}}(e_-,e'_+)^2,
 \end{equation}
 for some constant $C$,
which, together with \eqref{eq:fourthmomentdominates} yield \eqref{eq:nonuniform1}.

 {\it Second step:  Proof of display \eqref{eq:nonuniform2}.}
 
\noindent 
We have \begin{equation}\label{eq:controlofthemaximum}
P^{G_n}\left[n^{-1/2}\left|\eta_{l^{(n,K,\rightarrow)}_t(e)}(e,e')\right|\geq \epsilon/3  \right]\leq E^{G_n}\left[ \sum_{i=1}^{l_t^{(n,K,\rightarrow)}(e)} 1_{\{|\eta_i(e,e') |\geq n^{1/2} \epsilon/3 \}}. \right]
\end{equation}
 By Wald's identity, the right hand side of \eqref{eq:controlofthemaximum} equals
 \begin{align*}
& E^{G_n}[l^{(n,K,\rightarrow)}_t(e)]P^{G_n}[|\eta_1(e,e')|\geq n^{1/2}\epsilon/3]\\
 \leq &\epsilon^{-4}3^4 n^{-2} E^{G_n}[l^{(n,K,\rightarrow)}_t(e) ] E^{G_n}[\eta_1(e,e')^4].
 \end{align*}
Since by assumption we have $d^{(n,K)}_{\text{res}}(e_-,e'_+)<1$, we get  $d^{(n,K)}_{\text{res}}(e_-,e'_+)^3\leq d^{(n,K)}_{\text{res}}(e_-,e'_+)^2$. Therefore, from the display above~\ref{eq:localtimemoments} and \eqref{eq:piccolo}, we get that \eqref{eq:nonuniform2} holds.
  This finishes the proof of the lemma.
\end{proof}

\subsection{Regularity of  local times}\label{sect:regularity_local_time}

Fix two edges $e=(e_-,e_+),e'=(e'_-,e'_+)\in E(\T^{(n,K)})$ and recall that
 \[d^{(n,K)}_{\text{res}}(e,e')=\min\{d_{\text{res}}^{(n,K)}(e_-,e'_-),d^{(n,K)}_{\text{res}}(e_-,e'_+),d^{(n,K)}_{\text{res}}(e_+,e'_-),d^{(n,K)}_{\text{res}}(e_+,e'_+)\}.\]
 
The main result of this section is the following
\begin{lemma}\label{lem:modulusofcontinuity} Let  $t>0$, $K\in\N$ and $\epsilon>0$. We have
\[\lim_{\delta\to0}\limsup_{n\in\N}\mathbf{P}^{(K)}\left[\sup_{\stackrel{e,e'\in E(\T^{(n,K)})}{d^{(n,K)}_{\text{res}}(e,e')\leq\delta}} n^{-1/2}\abs{R^{G_n}_{\text{eff}}(e)l^{(n,K)}_t(e)-R^{G_n}_{\text{eff}}(e')l^{(n,K)}_t(e')}\geq\epsilon\right]=0.\]
\end{lemma}

 Recalling the definition of $l^{(n,K,\text{vert})}_t(x)$ in \eqref{eq:defoflvert} (and the definition of $e^+(x),e^+(x)$ just above \eqref{eq:sopracontrariar}), we have that
 \begin{equation}\label{eq:vertupdown}
l^{(n,K,\text{vert})}_{m(t)}(x)= l^{(n,K,\rightarrow)}_{t}(e^-(x))+l^{(n,K,\leftarrow)}_{t}(e^+(x))
 \end{equation}
 The proof of Lemma \ref{lem:modulusofcontinuity} relies on the following result.
\begin{lemma}\label{lem:nonuniform}
Let  $t>0$, $K\in\N$ and $\epsilon>0$. We have
\[ 
\lim_{\delta\to0}\limsup_{n\in\N} \mathbf{P}^{(K)}\left[\sup_{\stackrel{e,e'\in E^\ast(\T^{(n,K)})}{d^{(n,K)}_{\text{res}}(e,e')\leq\delta}} n^{-1/2}\abs{R^{G_n}_{\text{eff}}(e)l^{(n,K,\rightarrow)}_t(e)-R^{G_n}_{\text{eff}}(e')l^{(n,K,\rightarrow)}_t(e')}\geq\epsilon \right]=0.
\]
\end{lemma}
Next, we assume Lemma \ref{lem:nonuniform} and deduce Lemma \ref{lem:modulusofcontinuity} from it. 
\begin{proof}[Proof of Lemma \ref{lem:modulusofcontinuity}]
We need to get an estimate as that of Lemma \ref{lem:nonuniform} but with $l^{(n,K)}_t(e)$ instead of $l^{(n,K,\rightarrow)}_t(e)$.
Since for any $e\in E(\T^{(n,K)})$ and $t\geq0$, either $l^{(n,K)}_t(e)=2l^{(n,K,\rightarrow)}_t(e)$ or $l^{(n,K)}_t(e)=2l^{(n,K,\rightarrow)}_t(e)-1$, we have
\begin{align*}
&\abs{\reff^{G_n}(e)l^{(n,K)}_t(e)- \reff^{G_n}(e')l^{(n,K)}_t(e')}\\
 \leq &2 \abs{\reff^{G_n}(e) l^{(n,K,\rightarrow)}_t(e)-\reff^{G_n}(e')l^{(n,K,\rightarrow)}(e')} + \max\{\reff^{G_n}(e),\reff^{G_n}(e')\}.
\end{align*}
Therefore, the lemma follows from Lemma \ref{lem:nonuniform} and Lemma \ref{lem_no_macro_res}.
\end{proof}

The rest of the section will be devoted to the proof of Lemma \ref{lem:nonuniform}. This proof consists of two steps
\begin{enumerate}
\item we link the variations in local times to the random variables $\eta_i$ introduced at~\eqref{eq:defeta}, (see Lemma~\ref{lem:etatrick})
\item we apply a maximal inequality argument, which is inspired from~~\cite{billingsley2009convergence}.
\end{enumerate}

\subsubsection{Linking variations in local times to $\eta_i$}

Let $e=(e_-,e_+),e'=(e'_-,e'_+)\in E(\T^{(n,K)})$ be two edges such that $e$ is in the path from the root to $e'$.
 We start by relating $R_{\text{eff}}^{G_n}(e)l^{(n,K,\rightarrow)}_t(e)-R^{G_n}_{\text{eff}}(e')l^{(n,K,\rightarrow)}_t(e')$ to the behavior of the partial sums of $\eta_i(e,e')$, where $\eta_i(e,e')$ is defined as in \eqref{eq:defeta}.
 \begin{lemma}\label{lem:etatrick}
 Let $e^1\prec e^2\prec \cdots e^k$ be an ordered sequence of edges in $E(\T^{(n,K)})$, where $\prec$ denotes genealogical order. 
 We have that
\begin{equation}
 \begin{aligned}
 &\abs{\reff^{G_n}(e^1)l_t^{(n,K,\rightarrow)}(e^1)-\reff^{G_n}(e^k)l_t^{(n,K,\rightarrow)}(e^k)}\\\leq &\max_{i=1,\dots,k}\{\reff^{G_n}(e^i)\}+ \sum_{j=1}^{k-1}\abs{\sum_{i=1}^{l_t^{(n,K,\rightarrow)}(e^j)}\eta_i(e^j,e^{j+1})} + \sum_{j=1}^{k-1}\abs{\eta_{l^{(n,K,\rightarrow)}_t(e^j)}(e^j,e^{j+1})}.
 \end{aligned}
 \end{equation}
\end{lemma}
\begin{proof}[Proof of Lemma \ref{lem:etatrick}]
We will split the proof in two parts. First we will get a display similar to the one in the lemma in the case $k=2$ (two-edge comparison). After that, we will use the two-edge comparison to get the multiple edge comparison in the Lemma.\\
\emph{Two-edge comparison:}\\
 Let $e,e'\in E(\T^{(n,K)})$ be a pair of edges satisfying $e\prec e'$.
First of all, observe that, by definition
\[
0\leq\reff^{G_n}(e)l_t^{(n,K,\rightarrow)}(e)-\reff(e')l_t^{(n,K,\rightarrow)}(e')+\sum_{i=1}^{l_t^{(n,K,\rightarrow)}(e)}\eta_i(e,e').
\]
Therefore
 \begin{equation}\label{eq:estosdiasterribles}
 \begin{aligned}
 &-\eta_{l^{(n,K,\rightarrow)}_t(e)}(e,e')\\\leq&\reff^{G_n}(e)l_t^{(n,K,\rightarrow)}(e)-\reff(e')l_t^{(n,K,\rightarrow)}(e')+\sum_{i=1}^{l_t^{(n,K,\rightarrow)}(e)-1}\eta_i(e,e').
 \end{aligned}
 \end{equation}
 Now we will distinguish between three cases: 
 \begin{enumerate}
 \item  $e_-\npreceq B^{(n,K)}_t$, i.e., $B^{(n,K)}$ is to the left of $e$ at time $t$.
\item $e'_+\preceq B^{(n,K)}_t$, i.e., $B^{(n,K)}$ is to the right of $e'$ at time $t$.
\item $e_-\preceq B^{(n,K)}_t$ but $e'_+\npreceq B^{(n,K)}_t$, i.e., $B^{(n,K)}$ is between $e_-$ and $e'_+$ by time $t$.
\end{enumerate} 
  In case $(1)$, $B^{(n,K)}$ cannot right-cross $e'$ without first right-crossing $e$ (since $e\prec e'$ and the tree structure of $\T^{(n,K)}$). Therefore, all the right-crossings of $e'$ corresponding to the last right-crossing of $e$ (in the definition of $N_i(e,e')$) have already been performed. Hence, we have that
 \begin{equation}\label{eq:pterodactilo1} 
 \reff^{G_n}(e)l_t^{(n,K,\rightarrow)}(e)-\reff^{G_n}(e')l_t^{(n,K,\rightarrow)}(e')=-\sum_{i=1}^{l_t^{(n,K,\rightarrow)}(e)}\eta_i(e,e').
 \end{equation}
 In case $(2)$, since $e\prec e'$, there has been at least $1$ right-crossing of $e'$ after the last right-crossing of $e$. Therefore, we have that 
 \begin{equation}
 \begin{aligned}
 &\reff^{G_n}(e)l_t^{(n,K,\rightarrow)}(e)-\reff(e')l_t^{(n,K,\rightarrow)}(e')+\sum_{i=1}^{l_t^{(n,K,\rightarrow)}(e)-1}\eta_i(e,e')\\\leq &\reff^{G_n}(e)-\reff^{G_n}(e').
 \end{aligned}
 \end{equation}
   The last two display together with \eqref{eq:estosdiasterribles} yield that in the cases $(1)$ and $(2)$
 \begin{equation}\label{eq:pterodactilo2}
 \begin{aligned}
 0&\leq \reff^{G_n}(e)l_t^{(n,K,\rightarrow)}(e)-\reff(e')l_t^{(n,K,\rightarrow)}(e')+\sum_{i=1}^{l_t^{(n,K,\rightarrow)}(e)}\eta_i(e,e')\\
 &\leq \reff^{G_n}(e)-\reff^{G_n}(e')+\eta_{l^{(n,K,\rightarrow)}_t(e)}(e,e').
 \end{aligned}
 \end{equation}
 In case $(3)$ we only get that
 \[
 \begin{aligned}
 &-\eta_{l^{(n,K,\rightarrow)}_t(e)}(e,e')\\ 
 \leq&\reff^{G_n}(e)l_t^{(n,K,\rightarrow)}(e)-\reff^{G_n}(e')l_t^{(n,K,\rightarrow)}(e')+\sum_{i=1}^{l_t^{(n,K,\rightarrow)}(e)-1}\eta_i(e,e')\\\leq &\reff^{G_n}(e),
 \end{aligned}
 \]
 or, equivalently,
 \begin{equation}
 \begin{aligned}
0&\leq\reff^{G_n}(e)l_t^{(n,K,\rightarrow)}(e)-\reff^{G_n}(e')l_t^{(n,K,\rightarrow)}(e')+\sum_{i=1}^{l_t^{(n,K,\rightarrow)}(e)}\eta_i(e,e')\\&\leq \reff^{G_n}(e) + \eta_{l^{(n,K,\rightarrow)}_t(e)}(e,e').
 \end{aligned}
 \end{equation}
 \emph{Multiple-edge comparison:}\\
 Let $e^1\prec e^2\prec \cdots e^k$ be an ordered sequence of edges.
 In this part we also distinguish between three cases: 
 \begin{enumerate}
 \item $e^1_-\npreceq B^{(n,K)}_t$, i.e., $B^{(n,K)}$ is to the left of $e^1_-$ at time $t$. 
 \item $e^k_+ \preceq B^{(n,K)}_t$, i.e., $B^{(n,K)}$ is to the right of $e^k_+$ at time $t$.
 \item  There exists an index $i^\ast\in\{1,\cdots,k-1\}$ that satisfies $e^{i^\ast}_-\preceq B^{(n,K)}_t$ and $e^{i^\ast+1}_+\npreceq B^{(n,K)}_t$, i.e., $B^{(n,K)}$ is between $e_-^1$ and $e^k_+$ at time $t$.
 \end{enumerate}
    In all the previous cases, 
 \begin{equation}\label{eq:fuentedepterodactilos}
\begin{aligned}
&\reff^{G_n}(e^1)l_t^{(n,K,\rightarrow)}(e^1)-\reff^{G_n}(e^k)l_t^{(n,K,\rightarrow)}(e^k)+\sum_{j=1}^{k-1}\sum_{i=1}^{l^{(n,K,\rightarrow)}_t(e^j)}\eta_i(e^j,e^{j+1})\\
=&\sum_{j=1}^{k-1}\left( \reff^{G_n}(e^j)l_t^{(n,K,\rightarrow)}(e^j)-\reff^{G_n}(e^{j+1})l_t^{(n,K,\rightarrow)}(e^{j+1})+\sum_{i=1}^{l^{(n,K,\rightarrow)}_t(e^j)}\eta_i(e^j,e^{j+1})\right).
\end{aligned} 
 \end{equation}
 In case $(1)$, by display \eqref{eq:pterodactilo1}, all the terms of the sum in the right hand side of \eqref{eq:fuentedepterodactilos} are equal to $0$. Hence, 
 \begin{equation}
 \abs{\reff^{G_n}(e^1)l_t^{(n,K,\rightarrow)}(e^1)-\reff^{G_n}(e^k)l_t^{(n,K,\rightarrow)}(e^k)}\leq \sum_{j=1}^{k-1}\abs{\sum_{i=1}^{l^{(n,K,\rightarrow)}_t(e^j)}\eta_i(e^j,e^{j+1})},
 \end{equation}
 and the claim of the lemma holds.\\
 In case $(2)$, by \eqref{eq:pterodactilo2}, each term of the sum in the right hand side of \eqref{eq:fuentedepterodactilos} is between $0$ and $\reff^{G_n}(e^j)-\reff^{G_n}(e^{j+1})+\eta_{l^{(n,K,\rightarrow)}_t(e^j)}(e^j,e^{j+1})$. Hence,
 \begin{equation}
\begin{aligned}
0&\leq \reff^{G_n}(e^1)l_t^{(n,K,\rightarrow)}(e^1)-\reff^{G_n}(e^k)l_t^{(n,K,\rightarrow)}(e^k)+\sum_{j=1}^{k-1}\sum_{i=1}^{l^{(n,K,\rightarrow)}_t(e^j)}\eta_i(e^j,e^{j+1})\\
&\leq \sum_{j=1}^{k-1} \left(\reff^{G_n}(e^j)-\reff^{G_n}(e^{j+1})\right)+\sum_{j=1}^{k-1} \eta_{l^{(n,K,\rightarrow)}_t(e^j)}(e^j,e^{j+1})\\
&=\reff^{G_n}(e^1)-\reff^{G_n}(e^k)+\sum_{j=1}^{k-1} \eta_{l^{(n,K,\rightarrow)}_t(e^j)}(e^j,e^{j+1}),
\end{aligned} 
 \end{equation} 
 where, in the last equality we have used the telescopic property of the sum.
 Hence
 \begin{equation}
 \begin{aligned}
 &\abs{\reff^{G_n}(e^1)l_t^{(n,K,\rightarrow)}(e^1)-\reff^{G_n}(e^k)l_t^{(n,K,\rightarrow)}(e^k)}\\\leq &\abs{\reff^{G_n}(e^1)-\reff^{G_n}(e^k)}+ \sum_{j=1}^{k-1}\abs{\sum_{i=1}^{l^{(n,K,\rightarrow)}_t(e^j)}\eta_i(e^j,e^{j+1})} + \sum_{j=1}^{k-1}\abs{\eta_{l^{(n,K,\rightarrow)}(e_j)}(e_j,e_{j+1})},
 \end{aligned}
 \end{equation}
 and the claim of the lemma holds.\\
 In case $(3)$, we simply analyze separately the indices from $1$ to $i^\ast-1$, the index $i^\ast$ and the indices from $i^{\ast}+1$ to $k-1$ in \eqref{eq:fuentedepterodactilos} to get
 \begin{equation}
\begin{aligned}
0&\leq \reff^{G_n}(e^1)l_t^{(n,K,\rightarrow)}(e^1)-\reff^{G_n}(e^k)l_t^{(n,K,\rightarrow)}(e^k)+\sum_{j=1}^{k-1}\sum_{i=1}^{l^{(n,K,\rightarrow)}_t(e^j)}\eta_i(e^j,e^{j+1})\\
&\leq \sum_{j=1}^{i^{\ast}-1} \left(\reff^{G_n}(e^j)-\reff^{G_n}(e^{j+1})\right)+\sum_{j=1}^{i^{\ast}-1} \eta_{l^{(n,K,\rightarrow)}_t(e^j)}(e^j,e^{j+1})\\
&\ \ + \reff^{G_n}(e^{i^\ast})+ \eta_{l^{(n,K,\rightarrow)}_t(e^{i^\ast})}(e^{i^\ast},e^{i^\ast+1})\\
&=\reff^{G_n}(e^k)+\sum_{j=1}^{i^{\ast}} \eta_{l^{(n,K,\rightarrow)}_t(e^j)}(e^j,e^{j+1}),
\end{aligned} 
 \end{equation} 
 Hence, we get that
 \begin{equation}
 \begin{aligned}
 &\abs{\reff^{G_n}(e^1)l_t^{(n,K,\rightarrow)}(e^1)-\reff(e^k)l_t^{(n,K,\rightarrow)}(e^k)}\\
 \leq & \sum_{j=1}^{k-1} \abs{\sum_{i=1}^{l^{(n,K,\rightarrow)}_t(e^j)}\eta_i(e^j,e^{j+1}) }+ \sum_{j=1}^{k-1} \abs{\eta_{l^{(n,K,\rightarrow)}_t(e)}(e^j,e^{j+1})}+ \reff^{G_n}(e^k),
 \end{aligned}
 \end{equation}
 which is the claim of the lemma. This finishes the proof.
\end{proof}

\subsubsection{Maximal inequality - proof of Lemma~\ref{lem:nonuniform}}

Let us now prove Lemma \ref{lem:nonuniform}.

\begin{proof}[Proof of Lemma \ref{lem:nonuniform}]

For the sake of simplicity, we will cover $\T^{(n,K)}$ with a finite number of linear segments and then work with each segment separately. Consider $h_{n,K}(i), i=1,\dots,K$, the set of leaves of $\T^{(n,K)}$ and  the intervals $I_i$ whose endpoints are the root and $h_{n,K}(i)$. 

Let $I^\star$ be a sub-interval of $I_i$ and $e^1\prec e^2\prec \cdots\prec e^{l}$ be the edges of $I^\star$ labeled in increasing genealogical order.
 Let 
 \begin{equation}
 \nu(e^j):=\reff^{G_n}(e^{j-1})+\reff^{G_n}(e^j)+\reff^{G_n}(e^{j+1})
\end{equation}
  for all $1<j<l$.
  It is not hard to see that 
  \begin{equation}\label{eq:ledebounacancion}
  \nu(I^\star):=\sum_{j=2}^{l-1}\nu(e^j)\leq 3\reff^{G_n}(I^\star),
  \end{equation}
  where for an interval $I$ of $\T^{(n,K)}$ (whose endpoints $I_-$ and $I_+$ are vertices of $\T^{(n,K)}$) we define $\reff^{G_n}(I)$ as $\reff^{G_n}(I_-,I_+)$.
  
Our first goal is to construct, for each $k\in\N$, a decomposition of the interval $I^{\star}$ into $2^k$ sub-intervals $I^{k,j}, j=1,\dots 2^k$, each one satisfying $\reff^{G_n}(I^{k,j})\leq \nu(I^\star) 2^{-k}$. This construction will have the following property: For every pair of edges $e,e'$ there exists a chain of edges between $e$ and $e'$ where every two consecutive edges in the chain are the left and right-most edges of some interval $I^{k,j}$ in the decomposition. Moreover, only a fixed number (indeed 2) of intervals of each level $k$ appear in the chain. Therefore, we will be able to control $\abs{\reff^{G_n}(e)l^{(n,K,\rightarrow)}_t(e)-\reff^{G_n}(e')l^{(n,K,\rightarrow)}_t(e')}$ by controlling the same quantity for edges which are left and right-most edges of some $I^{k,j}$.
The decomposition will be constructed by induction on $k$ and using the following result: Let $1\leq j_1<j_2\leq l$ be indices with $j_2-j_1\geq 2$ and $I^\circ:=\{e^j\in E(I_i): j_1\leq j\leq j_2\}$ be a sub-interval of $I^\star$ such that 
\begin{equation}\label{eq:splittrickhipo}
\nu(e^{j_1+1})+\nu(e^{j_1+2})+\cdots+\nu(e^{j_2-1})\leq \frac{\nu(I^\star)}{2^k}\end{equation}
for some $k\geq0$. Define 
\[i^\text{mid}:=\min\left\{i\geq j_1+1: \sum_{j=j_1+1}^{i}\nu(e^j)\geq \frac{\nu(I^\star)}{2^{k+1}}\right\}.\]
We will distinguish between four cases 
\begin{enumerate}
\item $j_1+1<i^\text{mid}<j_2-1$.
\item $i^\text{mid}\geq j_2$.
\item $i^\text{mid}=j_2-1$.
\item $i^{\text{mid}}=j_1+1$.
\end{enumerate}
In case $(1)$, a simple analysis shows that
\begin{equation}\label{eq:splittrick}
\begin{aligned}
&\nu(e^{j_1+1})+\nu(e^{j_1+2})+\cdots+\nu(e^{i^{\text{mid}}-1})\leq \frac{\nu(I^\star)}{2^{k+1}}\\
\text{and}\\
&\nu(e^{i^\text{mid}+1})+\nu(e^{i^\text{mid}+2})+\cdots+\nu(e^{j_2-1})\leq \frac{\nu(I^\star)}{2^{k+1}}.
\end{aligned}
\end{equation}
In this case we define 
\[I^{\circ,\text{left}}:=\{e^i:j_1\leq i \leq i^{\text{mid}}\},\] 
\[I^{\circ,\text{right}}:=\{e^i:i^\text{mid}\leq i\leq j_2\}.\]
Note that $I^{\circ,\text{left}}$ and $I^{\circ,\text{right}}$ satisfy \eqref{eq:splittrick}, which is analogous to the condition \eqref{eq:splittrickhipo} satisfied by $I^{\circ}$ but with $k+1$ in place of $k$.
In particular, a straightforward consequence of \eqref{eq:splittrick} is that
\begin{equation}\label{eq:splittrick2}
\reff^{G_n}(I^{\circ,\text{left}})\leq\frac{\nu(I^\star)}{2^{k+1}}
\quad \text{and}\quad\reff^{G_n}(I^{\circ,\text{right}})\leq\frac{\nu(I^\star)}{2^{k+1}}.  
\end{equation}
In case $(2)$ we have
\[\nu(e^{j_1+1})+\nu(e^{j_1+2})+\cdots+\nu(e^{j_2-1})\leq \frac{\nu(I^\star)}{2^{k+1}}.\]
In this case we define $I^{\circ,\text{left}}=I^{\circ,\text{right}}=I^\circ$. Observe that in this case $I^{\circ,\text{left}}$ and $I^{\circ,\text{right}}$ also satisfy \eqref{eq:splittrick} and \eqref{eq:splittrick2}.\\
In case $(3)$ we have that
\[\nu(e^{j_1+1})+\nu(e^{j_1+2})+\cdots+\nu(e^{j_2-2})\leq \frac{\nu(I^\star)}{2^{k+1}}\] 
and we can only get that
\[\nu(e^{j_2-1})\leq \frac{\nu(I^\star)}{2^{k}}.\]
In this case we define $I^{\circ,\text{left}}=\{e^i:j_1\leq i \leq j_2-1\}$ and $I^{\circ,\text{right}}:=\{e^{j_2-1},e^{j_2}\}$. We have
\[
\reff^{G_n}(I^{\circ,\text{left}})\leq \frac{\nu(I^\star)}{2^{k+1}} \quad \text{and} \quad \reff^{G_n}(I^{\circ,\text{right}})\leq \frac{\nu(I^\star)}{2^{k}}.
\]
In case $(4)$, we get that
\[\nu(e^{j_1+2})+\nu(e^{j_1+3})+\cdots+\nu(e^{j_2-2})\leq \frac{\nu(I^\star)}{2^{k+1}}\]
and 
\[\nu(e^{j_1+1})\leq \frac{\nu(I^\star)}{2^{k}}.\]
In this case we define $I^{\circ,\text{right}}:=\{e^{j_1},e^{j_1+1}\}$ and $I^{\circ,\text{left}}=\{e^i:j_1+1\leq i \leq j_2\}$. We have
\[
\reff^{G_n}(I^{\circ,\text{left}})\leq \frac{\nu(I^\star)}{2^{k}} \quad \text{and} \quad \reff^{G_n}(I^{\circ,\text{right}})\leq \frac{\nu(I^\star)}{2^{k+1}}.
\]

Note that, except when $I^{\circ,\text{left}}$ (resp. $I^{\circ,\text{right}}$) is composed of only two edges, we can use the result in \eqref{eq:splittrick} to decompose recursively $I^{\circ,\text{left}}$ (resp. $I^{\circ,\text{right}}$) into two intervals in the same way we have decomposed $I^{\circ}$. We do not decompose intervals composed of only two edges. By successively applying this procedure, we get for each $k\in\N$ a family of at most $2^k$ intervals $I^{i,k},i=1\dots j(k)\leq 2^k$ which are labeled in increasing genealogical order. Note that if $I^{i,k}$ is composed of more than two intervals, it satisfies
\begin{equation}\label{eq:gb}
\reff^{G_n}(I^{i,k})\leq \frac{\nu(I^\star)}{2^{k}},
\end{equation}
whereas, if $I^{i,k}$ is composed of two intervals, it satisfies
\begin{equation}\label{eq:bb}
\reff^{G_n}(I^{i,k})\leq \frac{\nu(I^\star)}{2^{k-1}}.
\end{equation}

Let $e^{i,k,\text{left}}$ and $e^{i,k,\text{right}}$ the leftmost and rightmost edges of $I^{i,k}$. Note that $e^{i,k,\text{right}}=e^{i+1,k,\text{left}}$ (for any index $i$ such that both $I^{i,k}$ and $I^{i+1,k}$ have more than two edges). The key observation is the following: for any pair of edges $e\prec e'$ in $I^\star$, let $k_\text{mid}:=\min\{k  \in\N: \exists i_\text{mid} \text{ with } e\preceq e^{k_\text{mid},i_\text{mid},\text{left}}, e^{k_\text{mid},i_\text{mid},\text{right}}\preceq e' \}$. That is, $I^{i_\text{mid},k_\text{mid}}$ is the interval of minimum level which lies between $e$ and $e'$. It is not hard to see that there exists a (finite) sequence of  levels $k_1^->k_2^-\dots>k_{l^-}^-$ and indices $i_1^-,i_2^-,\dots,i_{l^-}^-$ which satisfy $e^{i_1^-,k_1^-,\text{left}}=e$, $e^{i_{l^-}^-,k_{l^-}^-,\text{right}}=e^{i_{\text{mid}},k_\text{mid},\text{left}}$ and
\begin{equation}
e^{i_j^-,k_j^-,\text{right}}=e^{i_{j+1}^-,k_{j+1}^-,\text{left}} \quad \forall j=1\dots l^--1.
\end{equation}
Similarly, there exists a sequence of levels $k_1^+<k_2^+\dots<k_{l^+}^+$ and indices $i_1^+,i_2^+,\dots,i_{l^+}^+$ which satisfy $e^{i_1^+,k_1^+,\text{left}}=e^{i_{\text{mid}},k_\text{mid},\text{right}}$, $e^{i_{l^+}^+,k_{l^+}^+,\text{right}}=e'$,  and
\begin{equation}
e^{i_j^+,k_j^+,\text{right}}=e^{i_{j+1}^+,k_{j+1}^+,\text{left}} \quad \forall j=1\dots l^+-1.
\end{equation}
Therefore, we have constructed a chain of edges 
\[
\begin{aligned}
e=&e^{i_1^-,k_1^-,\text{left}}\preceq e^{i_1^-,k_1^-,\text{right}}=e^{i_2^-,k_2^-,\text{left}}\preceq\dots\preceq e^{i_{l^-}^-,k_{l^-}^-,\text{right}}=e^{i_{\text{mid}},k_\text{mid},\text{left}}\\
&\preceq e^{i_{\text{mid}},k_\text{mid},\text{right}}=e^{i_1^+,k_1^+,\text{left}}\preceq e^{i_1^+,k_1^+,\text{right}}=e^{i_2^+,k_2^+,\text{left}}\preceq\dots \preceq e^{i_{l^+}^+,k_{l^+}^+,\text{right}}=e',
\end{aligned}
\]
which connect $e$ to $e'$ and all of them are endpoints of intervals of the decomposition. 

This construction is depicted in figure 7. The edges of the chain  are numbered from 1 to 5 in this specific case. Note that the edges of the chain (except for the first and the last, i.e., from 2 to 4) are those which lie in the overlap of different intervals of the decomposition, i.e., they are the right-most edge of one interval and the left-most edge of the next interval. That is explicit in the figure for edge 4.

\begin{figure}\label{fig:maximalinequality}
  \includegraphics[scale=0.7]{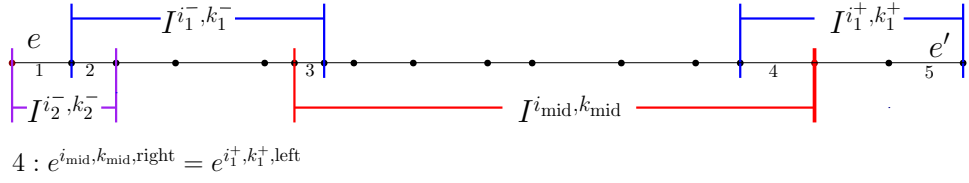}
  \caption{Chain of edges}
\end{figure}

Moreover, since both $e^{i_j^-,k_j^-,\text{left}}$ and $e^{i_{j}^+,k_{j}^-,\text{right}}$ are contained in $I^{i_j^-,k_j^-}$, we get from \eqref{eq:gb} and \eqref{eq:bb} that
\begin{equation}\label{eq:unsoloamor}
\reff^{G_n}(e^{i_j^-,k_j^-,\text{left}}_-,e^{i_j^-,k_j^-,\text{right}}_+)\leq \frac{\nu(I^\star)}{2^{k_j^--1}}.
\end{equation}

Similarly
\begin{equation}\label{eq:unsoloamor2}
\reff^{G_n}(e^{i_j^+,k_j^+,\text{left}}_-,e^{i_j^+,k_j^+,\text{right}}_+)\leq \frac{\nu(I^\star)}{2^{k_j^+-1}}.
\end{equation}

Setting
 \begin{align*}
\mathrm{M}^{k}(I^\star):=
 \max_{j=1,\dots,j(k)-1} \biggl\{&\abs{\sum_{i=1}^{l_t^{(n,K,\rightarrow)}(e^{j,k,\text{left}})}\eta_i(e^{j,k,\text{left}},e^{j,k,\text{right}})}  +\\& \abs{\eta_{l^{(n,K,\rightarrow)}_t(e^{j,k,\text{left}})}(e^{j,k,\text{left}},e^{j,k,\text{right}})}\biggl\}
 \end{align*}
we get, by Lemma \ref{lem:etatrick} that
 \begin{equation}\label{eq:pickyalex}
 \sup_{e,e'\in E(I^\star)} \abs{R^{G_n}_{\text{eff}}(e)l^{(n,K,\rightarrow)}_t(e)-R^{G_n}_{\text{eff}}(e')l^{(n,K,\rightarrow)}_t(e')}\leq  \sum_{k=1}^{\infty}2\mathrm{M}^k(I^\star)+\max_{e\in E(I_i)}\reff^{G_n}(e),
 \end{equation}
 where the $2$ in front of $\mathrm{M}^k(I^\star)$ appears since there might be up to two intervals of the same level in the chain of edges which join $e$ to $e'$(one to the right of $I^{i_{\text{mid}},k_{\text{mid}}}$ and the other to the left of $I^{i_{\text{mid}},k_{\text{mid}}}$).
 Let $\theta$ be a constant in $(0,1)$ and $C$ such that $C\sum_{k=1}^\infty \theta^k=1/2$. Then
\begin{equation}\label{eq:escaramujo}
  \mathbf{P}^{(K)}\left[\sum_{k=1}^{\infty}2 \mathrm{M}^k(I^\star)\geq \epsilon \right]
 \leq \sum_{k=1}^\infty \mathbf{P}^{(K)}\left[  \mathrm{M}^k(I^\star) \geq C \theta^k\epsilon \right].
  \end{equation}
  Using Lemma \ref{lem:premaximalinequality}, \eqref{eq:unsoloamor} and \eqref{eq:unsoloamor2} we get that
  \begin{equation}
  \begin{aligned}
  \mathbf{P}^{(K)}&\left[ \mathrm{M}^k(I^\star) \geq C \theta^k\epsilon \right]\\
  \leq\sum_{j=1}^{2^k-1} \mathbf{P}^{(K)}&\Bigg[\abs{\sum_{i=1}^{l_t^{(n,K,\rightarrow)}(e^{j,k,\text{left}})}\eta_i(e^{j,k,\text{left}},e^{j,k,\text{right}})}\\
  & +\abs{\eta_{l^{(n,K,\rightarrow)}_t(e^{j,k,\text{left}})}(e^{j,k,\text{left}},e^{j,k,\text{right}})}\geq C \theta^k\epsilon \Bigg] \\
  \leq  2^{k}\mathbf{P}^{(K)}&\left[\abs{\sum_{i=1}^{l_t^{(n,K,\rightarrow)}(e^{j,k,\text{left}})}\eta_i(e^{j,k,\text{left}},e^{j,k,\text{right}})} \geq \frac{C}{2} \theta^k\epsilon \right]\\
  +2^k\mathbf{P}^{(K)}&\left[\abs{\eta_{l^{(n,K,\rightarrow)}_t(e^{j,k,\text{left}})}(e^{j,k,\text{left}},e^{j,k,\text{right}})}\geq \frac{C}{2} \theta^k\epsilon \right]\\
  \leq 2^{k}(C_1+&C_2) \left(\frac{C}{2}\theta^k\epsilon\right)^{-4} \left(\frac{\nu(I^\star)}{2^{k-1}}\right)^2
  \end{aligned}
  \end{equation}
Therefore, the right hand side of \eqref{eq:escaramujo} is bounded above by
\begin{equation}
\begin{aligned}
&\sum_{k=1}^\infty4(C_1+C_2) 2^k\left(\frac{C}{2}\theta^k\epsilon\right)^{-4} \left(\frac{\nu(I^\star)}{2^{k}}\right)^2\\
= & 4(C_1+C_2) \left(\frac{C}{2}\right)^{-4} \epsilon^{-4} \nu(I^\star)^2\sum_{k=1}^\infty \left(\frac{1}{2\theta^{4}}\right)^k.
\end{aligned}
\end{equation}
It is possible to choose $\theta\in (0,1)$ such that $(2\theta^{4})^{-1}<1$.
Therefore, from the display above we get that there exists a constant $C_3$ such that
\begin{equation}
 \mathbf{P}^{(K)}\left[\sum_{k=1}^{\infty} \mathrm{M}^k(I^\star)\geq \epsilon \right]\leq C_3 \epsilon^{-4}(\nu(I^\star))^2.
\end{equation}
Moreover, by \eqref{eq:ledebounacancion} and setting $C_4=9C_3$ we get that
\begin{equation}
 \mathbf{P}^{(K)}\left[\sum_{k=1}^{\infty} \mathrm{M}^k(I^\star)\geq \epsilon \right]\leq C_4\epsilon^{-4} (\reff^{G_n}(I^\star))^2.
\end{equation}
Recall that $I_i$ is one of elements in the decomposition of $\T^{(n,K)}$ into linear segments. 
Let $(I^{(\delta,j)}_i)_{j=1}^{l(\delta)}$ be a covering of $I_i$ by sub-intervals of resistance-length between $\delta/2$ and $\delta$, where $l(\delta)=l(\delta,n,K)$ is the number of sub-intervals in the covering. By Lemma \ref{lem:graphspatialresistancetreeconvergence} and Lemma \ref{lem_no_macro_res} we have that we can choose the covering in a way that 
\begin{equation}\label{eq:jovensoldado}
\sup_{n\in \N,\delta>0}\delta l(\delta,n,K)<\infty
\end{equation} 
We have that, for any $i\leq K$
\begin{equation}\label{eq:evadejadesercostilla}
\begin{aligned}
&\mathbf{P}^{(K)}\left[\sup_{\stackrel{e,e'\in E(I_i)}{d^{(n,K)}_{\text{res}}(e,e')\leq\delta}} n^{-1/2}\abs{R^{G_n}_{\text{eff}}(e)l^{(n,K,\rightarrow)}_t(e)-R^{G_n}_{\text{eff}}(e')l^{(n,K,\rightarrow)}_t(e')}\geq\epsilon \right]\\
\leq&\sum_{j=1}^{l(\delta)}\mathbf{P}^{(K)}\left[\sup_{e,e'\in E(I^{(\delta,j)}_i)} n^{-1/2}\abs{R^{G_n}_{\text{eff}}(e)l^{(n,K,\rightarrow)}_t(e)-R^{G_n}_{\text{eff}}(e')l^{(n,K,\rightarrow)}_t(e')}\geq\frac{\epsilon}{2} \right]\\
\leq & \sum_{j=1}^{l(\delta)}\mathbf{P}^{(K)}\left[\sum_{k=1}^\infty 2M^k(I^{(\delta,j)}_i)+ \max_{e\in E(I^{(\delta,j)}_i)}\reff^{G_n}(e) \geq\frac{\epsilon}{2} \right]\\
\leq& \sum_{j=1}^{l(\delta)}\left[\mathbf{P}^{(K)}\left[\sum_{k=1}^\infty 2M^k(I^{(\delta,j)}_i) \geq\frac{\epsilon}{4} \right]+\mathbf{P}^{(K)}\left[ \max_{e\in E(I^{(\delta,j)}_i)}\reff^{G_n}(e) \geq\frac{\epsilon}{4} \right]\right]\\
\leq & \sum_{j=1}^{l(\delta)} C_4\left(\frac{\epsilon}{8}\right)^{-4} \delta^2+\sum_{j=1}^{l(\delta)}\mathbf{P}^{(K)}\left[ \max_{e\in E(I^{(\delta,j)}_i)}\reff^{G_n}(e) \geq\frac{\epsilon}{4} \right]\\
\leq& l(\delta)\delta C_4\left(\frac{\epsilon}{8}\right)^{-4} \delta+l(\delta)\mathbf{P}^{(K)}\left[ \max_{e\in E(\T^{(n,K)})}\reff^{G_n}(e) \geq\frac{\epsilon}{4} \right],
\end{aligned}
\end{equation}
where we have used \eqref{eq:pickyalex} between line $2$ and $3$.
By display \eqref{eq:jovensoldado} we have that
\begin{equation}
\lim_{\delta\to0}\limsup_{n\in\N}l(\delta)\delta C_4\left(\frac{\epsilon}{8}\right)^{-4} \delta=0.
\end{equation}
Moreover, by Lemma \ref{lem_no_macro_res} we have that, for each $\delta$ fixed,
\begin{equation}
\lim_{n\to\infty} l(\delta)\mathbf{P}^{(K)}\left[ \max_{e\in E(\T^{(n,K)})}\reff^{G_n}(e) \geq\frac{\epsilon}{4} \right]=0.
\end{equation}
The two displays above together with \eqref{eq:evadejadesercostilla} imply that
\begin{equation}
\lim_{\delta\to0}\limsup_{n\in\N}\mathbf{P}^{(K)}\left[\sup_{\stackrel{e,e'\in E(I_i)}{d^{(n,K)}_{\text{res}}(e,e')\leq\delta}} n^{-1/2}\abs{R^{G_n}_{\text{eff}}(e)l^{(n,K,\rightarrow)}_t(e)-R^{G_n}_{\text{eff}}(e')l^{(n,K,\rightarrow)}_t(e')}\geq\epsilon \right]=0.
\end{equation}
  
 Since the number of linear segments $I_i$ in the decomposition of $\T^{(n,K)}$ is at most $K$, the claim of the Lemma follows from the display above.
\end{proof}

\subsection{Proof of the convergence of local times}

\subsubsection{Finite-dimensional convergence of local times}

The strategy to prove Proposition \ref{l:convergenceoflocaltimes} is to get the convergence at a single site/edge and then to apply Lemma \ref{lem:modulusofcontinuity} to get uniformity.
The following lemma takes care of the convergence at a single site/edge.
 For any edge $e\in E(\T^{(n,K)})$, let  $x_e\in \T^{(n,K)}$ be the midpoint of $e$. 
\begin{lemma}\label{lem:convergenceoflocaltimessingle}
For all $t,\epsilon>0$ and $K\in\N$
\[
\lim_{n\to\infty}\sup_{e\in E^\ast(\T^{(n,K)})}\mathbf{P}^{(K)}\left[ \abs{ n^{-1/2}\reff^{G_n}(e)l^{(n,K)}_t(e)-L^{(n,K)}_t(x_e) }  \geq\epsilon\right]=0.
\]
\end{lemma}
\begin{proof}[Proof of Lemma \ref{lem:convergenceoflocaltimessingle}:]
  Define $\theta^{n,K,\text{mid}}_0(e):=0$, 
   \[
 \theta^{n,K,\text{side}}_0(e):=\inf\{s\geq0:B^{(n,K)}_s\in\{e_{-},e_{+}\}\},
 \]
  \[
  \theta^{n,K,\text{mid}}_i(e):=\inf\{s\geq \theta^{n,K,\text{mid}}_{i-1}(e):B^{(n,K)}_s=x_e\},
  \]  
  and
   \[
   \theta^{n,K,\text{side}}_i(e):=\inf\{s \geq \theta^{n,K,\text{mid}}_i(e): B^{(n,K)}_s\in\{e_{-},e_+\} \}.
   \]
   That is, $\theta^{n,K,\text{side}}_i$ is the time of the $i$-th visit of $\{e_-,e_+\}$ by  $B^{(n,K)}$, if we regard two consecutive visits as different if they are separated by a visit to $x_e$.
Let
\[
\zeta^{n,K}_i(e):=L^{(n,K)}_{\theta^{n,K,\text{side}}_{i+1}(e)}(x_e)-L^{(n,K)}_{\theta^{n,K,\text{side}}_{i}(e)}(x_e), \quad{i\in\mathbb{N}},
\]
i.e., $\zeta^{n,K}_i(e)$ is the increment of the local time at $x_e$ between $\theta^{n,K,\text{side}}_{i}(e)$ and $\theta^{n,K,\text{side}}_{i+1}(e)$.
 By the strong Markov property, we have that $(\zeta^{n,k}_i(e))_{i\in\N}$ is an i.i.d.~sequence of random variables. Moreover, it is a known fact that the $(\zeta_i^{n,k}(e))_{i\in\N}$ are exponentially distributed with mean
  \begin{equation}\label{eq:meanofzeta}
  E^{G_n}[\zeta_i^{n,k}(e)]=\frac{d^{(n,K)}_{\text{res}}(e_-,e_+)}{2}=\frac{\reff^{G_n}(e)}{2n^{1/2}}.
  \end{equation}
Let $a(0):=0$ and 
\[a(i):=\min\left\{i > a(i-1): B^{(n,K)}_{\theta^{n,K,\text{side}}_{a(i)}(e)}\neq B^{(n,K)}_{\theta^{n,K,\text{side}}_{a(i-1)}(e)}\right\},\qquad i\in\N.\]
  Let 
 \begin{equation}\label{eq:defofchi}
 \chi_i^{n,K}(e):=\sum_{j=a(i-1)+1}^{a(i)}\zeta^{n,K}_j(e), \qquad i\in\N.
 \end{equation}
That is, $\chi_i^{n,K}(e)$ is the increment of the local time at $x_e$ between the $(i-1)$-th and $i$-th undirected crossing of $e.$ 
  By the strong Markov property we have that $a(i)-a(i-1)-1$ is a geometric random variable. Moreover, since $x_e$ is the midpoint of the edge $e$, $a(i)-a(i-1)-1$ is geometric of parameter $1/2$. Therefore
  \begin{equation}\label{eq:expectationofthelocaltime}
 E^{G_n}[\chi_i^{n,K}(e)]=\sum_{j=1}^{\infty}j E^{G_n}[\zeta_i^{n,K}(e)] (1/2)^{j}
 =\frac{\reff^{G_n}(e)}{n^{1/2}},
 \end{equation}
 where the last equality follows from \eqref{eq:meanofzeta}.

We have
 \begin{equation}\label{eq:triangleineq}
 \begin{aligned}
 &P^{G_n}\left[\abs{ n^{-1/2}\reff^{G_n}(e)l^{(n,K)}_t(e)-L^{(n,K)}_t(x_e) }  \geq\epsilon\right]\\
 \leq &P^{G_n}\left[ \left\vert n^{-1/2}\reff^{G_n}(e)l^{(n,K)}_t(e)-\sum_{i=1}^{l^{(n,K)}_t(e)} \chi_i^{n,K}(e) \right\vert  \geq\epsilon/2\right]\\
 + &P^{G_n}\left[ \left\vert L_t^{(n,K)}(x_e)-\sum_{i=1}^{l^{(n,K)}_t(e)} \chi_i^{n,K}(e) \right\vert  \geq\epsilon/2\right]
 \end{aligned}
 \end{equation}
By \eqref{eq:expectationofthelocaltime} and Chebishev's inequality
\begin{equation}
\begin{aligned}
&P^{G_n}\left[ \left\vert n^{-1/2}\reff^{G_n}(e)l^{(n,K)}_t(e)-\sum_{i=1}^{l^{(n,K)}_t(e)} \chi_i^{n,K}(e) \right\vert  \geq\epsilon/2\right] \\
\leq & 4\epsilon^{-2}E^{G_n}\left[\left( \sum_{i=1}^{l^{(n,K)}_t(e)} (\chi_i^{n,K}(e)-E^{G_n}[\chi_i^{n,K}(e)]) \right)^2 \right]
 \end{aligned}
 \end{equation}
 Using display \eqref{eq:secondmoment} in the appendix, we find that the display above is upper bounded by
 \[ 4 \epsilon^{-2}E^{G_n}[l^{(n,K)}_t(e)] E^{G_n}\left[  (\chi_i^{n,K}(e)-E^{G_n}[\chi_i^{n,K}(e)])^2 \right],\]
 which, by \eqref{eq:localtimemoments}, is bounded above by
 \begin{equation}\label{eq:gettingridofdeltanet}
 \begin{aligned}
  &4 \epsilon^{-2} C \reff^{G_n}(e)^{-1} n^{1/2} E^{G_n}\left[  (\chi_i^{n,K}(e)-E^{G_n}[\chi_i^{n,K}(e)])^2 \right]\\
  =&4 \epsilon^{-2} C \reff^{G_n}(e)^{-1} n^{1/2} E^{G_n}[\chi_i^{n,K}(e)]^{2} E^{G_n}\left[\left(\frac{\chi^{n,K}_i(e)}{E^{G_n}[\chi_i^{n,K}(e)]}-1\right)^2\right],
  \end{aligned}
  \end{equation}
for a positive constant $C$. 
 Moreover, by \eqref{eq:defofchi} and the fact that $\zeta^{n,K}_i$ are exponentially distributed,  $E^{G_n}[\chi_i^{n,K}(e)]^{-1}\chi_i^{n,K}(e)$ is a geometric sum (of parameter $1/2$) of independent exponential random variables of mean $1/2$. Therefore the distribution of $E^{G_n}[\chi_i^{n,K}(e)]^{-1}\chi_i^{n,K}(e)$ is independent of $n$ and, consequently, there exists $C$, independent of $n$, such that
 \[ E^{G_n}\left[\left(\frac{\chi^{n,K}_i(e)}{E^{G_n}[\chi_i^{n,K}(e)]}-1\right)^2\right]<C.\]
 
Using the display above and \eqref{eq:expectationofthelocaltime} in \eqref{eq:gettingridofdeltanet} we get that
 \begin{equation}
 P^{G_n}\left[ \left\vert n^{-1/2}\reff^{G_n}(e)l^{(n,K)}_t(e)-\sum_{i=1}^{l^{(n,K)}_t(e)} \chi_i^{n,K}(e) \right\vert  \geq\epsilon/2\right]
  \leq  C n^{-1/2} \reff^{G_n}(e),
 \end{equation}
  for some constant $C$ (which depends on $\epsilon$). By Lemma $\ref{lem_no_macro_res}$, the last term goes to $0$ in $\mathbf{P}^{(K)}$-probability, uniformly over $e\in E^{\ast}(\T^{(n,K)})$. 
 Therefore
  \begin{equation}\label{eq:step1}
 \sup_{e\in E^\ast(\T^{(n,K)})}  P^{G_n}\left[  \left\vert n^{-1/2}\reff^{G_n}(e)l^{(n,K)}_t(e)-\sum_{i=1}^{l^{(n,K)}_t(e)} \chi_i^{n,K}(e) \right\vert  \geq\epsilon/2\right]\to0 \qquad \mathbf{P}^{(K)}-\text{a.s.},  
  \end{equation}
as $n\to\infty$.

 On the other hand, recalling the coupling between $J^{(n,K)}$ and $B^{(n,K)}$ in Section \ref{sect:couplingrwbm} and the definition of $l^{(n,K)}_t(e)$, it follows that
\begin{equation}\label{eq:landL}
\sum_{i=1}^{l^{(n,K)}_t(e)}\chi_i^{n,K}(e)\leq L^{(n,K)}_t(x_e)\leq\sum_{i=1}^{l^{(n,K)}_t(e)+1}\chi_i^{n,K}(e).
\end{equation}
Therefore
\begin{equation}\label{eq:prestep2}
\begin{aligned}
P^{G_n}\left[ \left\vert L_t^{(n,K)}(x_e)-\sum_{i=1}^{l^{(n,K)}_t(e)} \chi_i^{n,K}(e) \right\vert  \geq\epsilon/2\right]
\leq &P^{G_n}\left[  \chi_{l^{(n,K)}_t(e)+1} ^{n,K}(e) \geq \epsilon/2 \right]\\
\leq &P^{G_n}\left[\sup_{i\leq l^{(n,K)}_t(e)+1} \chi_i^{n,K}(e)\geq \epsilon/2\right]\\
\leq&E^{G_n}\left[\sum_{i=1}^{l^{(n,K)}_t(e)+1} 1_{\{ \chi_i^{n,K}(e) \geq \epsilon/2 \}}\right],
\end{aligned}
\end{equation}
which, by Wald's identity equals 
\begin{align*}
&E^{G_n}[l^{(n,K)}_t(e)+1]P^{G_n}[\chi_i^{n,K}(e)\geq\epsilon/2]\\
\leq &4\epsilon^{-2}E^{G_n}[l^{(n,K)}_t(e)+1]E^{G_n}[(\chi_i^{n,K}(e))^2]\\
\leq& 4\epsilon^{-2}(Cn^{1/2}\reff^{G_n}(e)^{-1}+1)\reff^{G_n}(e)^{2}n^{-1} E^{G_n}\left[\left(\frac{\chi_i^{n,K}(e)}{E[\chi_i^{n,K}(e)]}\right)^2\right],
\end{align*}
where we have used \eqref{eq:localtimemoments}, \eqref{eq:expectationofthelocaltime} and Chebyshev's inequality. Moreover, recalling that the distribution of $\frac{\chi_i^{n,k}(e)}{E^{G_n}[\chi_i^{n,k}(e)]}$ does not depend upon $n$ nor $e$ and has finite second moment, we find that there exists a constant $C$ such that the display above is bounded by
\begin{align*}
 4\epsilon^{-2}Cn^{-1/2}\reff^{G_n}(e) + 4\epsilon^{-2}C \reff^{G_n}(e)^2n^{-1}.
\end{align*}
which, by Lemma \ref{lem_no_macro_res}, goes to $0$ in $\mathbf{P}^{(K)}$ probability.  
Therefore
\begin{equation}\label{eq:step2}
\lim_{n\to\infty}\sup_{e\in \T^{(n,K)}}\bold{P}^{(K)}\left[ \left\vert L_t^{(n,K)}(x_e)-\sum_{i=1}^{l^{(n,K)}_t(e)} \chi_i^{n,K}(e) \right\vert  \geq\epsilon/2\right]  \to 0
.
\end{equation}


Putting together \eqref{eq:triangleineq}, \eqref{eq:step1} and \eqref{eq:step2} we finish the proof. 
\end{proof}

\subsubsection{First part of the proof of Proposition~\ref{l:convergenceoflocaltimes}}\label{proof_prop61b}

The main purpose of this section is to prove the following:

\begin{lemma}\label{lem:abstractcoupling} Under \eqref{eq:couplinggeometry} and condition $(R)_{\rho}$, we have that, for all $R\geq0$ and $K\in\N$,
\[
(L^{(n,K)}_t(\Upsilon_{n,K}(x)))_{t\in[0,R],x\in\frak{T}^{(K)}} \stackrel{n\to\infty}{\to}(\sigma_d\rho L^{(K)}_{(\rho\sigma_d)^{-1}t}(x)))_{t\in[0,R],x\in\frak{T}^{(K)}}
\]
in distribution in the space $C([0,R]\times\frak{T}^{(K)},\R_+)$ endowed with the uniform topology, $\mathbf{P}^{(K)}$-a.s.
 \end{lemma}
 Note that display \eqref{eq:conv3} in Lemma \ref{l:convergenceoflocaltimes} follows immediately from \eqref{eq:convergenceofedgelocaltimes} and the lemma above. The techniques of this (sub)section are somewhat different from the rest of Section \ref{section_cvg_local_time} due to the continuous nature of the processes involved in the main result. On a first reading, the reader can skip the proof of the lemma above and go directly to Section \ref{proof_prop61a} for a more fluent reading. We choose to put Lemma \ref{lem:abstractcoupling} here since we will use it in Section \ref{proof_prop61a}.
 
 For the proof of the Lemma above we will need to introduce some notation.
 Recalling the definition at~\ref{eq:defofupsilon}, we define a distance $d^{(n,K)}_{\frak{T},\text{res}}$ on $\frak{T}^{(K)}$ as
 \[d^{(n,K)}_{\frak{T},\text{res}}(x,y):=d^{(n,K)}_{\text{res}}(\Upsilon_{n,K}(x),\Upsilon_{n,K}(y)).\]
  Let
 \begin{equation}\label{eq:defofbarL}
 \bar{L}^{(n,K)}_t(x):=L^{(n,K)}_t(\Upsilon_{n,K}(x))
 \end{equation}
 and 
 \begin{equation}
 \bar{B}^{(n,K)}_t:=\Upsilon_{n,K}^{-1}(B^{(n,K)}_t).
 \end{equation}
 
 Lemma \ref{lem:abstractcoupling} relies on the following result.
\begin{lemma}\label{lem:localtimeexpression}
 The process $\bar{B}^{(n,K)}$ is the Brownian motion in $(\frak{T}^{(K)}, d^{(n,K)}_{\frak{T},\text{res}},\lambda^{(n,K)}_{\frak{T},\text{res}})$ according to Proposition \ref{prop_def_process}, where $\lambda^{(n,K)}_{\frak{T},\text{res}}$ is the Lebesgue measure associated to $d^{(n,K)}_{\frak{T},\text{res}}$ normalized to become a probability measure.
Also $(\bar{L}^{(n,K)}_t(x))_{t\geq0,x\in\frak{T}^{(K)}}$ is the local time of $\bar{B}^{(n,K)}$ with respect to $\lambda^{(n,K)}_{\frak{T},\text{res}}$.
 \end{lemma}

 \begin{proof}
 The first part of the claim follows from the fact that  $\bar{B}^{(n,K)}$ satisfies all the properties in the definition of the Brownian motion in $(\frak{T}^{(K)},d^{(n,K)}_{\frak{T},\text{res}},\lambda^{(n,K)}_{\frak{T},\text{res}})$ given above Proposition \ref{prop_def_process}.
All properties are straightforward to verify.
 
 The part of the local times is proved as follows. Let $A$ be a Borelian of $\frak{T}^{(K)}$
 \begin{align*}
 &\text{Leb}\{s\leq t: \bar{B}^{(n,K)}_s\in A\}=\text{Leb}\{s\leq t: B^{(n,K)}_s\in \Upsilon_{n,K} (A)\}\\
 =&\int_{\Upsilon_{n,K}(A)}L_t^{(n,K)}(x)\lambda^{(n,K)}_{\text{res}}(dx)=\int_{A}L_t^{(n,K)}(\Upsilon_{n,K}(y))\lambda^{(n,K)}_{\frak{T},\text{res}}(dy)\\
 \end{align*}
 where the third equality follows from the change of variables $x=\Upsilon^{-1}_{n,K}(y)$, and the fact that $\lambda^{(n,K)}_{\frak{T},\text{res}}(\Upsilon^{-1}_{n,K}(\cdot))=\lambda^{(n,K)}_{\text{res}}(\cdot)$.
 \end{proof}

 \begin{proof}[Proof of Lemma \ref{lem:abstractcoupling}] 
First, we will use a result of \cite{croydon2012scaling} which gives conditions for the convergence of local times in graph-trees. After that we will show that  Lemma \ref{lem:localtimeexpression} allows us to prove the lemma from the results of the first step.

\emph{First step: Application of a general result concerning the convergence of local times.}

We start by stating the general convergence result of \cite{croydon2012scaling}. Let $T$ be a graph spatial tree and $d$, $d_n,n\in\N$ be distances in $T$ which are compatible with its topology.
Let $\lambda$, $\lambda_n, n\in\N$ be the non-normalized Lebesgue measures in $T$ associated to $d,d_n,n\in\N$ respectively. 
Also, let $B, B^n,n\in\N$ be the Brownian motions in $(T,d,\lambda),(T,d_n,\lambda_n),n\in\N$ respectively. Finally, let $L_t(x),L^n_t(x),n\in\N$ be the local times of $B,B^{n},n\in\N$ with respect to the measures $\lambda,\lambda_n,n\in\N$ respectively. 

 Proposition 3.1 of \cite{croydon2012scaling} states that if there exists a sequence $\delta_n\downarrow 1$ as $n\to\infty$ such that, for $n$ large enough.
\begin{equation}
\delta_n^{-1}d_n(x,y) \leq  d(x,y)\leq \delta_nd_n(x,y) \quad\forall x,y \in T,
\end{equation}
then we have that
 \begin{equation}\label{aeq:tildes!}
 (L^n_t(x))_{x\in T,t\in[0,R]}\stackrel{n\to\infty}{\to}(L_t(x))_{x\in T,t\in[0,R]}.
 \end{equation}
 
 To be able to apply the result above, we need that for processes to be defined as Brownian motions with respect to the non-normalized Lebesgue measure. Therefore, let $\tilde{B}^{(K)}$ be the Brownian motion on $(\frak{T}^{(K)},\sigma_d\rho d_{\frak{T}},\sigma_d\rho\Lambda^{(K)}\lambda_{\frak{T}}^{(K)})$, where $\Lambda^{(K)}$ is the total $d_{\frak{T}}$-length of $\frak{T}$. Note that $\sigma_d\rho\Lambda^{(K)}\lambda_{\frak{T}}^{(K)}$ is the non-normalized Lebesgue measure associated to the distance $\sigma_d\rho d_{\frak{T}}$. Similarly, let $\tilde{B}^{(n,K)}$ be the Brownian motion on $(\frak{T}^{(K)},d_{\frak{T},\text{res}}^{(n,K)},\Lambda^{(n,K)}_{\text{res}}\lambda^{(n,K)}_{\frak{T},\text{res}})$, where $\Lambda^{(n,K)}_{\text{res}}$ is the total $d^{(n,K)}_{\frak{T},\text{res}}$-length of $\frak{T}^{(K)}$. Let $\tilde{L}^{(K)}$ be the local time of $\tilde{B}^{(K)}$ with respect to $\sigma_d\rho\Lambda^{(K)}\lambda_{\frak{T}}^{(K)}$ and $\tilde{L}^{(n,K)}$ be the local time of $\tilde{B}^{(n,K)}$ with respect to the measure $\Lambda^{(n,K)}_{\text{res}}\lambda^{(n,K)}_{\frak{T},\text{res}}$. 
By Proposition 3.1 of \cite{croydon2012scaling} stated above if, $\mathbf{P}^{(K)}$-almost surely, there exists a sequence $\delta_n\downarrow 1$ as $n\to\infty$ such that, for $n$ large enough,
\begin{equation}\label{eq:homogenizationofR}
\delta_n^{-1}d^{(n,K)}_{\frak{T},\text{res}}(x,y) \leq \sigma_d\rho d_{\frak{T}}(x,y)\leq \delta_nd^{(n,K)}_{\frak{T},\text{res}}(x,y) \quad\forall x,y \in\frak{T}^{(K)},
\end{equation}
then 
 \begin{equation}\label{eq:tildes!}
 (\tilde{L}^{(n,K)}_t(x))_{x\in\frak{T}^{(K)},t\in[0,R]}\stackrel{n\to\infty}{\to}(\tilde{L}^{(K)}_t(x))_{x\in\frak{T}^{(K)},t\in[0,R]}
 \end{equation}
 in distribution in $C(\frak{T}^{(K)}\times [0,R],\R_+)$, for all $R>0$, $\mathbf{P}^{(K)}$-almost surely.
Now we focus on proving \eqref{eq:homogenizationofR}.
Recall that $\frak{T}^{(K)}$ and $\frak{T}^{(n,K)}$ are composed of a finite number of edges  $(e(i))_{i=1,\dots,s}$ and $(e^n(i))_{i,\dots,s}$ respectively, where $s$ is the number of line segments in the decomposition of $\frak{T}^{(K)}$ and $\frak{T}^{(n,K)}$. First, we will treat the case when $x,y$ lie in the same line segment $e(i)=[a_i,b_i]$ of $\frak{T}^{(K)}$. Let $\alpha_x$ and $\alpha_y$ be the $d_{\frak{T}^{(K)}}$-distances to $a_i$ of $x$ and $y$ respectively.  It follows from the definition of $\Upsilon_{n,K}$ in \eqref{eq:defofupsilon} that $\Upsilon_{n,K}(x), \Upsilon_{n,K}(y)$ are the points which lie in the edge $e^n(i)=[\tilde{a}_i,\tilde{b}_i]$ at distances $\alpha_x \frac{d^{(n,K)}_{\text{res}}(\tilde{a}_i,\tilde{b}_i)}{d_{\frak{T}}(a_i,b_i)}$ (reps. $\alpha_y\frac{d^{(n,K)}_{\text{res}}(\tilde{a}_i,\tilde{b}_i)}{d_{\frak{T}}(a_i,b_i)}$) of $\tilde{a}_i$. Therefore, 
 \[d^{(n,K)}_{\frak{T},\text{res}}(x,y)=|\alpha_x-\alpha_y| \frac{d^{(n,K)}_{\text{res}}(\tilde{a}_i,\tilde{b}_i)}{d_{\frak{T}}(a_i,b_i)}= d_{\frak{T}}(x,y)  \frac{d^{(n,K)}_{\text{res}}(\tilde{a}_i,\tilde{b}_i)}{d_{\frak{T}}(a_i,b_i)}. \] 
  By \eqref{eq:srt2}, 
  \[
 \abs{\frac{d^{(n,K)}_{\text{res}}(\tilde{a}_i,\tilde{b}_i)}{ d_{\frak{T}}(a_i,b_i)}-\sigma_d\rho}\stackrel{n\to\infty}{\to}0, \quad \mathbf{P}^{(K)}\text{-almost surely.}\] 
 This proves our claim in this case. In general, take $x,y$ lying in different line segments, the same argument as above gives that 
 \begin{equation}\label{eq:lipschitznorm}
 \min_{i=1,\dots,s} \frac{d^{(n,K)}_{\text{res}}(\tilde{a}_i,\tilde{b}_i)}{d_{\frak{T}}(a_i,b_i)}\leq \frac{d^{(n,K)}_{\frak{T}}(x,y)}{ d_{\frak{T}}(x,y)}\leq\max_{i=1,\dots,s} \frac{d^{(n,K)}_{\text{res}}(\tilde{a}_i,\tilde{b}_i)}{  d_{\frak{T}}(a_i,b_i)} .
 \end{equation}
 Therefore again, the result follows from \eqref{eq:srt2}.
 We have established \eqref{eq:homogenizationofR} which, as we saw in~\eqref{aeq:tildes!}, implies \eqref{eq:tildes!}. 
 This finishes the first part of the proof.
 
 \emph{Second step: From local times with respect to non-normalized measures to local times with respect to normalized measures.}

By virtue of Lemma \ref{lem:localtimeexpression}, we know that $\bar{B}^{(n,K)}$ is the Brownian motion in $(\frak{T}^{(K)}, d_{\frak{T},\text{res}}^{(n,K)}, \lambda^{(n,K)}_{\frak{T},\text{res}})$ and $\bar{L}^{(n,K)}$ is its local time with respect to $\lambda^{(K)}_{\frak{T},\text{res}}$. We start by relating $\tilde{L}^{(n,K)}$ to $\bar{L}^{(n,K)}$. First, observe that it follows from the definition of Brownian motion give above Proposition \ref{prop_def_process} that 
\begin{equation}\label{eq:supertramp}
(\bar{B}^{(n,K)}_t)_{t\geq0}\stackrel{d}{=}(\tilde{B}^{(n,K)}_{\Lambda_{\text{res}}^{(n,K)}t})_{t\geq0}.
\end{equation}
Moreover, we have that for any $A$ Borelian of $\frak{T}^{(K)}$,
\begin{equation}\label{eq:supertramp1}
\text{Leb}\{s\leq t:\bar{B}^{(n,K)}_s\in A\}=\int_A\bar{L}^{(n,K)}_t\lambda^{(n,K)}_{\text{res}}(dx).
\end{equation}
On the other hand
\begin{equation}\label{eq:supertramp2}
\begin{aligned}
\text{Leb}\{s\leq t:\bar{B}^{(n,K)}_s\in A\}&\stackrel{d}{=}\text{Leb}\{s\leq t:\tilde{B}^{(n,K)}_{\Lambda_{\text{res}}^{(n,K)}s}\in A\}\\
&=(\Lambda_{\text{res}}^{(n,K)})^{-1}\text{Leb}\{s\leq\Lambda_{\text{res}}^{(n,K)} t:\tilde{B}^{(n,K)}_{s}\in A\}\\
&=(\Lambda_{\text{res}}^{(n,K)})^{-1}\int_A\tilde{L}^{(n,K)}_{\Lambda_{\text{res}}^{(n,K)} t} \Lambda^{(n,K)}_{\text{res}}\lambda^{(n,K)}_{\text{res}}(dx)\\
&=\int_A\tilde{L}^{(n,K)}_{\Lambda_{\text{res}}^{(n,K)} t} \lambda^{(n,K)}_{\text{res}}(dx),
\end{aligned}
\end{equation}
where the first equality follows from $\eqref{eq:supertramp}$, the second equality is trivial and the third equality follows because $\tilde{L}^{(n,K)}$ is the local time of $\tilde{B}^{(n,K)}$ with respect to $\Lambda^{(n,K)}_{\text{res}}\lambda^{(n,K)}_{\text{res}}$. It follows from \eqref{eq:supertramp1} and \eqref{eq:supertramp2} that
\begin{equation}\label{eq:fafifein}
(\bar{L}^{(n,K)}_t(x))_{x\in\frak{T}^{(K)},t\geq0}\stackrel{d}{=}(\tilde{L}^{(n,K)}_{\Lambda^{(n,K)}_{\text{res}} t}(x))_{x\in\frak{T}^{(K)},t\geq0}.
\end{equation}
Similarly, let $\hat{B}^{(K)}$ is the Brownian motion in $(\frak{T}^{(K)},\sigma_d\rho d_{\frak{T}}, \lambda_{\frak{T}}^{(K)})$ and $\hat{L}^{(K)}$ be its local time with respect to $\lambda^{(K)}_{\frak{T}}$. By the same reasoning leading to \eqref{eq:fafifein} we get that
\begin{equation}\label{eq:fafifeintkg}
(\hat{L}^{(K)}_t(x))_{x\in\frak{T}^{(K)},t\geq0} \stackrel{d}{=} (\tilde{L}^{(K)}_{\sigma_d\rho \Lambda^{(K)}t})_{x\in\frak{T}^{(K)},t\geq 0}.
\end{equation} 
 On the other hand, by Lemma \ref{lem:graphspatialresistancetreeconvergence} we have that 
 \begin{equation}\label{eq:susss}
 \Lambda^{(n,K)}_{\text{res}}\stackrel{n\to\infty}{\to}\sigma_d\rho\Lambda^{(K)}.
 \end{equation} From \eqref{eq:susss}, \eqref{eq:tildes!}, \eqref{eq:fafifein} and \eqref{eq:fafifeintkg}, it can be deduced that 
 \begin{equation}\label{eq:fafifo}
( \bar{L}^{(n,K)}_t(x))_{x\in\frak{T}^{(K)},t\in[0,R]}\stackrel{n\to\infty}{\to}(\hat{L}^{(K)}_t(x))_{x\in\frak{T}^{(K)},t\in[0,R]}
 \end{equation}
 in distribution in $C(\frak{T}^{(K)}\times [0,R],\R_+)$, for all $R>0$.
 
 It remains to relate $\hat{L}^{(K)}$ to $L^{(K)}$. By the definition of Brownian motion given above Proposition \ref{prop_def_process} we have that 
\begin{equation}
(B^{(K)}_t)_{t\geq0}\stackrel{d}{=}(\hat{B}^{(n,K)}_{\sigma_d\rho t})_{t\geq0}.
\end{equation}
Moreover, we have that, for any Borelian $A$ of $\frak{T}^{(K)}$,
\begin{equation}\label{eq:channnes}
\text{Leb}\{s\leq t: B^{(K)}_s\in A\}=\int_A L^{(K)}_t(x) \lambda_{\frak{T}}^{(K)},
\end{equation}
 while, on the other hand,
\begin{equation}\label{eq:channnes2}
\begin{aligned}
\text{Leb}\{s\leq t:B^{(K)}_s\in A\}&\stackrel{d}{=}\text{Leb}\{s\leq t:\hat{B}^{(n,K)}_{\sigma_d\rho s}\in A\}\\
&=(\sigma_d\rho)^{-1}\text{Leb}\{s\leq\sigma_d\rho t:\hat{B}^{(n,K)}_s\in A\}\\
&=(\sigma_d\rho)^{-1}\int_A\hat{L}^{(n,K)}_{\sigma_d\rho t}(x) \lambda^{(K)}_{\frak{T}}(dx),
\end{aligned}
\end{equation}
where the first equality follows from \eqref{eq:channnes}.
From \eqref{eq:channnes2},
 \begin{equation}\label{eq:theultimatefafifo}
 (\hat{L}^{(K)}_t(x))_{x\in\frak{T}^{(K)},t\geq0}\stackrel{d}{=} (\sigma_d\rho L^{(K)}_{(\sigma_d\rho)^{-1}t})_{x\in\frak{T}^{(K)},t\geq0}.
 \end{equation}
\noindent The claim of the lemma follows by displays \eqref{eq:fafifo}, \eqref{eq:defofbarL} and \eqref{eq:theultimatefafifo}.

 \end{proof}

\subsubsection{Second part of the proof of Proposition~\ref{l:convergenceoflocaltimes}}\label{proof_prop61a}

In this section we will prove the first two statements in Proposition \ref{l:convergenceoflocaltimes}, i.e.~displays \eqref{eq:convergenceofedgelocaltimes} and \eqref{eq:convergenceofvertexlocaltimes}. Recall from Section \ref{proof_prop61a} that display \eqref{eq:conv3} in Lemma \ref{l:convergenceoflocaltimes} follows from \eqref{eq:convergenceofedgelocaltimes} and Lemma \eqref{lem:abstractcoupling}. Therefore, Lemma \ref{l:convergenceoflocaltimes} will be proved after we have showed \eqref{eq:convergenceofedgelocaltimes} and \eqref{eq:convergenceofvertexlocaltimes}.

\begin{proof}

 We will split the proof in two parts. First we will do the proof of display \eqref{eq:convergenceofedgelocaltimes} and then, as a consequence, we will obtain \eqref{eq:convergenceofvertexlocaltimes}.
 
 \emph{Proof of display \eqref{eq:convergenceofedgelocaltimes}}:\\
Recall the definition of $\Upsilon_{n,K}$ from \eqref{eq:defofupsilon}.
It is easy to see that, as a consequence of Lemma \ref{lem:graphspatialresistancetreeconvergence}, the Lipschitz norm of $\Upsilon_{n,k}^{-1}$ converges to $(\sigma_d\rho)^{-1}$ as $n\to\infty$. Therefore $d^{(n,K)}_{\text{res}}(x,y)\leq \delta$ implies $d_{\frak{T}}(\Upsilon_{n,K}^{-1}(x),\Upsilon_{n,K}^{-1}(y)) \leq C \delta$ for $n$ large enough for some $C$ which is independent of $n$.
Therefore 
\begin{equation}\label{eq:lipschitz}
\sup_{\stackrel {x,y:}{d^{(n,K)}_{\text{res}}(x,y)}\leq\delta}\abs{L^{(n,K)}_t(x)-L^{(n,K)}_t(y)}\leq\sup_{\stackrel{x,y:}{d_{\frak{T}}(\Upsilon_{n,K}^{-1}(x),\Upsilon_{n,K}^{-1}(y))}\leq C\delta }\abs{L^{(n,K)}_t(x)-L^{(n,K)}_t(y)}.
\end{equation}
Let us set
\begin{equation}\label{eq:defoftildelk}
\tilde{L}^{(K)}_t(x):=\sigma_d\rho L^{(K)}_{(\rho\sigma_d)^{-1}t}(x),\qquad \forall x\in\frak{T}^{(K)},t\geq0.
\end{equation}
By \eqref{eq:lipschitz},
\begin{equation}\label{eq:arzelaascoliforlocaltime}
\begin{aligned}
&\mathbf{P}^{(K)}\left[ \sup_{\stackrel {x,y:}{d^{(n,K)}_{\text{res}}(x,y)}\leq\delta}\abs{L^{(n,K)}_t(x)-L^{(n,K)}_t(y)}\geq\epsilon \right]\\
 \leq &\mathbf{P}^{(K)}\left[ \sup_{\stackrel {x,y:}{d_{\frak{T}}(\Upsilon_{n,K}^{-1}(x),\Upsilon_{n,K}^{-1}(y))}\leq C\delta}\abs{L^{(n,K)}_t(x)-L^{(n,K)}_t(y)}\geq\epsilon \right]\\
\leq &\mathbf{P}^{(K)}\left[ \sup_{\stackrel {x,y:}{d_{\frak{T}}(\Upsilon_{n,K}^{-1}(x),\Upsilon_{n,K}^{-1}(y))}\leq C\delta}\abs{\tilde{L}^{(K)}_t(\Upsilon_{n,K}^{-1}(x))-\tilde{L}^{(K)}_t(\Upsilon_{n,K}^{-1}(y))}\geq\epsilon/3 \right]\\
&\qquad \qquad \qquad+2\mathbf{P}^{(K)}\left[\sup_{x\in \frak{T}^{(n,K)}}\abs{L_t^{(n,K)}(x)-\tilde{L}^{(K)}_t(\Upsilon_{n,K}^{-1}(x))} \geq \epsilon/3 \right] .
\end{aligned}
\end{equation}

By Lemma \ref{lem:abstractcoupling},  
\[\mathbf{P}^{(K)}\left[\sup_{x\in \frak{T}^{(n,K)}}\abs{L_t^{(n,K)}(x)-\tilde{L}^{(K)}_t(\Upsilon_{n,K}^{-1}(x))} \geq \epsilon/3 \right] 
\stackrel{n\to\infty}{\to}0.\]
Also, since $\tilde{L}^{(K)}_t$ is continuous in the space variable, we have that 
\[\lim_{\delta\to0}\limsup_{n\to\infty}\mathbf{P}^{(K)}\left[ \sup_{\stackrel {x,y:}{d_{\frak{T}}(\Upsilon_{n,K}^{-1}(x),\Upsilon_{n,K}^{-1}(y))}\leq C\delta}\abs{\tilde{L}^{(K)}_t(\Upsilon_{n,K}^{-1}(x))-\tilde{L}^{(K)}_t(\Upsilon_{n,K}^{-1}(y))}\geq\epsilon/3 \right]=0.\]
The two displays above and \eqref{eq:arzelaascoliforlocaltime} imply that 
\begin{equation}\label{eq:modulusofcontinuityforltilde}
\lim_{\delta\to0}\limsup_{n\to\infty}\mathbf{P}^{(K)}\left[ \sup_{\stackrel {x,y:}{d^{(n,K)}_{\text{res}}(x,y)}\leq\delta}\abs{L^{(n,K)}_t(x)-L^{(n,K)}_t(y)}\geq\epsilon \right]=0.
\end{equation}

Let $K$ be fixed. By Lemma \ref{lem:graphspatialresistancetreeconvergence} we have that, for each $n\in\N$ and $\delta>0$, we can choose a $\frac{\delta}{2}$-net $A_{\delta,n}=(x_i)_i$, $x_i\in \frak{T}^{(n,K)}$ of $(\frak{T}^{(n,K)},d^{(n,K)}_{\text{res}})$ (i.e., all points of $\frak{T}^{(n,K)}$ are at $d^{(n,K)}$-distance smaller than $\frac{\delta}{2}$ of some $x_i$.) such that, for each $\delta$,  $\abs{A_{\delta,n}}$ is bounded uniformly in $n$. Therefore, recalling that by virtue of Lemma \ref{lem_no_macro_res}, we have that for $n$ large enough $\mathbf{P}^{(K)}[\max_{e\in E^\ast(\T^{(n,K)})}n^{-1/2}\reff^{G_n}(e)\geq \delta/4]\to0$ as $n\to\infty$. Therefore
\begin{equation}\label{eq:doomgetthecash}
\begin{aligned}
&\mathbf{P}^{(K)}\left[\sup_{x\in \frak{T}^{(n,K)}}\abs{n^{-1/2} \reff^{G_n}(e(x))l^{(n,K)}_{s}(e(x))-L^{(n,K)}_{s}(x)}\geq \epsilon\right]\\
\leq &\mathbf{P}^{(K)}\left[\sup_{x,e':d^{(n,K)}_{\text{res}}(e(x),e')\leq \delta}  \abs{n^{-1/2}\reff^{G_n}(e(x))l^{(n,K)}_{s}(e(x))-n^{-1/2}\reff^{G_n}(e')l^{(n,K)}_{s}(e')} \geq \epsilon /3\right]\\
&\qquad\qquad+ \sum_{x_i\in A_n^\delta}  \mathbf{P}^{(K)}\left[\abs{ n^{-1/2} \reff^{G_n}(e(x_i))l^{(n,K)}_{s}(e(x_i))-L^{(n,K)}_{s}(x_i)} \geq \epsilon /3\right]\\
&\qquad\qquad+ \mathbf{P}^{(K)}\left[\sup_{\stackrel{x,y\in \T^{(n,K)}}{d^{(n,K)}_{\text{res}}(x,y)}\leq \delta} \abs{ L^{(n,K)}_{s}(x)-L^{(n,K)}_{s}(y)} \geq \epsilon /3\right]+o(1),
\end{aligned}
\end{equation}
where $o(1):=\mathbf{P}^{(K)}[\max_{e\in E^\ast(\T^{(n,K)})}n^{-1/2}\reff^{G_n}(e)\geq \delta/4]\to0$ as $n\to\infty$.
Let $\eta>0$ be arbitrary, the proof will follow after we have showed that
\begin{equation}\label{eq:cowboydoom}
\limsup_{n\to\infty}\mathbf{P}^{(K)}\left[\sup_{x\in \frak{T}^{(n,K)}}\abs{n^{-1/2} \reff^{G_n}(e(x))l^{(n,K)}_{s}(e(x))-L^{(n,K)}_{s}(x)}\geq \epsilon\right]\leq \eta.
\end{equation}
By Lemma \ref{lem:modulusofcontinuity} we have that there exists $\delta_1>0$  such that for all $\delta<\delta_1$ we have that 
\[
\limsup_{n\to\infty}\mathbf{P}^{(K)}\left[\sup_{x,e':d^{(n,K)}_{\text{res}}(e(x),e')\leq \delta}  \abs{n^{-1/2}\reff^{G_n}(e(x))l^{(n,K)}_{s}(e(x))-n^{-1/2}\reff^{G_n}(e')l^{(n,K)}_{s}(e')} \geq \epsilon /3\right]\leq \frac{\eta}{2}.
\]
On the other hand, by \eqref{eq:modulusofcontinuityforltilde}, there exists $\delta_2>0$ such that  for all $\delta<\delta_2$ we have that 
\[
\limsup_{n\to\infty}\mathbf{P}^{(K)}\left[\sup_{\stackrel{x,y\in \T^{(n,K)}}{d^{(n,K)}_{\text{res}}(x,y)}\leq \delta} \abs{ L^{(n,K)}_{s}(x)-L^{(n,K)}_{s}(y)} \geq \epsilon /3\right]\leq \frac{\eta}{2}.
\]
 Finally, for any fixed $\delta$, $\abs{A_{\delta,n}}$ remains bounded in $n$ and therefore, by virtue of Lemma \ref{lem:convergenceoflocaltimessingle}
 \[
 \limsup_{n\to\infty}\sum_{x_i\in A_n^\delta}  \mathbf{P}^{(K)}\left[\abs{ n^{-1/2} \reff^{G_n}(e(x_i))l^{(n,K)}_{s}(e(x_i))-L^{(n,K)}_{s}(x_i)} \geq \epsilon /3\right]=0.
 \]
 Taking $\delta=\min\{\delta_1,\delta_2\}$ in \eqref{eq:doomgetthecash} and using the last three displays we get \eqref{eq:cowboydoom}.
  This finishes the proof of \eqref{eq:convergenceofedgelocaltimes}.

\emph{Proof of display \eqref{eq:convergenceofvertexlocaltimes}:}\\ 
Since $l^{(n,K,\leftarrow)}_{t}(e(x))$ equals either $l^{(n,K)}_{t}(e(x))/2$ or $(l^{(n,K)}_{t}(e(x))-1)/2$, it follows from display \eqref{eq:convergenceofedgelocaltimes} and Lemma \ref{lem_no_macro_res} that
\begin{equation}
\lim_{n\to\infty}\mathbf{P}^{(K)}\left[\sup_{v\in V^\circ(\T^{(n,K)})}\abs{  \frac{d^{(n,K)}_{\text{res}}(e^+(v))}{2}l^{(n,K,\leftarrow)}_{t}(e^+(v))-L^{(n,K)}_{t}(v)}\geq \epsilon \right]=0.
\end{equation}
Moreover, since $\frac{d_{\text{res}}^{(n,K)}(e^-(v))}{d_{\text{res}}^{(n,K)}(e^-(v))+d_{\text{res}}^{(n,K)}(e^+(v))}\leq1$, it follows from the display above that
\begin{equation}
\begin{aligned}
\lim_{n\to\infty}\mathbf{P}^{(K)}\bigg[&\sup_{x\in V^\circ(\T^{(n,K)})}\bigg|\frac{d_{\text{res}}^{(n,K)}(e^-(v))d^{(n,K)}_{\text{res}}(e^+(v))}{2(d_{\text{res}}^{(n,K)}(e^-(v))+d_{\text{res}}^{(n,K)}(e^+(v)))} l^{(n,K,\leftarrow)}_{t}(e^+(v))\\
&-\frac{d_{\text{res}}^{(n,K)}(e^-(v))}{d_{\text{res}}^{(n,K)}(e^-(v))+d_{\text{res}}^{(n,K)}(e^+(v))}L^{(n,K)}_{t}(v)\bigg|\geq \epsilon \bigg]=0.
\end{aligned}
\end{equation}
An analogous argument yield that
\begin{equation}
\begin{aligned}
\lim_{n\to\infty}\mathbf{P}^{(K)}\bigg[&\sup_{x\in V^\circ(\T^{(n,K)})}\bigg|\frac{d_{\text{res}}^{(n,K)}(e^-(v))d^{(n,K)}_{\text{res}}(e^+(v))}{2(d_{\text{res}}^{(n,K)}(e^-(v))+d_{\text{res}}^{(n,K)}(e^+(v)))} l^{(n,K,\rightarrow)}_{t}(e^-(v))\\
&-\frac{d_{\text{res}}^{(n,K)}(e^+(v))}{d_{\text{res}}^{(n,K)}(e^+(v))+d_{\text{res}}^{(n,K)}(e^+(v))}L^{(n,K)}_{t}(v)\bigg|\geq \epsilon \bigg]=0.
\end{aligned}
\end{equation}
Summing the terms inside the absolute value in the two displays above, and recalling \eqref{eq:vertupdown}, we get display \eqref{eq:convergenceofvertexlocaltimes}.
\end{proof}

 \section{Asymptotic linearity of the time change}\label{s:AnKislinear}
 
 The goal of this section is to prove Proposition~\ref{l:Ankislinear}.
 
\subsection{Averaged time change}\label{sect:averagedtimechange}
We introduce an averaged version  $\hat{A}^{(n,K)}$ of the time change  $A^{(n,K)}$ (which was introduced below~\eqref{eq:coupling}) which is more tractable.
Let $\tau(x)$ be a random variable having the distribution of $A^{(n,K)}_{j+1}-A^{(n,K)}_j$ conditioned on $J^{(n,K)}_j=x$. That is, $\tau(x)$ has the law of $\min\{j\geq0: X^{G_n}_j\in V^\ast(\T^{(n,K)})\setminus\{ x \}\}$ conditioned on $X^{G_n}_0=x$. 
 Let $\hat{A}^{(n,K)}(0):=0$ and
 \begin{equation}\label{eq:defofhatank}
 \hat{A}^{(n,K)}(m):=\sum_{x\in V^\ast(\T^{(n,K)})}E^{G_n}[\tau(x)]l^{(n,K,\text{vert})}_{m}(x).
 \end{equation} 
 where $l^{(n,K,\text{vert})}_m$ is as in \eqref{eq:defoflvert}. We will need to introduce another approximation of $A^{(n,K)}$. We proceed to prepare its definition.
 For any edge $e=(e^-,e^+)\in E^*(\T^{(n,K)})$ (where $e^-$ is an ancestor of $e^+$ in $\T^{(n,K)}$), let $G_n(e)$ be the subgraph of $G_n$ whose set of edges is
 \begin{equation}\label{eq:defofgn1}
 E(G_n(e)):=\{(x,y)\in E(G_n): \pi^{(n,K)}(x)=e^- \text{ and }\pi^{(n,K)}(x)\neq x \}
  \end{equation}
  and its set of vertices is given by 
  \begin{equation}\label{eq:defofgn2}
  V(G_n(e)):=\{v\in V(G_n):\exists e=(e^-,e^+) \in E(G_n(e)) \text{ with } v\in \{e^-,e+\}\}.
  \end{equation} 
  
  That is, $G_n(e)$ is the sausage corresponding to the edge $e$ (see Figure 4) and we also want to emphasize the link between $E(G_n(e))$ and the measure $\mu$ introduced in Section~\ref{sect_mu}.

  We are ready to introduce the second approximation of $A^{(n,K)}$ as
\[\tilde{A}^{(n,K)}(m):=\sum_{e\in E^\ast(\T^{(n,K)})} 2 \reff^{G_n}(e) \abs{E(G_n(e))} l^{(n,K,\leftrightarrow)}_m(e).\]
Proposition~\ref{l:Ankislinear} is a direct consequence of the three following results:
 \begin{lemma}\label{lem:aproximationofank} For each $t\geq0$ and $\epsilon>0$,
 \[\lim_{K\to\infty}\limsup_{n\to\infty}\mathbf{P}^{(K)}\left[n^{-3/2}\abs{A^{(n,K)}(m(t))-\hat{A}^{n,K}(m(t))}\geq\epsilon\right]=0.\]
  \end{lemma}

  \begin{lemma}\label{lem:commutetime}
 For all $t\geq 0,\epsilon>0$ and $K>0$,
\[
\limsup_{n\to\infty}\mathbf{P}^{(K)}\left[n^{-3/2} \abs{\tilde{A}^{(n,K)}(m(t))-\hat{A}^{(n,K)}(m(t))}\geq \epsilon\right]=0.
\]
\end{lemma}
 \begin{lemma}\label{lem:hatankislinear}
 For each $t\geq0$ and $\epsilon>0$, 
\[\lim_{K\to\infty}\limsup_{n\to\infty}\mathbf{P}^{(K)}\left[\abs{n^{-3/2}\tilde{A}^{n,K}(m(t))-\nu t}\geq \epsilon \right]=0,\]
where $\nu$ is that of condition $(V)_\nu$. 
\end{lemma}

\subsubsection{Lemma~\ref{lem:aproximationofank}}

For any $v\in V^\ast(\T^{(n,K)})$, consider $G_n(v)$ the subgraph of $G_n$ consisting of all the vertices of $G_n$  which can be attained from $v$ through edges of $G_n$  without crossing $V^\ast(\T^{(n,K)})\backslash\{v\}$.
 For the proof of Lemma \ref{lem:aproximationofank} we will need the following auxiliary result which states that the time spent in transitions involving vertices which are not in $V^\circ(\T^{(n,K)})$ is negligible.

Let $e\in E^\ast(\T^{(n,K)})$. Recall the definition of $G_n(e)$ from \eqref{eq:defofgn1} and \eqref{eq:defofgn2}. Let $\tau^{\rightarrow}(e)$ (resp. $\tau^{\leftarrow}(e)$) the time spent by a random walk on $G_n(e)$ on reaching $v_2$ starting from $v_1$ (resp. $v_1$ starting from $v_2$).
Let $\tau (e)$ be a random variable having the distribution of $\tau^{\rightarrow}(e)+\tau^{\leftarrow}(e)$, assuming that $\tau^{\rightarrow}(e)$ is independent from $\tau^{\leftarrow}(e)$.
 
 Consider $\tilde{I}$ the subset of $\frak{T}^{(n,K)}$ corresponding to the edges in $E(\T^{(n,K)})\backslash E^\ast(\T^{(n,K)})$.
 \begin{lemma}\label{lem:vcircle} For all $\epsilon>0$ $t\geq 0$,
 \begin{enumerate}
\item\[\lim_{K\to\infty} \limsup_{n\to\infty} \mathbf{P}^{(K)}\left[n^{-3/2} \sum_{v\in V^\ast(\T^{(n,K)})\backslash V^\circ(\T^{(n,K)})} l^{(n,K,\text{vert})}_{m(t)}(v)E^{G_n}[\tau(v)]\geq \epsilon\right]=0.\]
\item \[\lim_{K\to\infty} \limsup_{n\to\infty} \mathbf{P}^{(K)}\left[n^{-3/2} \sum_{e\in E^\ast(\T^{(n,K)})\backslash E^\circ(\T^{(n,K)}))} l^{(n,K,\rightarrow)}_{t}(e)E^{G_n}[\tau(e)]\geq \epsilon\right]=0.\]
\item 
\[\limsup_{n\to\infty} {\bf P}^{(K)}\left[\int_{\tilde{I}} L^{(n,K)}_t(x) \lambda_{\text{res}}^{(n,K)}(dx)\geq \epsilon\right]=0.
\]
\end{enumerate}
 \end{lemma}
 \begin{proof}[Proof of Lemma \ref{lem:vcircle}] 
 \emph{Proof of display (1):}
 
 Note that vertices in  $V^\ast(\T^{(n,K)})\backslash V^\circ(\T^{(n,K)})$ appear only at branching points and leaves of $\T^{(n,K)}$. For every branching point of $\T^{(n,K)}$ there are $3$ vertices of $V^\ast(\T^{(n,K)})\backslash V^\circ(\T^{(n,K)})$ (the vertices of the triangle involved in the corresponding star-triangle transformation) while for every leaf there is only $1$ vertex of $V^\ast(\T^{(n,K)})\backslash V^\circ(\T^{(n,K)}))$. Therefore, since $\frak{T}^{(K)}$ has a finite number of branching points and leaves, it follows from \eqref{eq:couplinggeometry} that, for $K$ fixed, the set $V^\ast(\T^{(n,K)})\backslash V^\circ(\T^{(n,K)})$ is finite and its cardinality is uniformly bounded on $n$. Hence, to prove part $(1)$ it suffices to show that
 \begin{equation}\label{eq:singlesingle}
 \lim_{K\to\infty}\limsup_{n\to\infty}{\bf P}^{(K)}\left[n^{-3/2} l^{(n,K,\text{vert})}_{m(t)}(v)E^{G_n}[\tau(v)]\geq \epsilon \right]=0
 \end{equation}
 for any $v\in V^\ast(\T^{(n,K)})\backslash V^\circ(\T^{(n,K)})$ leaf or branching point.
 
  First, we will do  the case when $v$ corresponds to a branching point.
 In this case there will be three edges $e_1=(v,v_1),e_2=(v,v_2),e_3=(v,v_3)$ of $V(G^{(n,K)})$ that are adjacent to $v$ with $e_2$ and $e_3$ in $E(G^{(n,K)})\backslash E^\ast(\T^{(n,K)})$ and $e_1$ in $E^\ast(\T^{(n,K)})$, see Figure 7.

\begin{figure}
  \includegraphics[width=0.6\linewidth]{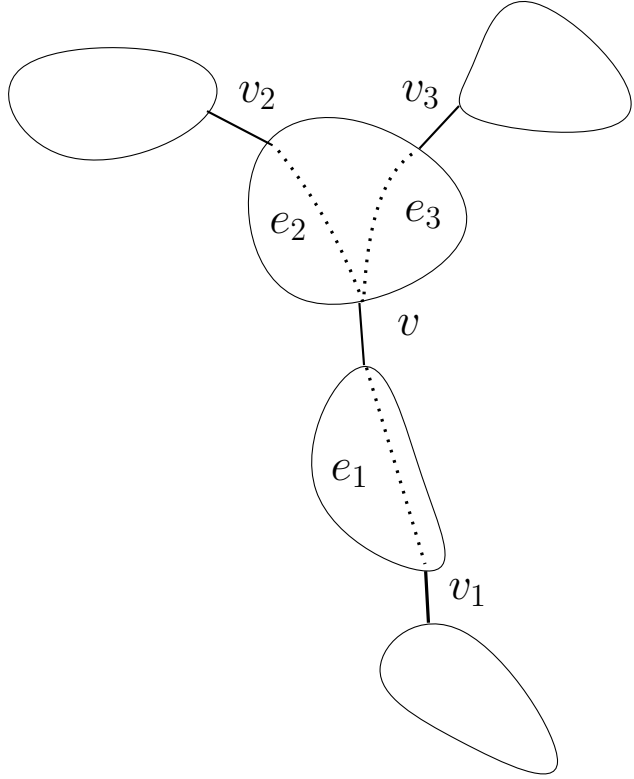}
  \caption{The neighbourhood of branching points}
\end{figure}

Consider the graph obtained by identifying the points $v_1,v_2,v_3$ in $G_n(v)$. Call $v'$ be the vertex obtained by the identification of $v_1,v_2,v_3$ and $\tau(v\leftrightarrow v')$ be the commute time between $v$ and $v'$ (i.e., the time that it takes to a random walk on $G_n(v)$ to go from $v$ to $v'$ and then to return to $v$). Clearly, 
\begin{equation}\label{eq:ctf}
E^{G_n}[\tau(v)]\leq E^{G_n}[\tau(v\leftrightarrow v')]=2|E(G_n(v))|\reff^{G_n}(v,v'),
\end{equation}
where the last equality follows by the commute time formula.
Therefore, \eqref{eq:singlesingle} will follow if we prove that
\begin{equation}\label{eq:boundonlocaltimetoneglect}
\lim_{M\to\infty}\limsup_{n\to\infty}{\bf P}^{(K)}\left[  n^{-1/2}\reff^{G_n}(v,v') l^{(n,K,\text{vert})}_{m(t)}(v)\geq M \right]=0
 \end{equation}
and 
 \begin{equation}\label{eq:boundoncardinalitytoneglect}
 \lim_{K\to\infty}\limsup_{n\to\infty}{\bf P}^{(K)}\left[ n^{-1}|E(G_n(v)|\geq \epsilon\right]=0.
 \end{equation}
 We will first do the proof of \eqref{eq:boundonlocaltimetoneglect}.
 By Rayleigh's monotonicity principle,
 \begin{equation}\label{eq:afterdeath}
 \reff^{G_n}(v,v')\leq \min_{i=1,2,3}\{\reff^{G_n}(v,v_i)\}.
 \end{equation}
Let us assume without loss of generality that the orientation of $e_i,i=1,2,3$ that $l^{(n,K,\rightarrow)}_t(e_i)$ counts is the one pointing away from $v$.
Therefore
\begin{equation}\label{eq:ltdbp}
 \abs{l^{(n,K,\text{vert})}_{m(t)}(v)-l^{(n,K,\rightarrow)}_t(e_1)-l^{(n,K,\rightarrow)}_t(e_2)-l^{(n,K,\rightarrow)}_t(e_3)}\leq 1,
 \end{equation}
  It follows from Lemma \ref{l:tailoflocaltimetrifurcation} that, for all $K\in\N$,
 \[\lim_{M\to\infty}\limsup_{n\to\infty}{\bf P}^{(K)}\left[n^{-1/2}\reff^{G_n}(e_i) l^{(n,K,\rightarrow)}_{t}(e_i)\geq M \right] =0,\]
for all $\epsilon>0$, $i=1,2,3$. 
It is not hard to see from the display above that 
 \[
 \begin{aligned}
 \lim_{M\to\infty}\limsup_{n\to\infty}{\bf P}^{(K)}\Bigg[&3n^{1/2} \min_{i=1,2,3}\{\reff^{G_n}(v,v_i)\}\times\\
 & (l^{(n,K,\rightarrow)}_{t}(e_1)+l^{(n,K,\rightarrow)}_{t}(e_2)+l^{(n,K,\rightarrow)}_{t}(e_3))\geq M\Bigg] =0,
 \end{aligned}
 \]
 for all $\epsilon>0$.
 Therefore, using displays \eqref{eq:afterdeath} and \eqref{eq:ltdbp} we can deduce \eqref{eq:boundonlocaltimetoneglect}.
 

Now we focus on the proof of \eqref{eq:boundoncardinalitytoneglect}. Recall that $e_1,e_2,e_3$ are the three edges of $V(G^{(n,K)})$ which are adjacent to $v$, with $e_2$ and $e_3$ in $E(G^{(n,K)})\backslash E^\ast(\T^{(n,K)})$ and $e_1$ in $E^\ast(\T^{(n,K)})$, see Figure 7. 
By Lemma \ref{lem_no_macro_res}, we have that
  \[\lim_{n\to\infty}{\bf P}^{(K)}\left[ n^{-1/2} \max\{\reff^{G_n}(e_1),\reff^{G_n}(e_2),\reff^{G_n}(e_3)\} \geq\eta \right]=0,\]
  for all $\eta>0$.
Let $e_4,e_5,e_6$ be the edges of $E(\T^{(n,K)})\backslash E^\ast(\T^{(n,K)})$ introduced by the star-triangle transformation corresponding to $v$. Without loss of generality we can assume that $e_4$ is the one which is adjacent to $v$, $e_5$ is adjacent to $v_2$ and $e_6$ is adjacent to $v_3$. Then, by the star-triangle transformation, we have that $d^{(n,K)}_{\text{res}}(e_2)=n^{-1/2}\reff^{G_n}(e_2)=n^{-1/2}\reff^{G_n}(e_4)+n^{-1/2}\reff^{G_n}(e_5)$ and $d^{(n,K)}_{\text{res}}(e_3)=n^{-1/2}\reff^{G_n}(e_3)=n^{-1/2}\reff^{G_n}(e_4)+n^{-1/2}\reff^{G_n}(e_6)$.
   Therefore by the display above, we have that
    \begin{equation}\label{eq:bgbdt}
    \lim_{n\to\infty}{\bf P}^{(K)}\left[   \lambda^{(n,K)}_{\text{res}}(e_1\cup e_4\cup e_5\cup e_6) \geq \eta\right]=0,
    \end{equation}
    for all $\eta>0$.
  Hence, it is not hard to see from condition (R) that
  \[\lim_{n\to\infty}{\bf P}^{(K)}\left[ \lambda^{(n,K)}(e_1\cup e_4\cup e_5 \cup e_6) \geq \eta\right]=0,\]
  for all $\eta >0$.
  The display above states that the Lebesgue measure of the portion of $\frak{T}^{(n,K)}$ which corresponds to $G_n(v)$ (i.e., $e_1\cup e_4\cup e_5\cup e_6$) is negligible. Since our sequence of graphs satisfy condition (V), this implies that 
 \[
 \lim_{n\to\infty} {\bf P}^{(K)}[\mu^{(n,K)}(e_1\cup e_4\cup e_5\cup e_6)\geq \eta]=0,
 \]
 for all $\eta>0$.
   Recalling that $\mu^{(n,K)}(e_1\cup e_4\cup e_5\cup e_6)$ corresponds to the edge cardinality of $G_n(v)$ scaled by $n^{-1}$, we get  \eqref{eq:boundoncardinalitytoneglect}. This shows \eqref{eq:singlesingle} for the case when $v$ is a branching point.

 To finish the proof of claim (1) of the lemma, it remains to show \eqref{eq:singlesingle} when $v$ is a leaf of $\T^{(n,K)}$. Let $e=(v',v)$ be the unique edge of $\T^{(n,K)}$ which is adjacent to $v$. By the commute time formula we have that
 \[
 E^{G_n}[\tau(v)]\leq E^{G_n}[\tau(v\leftrightarrow v')]\leq2 \reff^{G_n}(e)\abs{E(G_n(v))}.
 \]
 By arguments analogous to those leading to \eqref{eq:boundonlocaltimetoneglect} we find that
 \[
\lim_{M\to\infty}\limsup_{n\to\infty} P^{G_n}[n^{-1/2}l^{(n,K,\text{vert})}_{m(t)}(e)\reff^{G_n}(e)\geq M]=0.
 \]
 On the other hand, imitating the deduction of \eqref{eq:boundoncardinalitytoneglect} we get that 
\[\lim_{K\to\infty}\limsup_{n\to\infty}\mathbf{P}^{(K)}\left[n^{-1}\abs{E(G_n(v))}\geq \epsilon\right]=0,\]
for all $\epsilon>0$. We can use the three displays above to deduce \eqref{eq:singlesingle} by the same arguments as when $v$ is a branching point. This finishes the proof of the claim (1) of the lemma.

\emph{Proof of display (2):}\\
The same arguments used for the proof of claim (1) of the lemma can be used without major modifications. We omit the proof.

\emph{Proof of display (3):}
It follows from Lemma \ref{lem:abstractcoupling}  that 
\begin{equation}
\lim_{M\to\infty}\limsup_{n\to\infty}P^{G_n}\left[\sup_{x\in\tilde{I}}L^{(n,K)}_t(x)\geq M\right]=0.
\end{equation}
From \eqref{eq:bgbdt} we get that 
\[
\lim_{n\to\infty} \mathbf{P}^{(K)}\left[ \lambda_{\text{res}}^{(n,K)}(\tilde{I})\geq \eta\right]=0,
\]
for all $\eta>0$.
The result follows from the two displays above.
  \end{proof}

 \begin{proof}[Proof of Lemma \ref{lem:aproximationofank}]
 Let $\mathcal{F}_m$ be the $\sigma$-algebra generated by $X^{G_n}$ up to time $A^{(n,K)}(m)$ and by the Brownian motion $B^{(n,K)}$ up to time $h^{(n,K)}_m$ (recall the definition of $h^{(n,K)}_m$ from \eqref{eq:defofhnk}). Let $S^{(n,K)}(0)=0$ and $T^{n,K}(m)$ be the time spent by $X^{G_n}$ in the $m$-th transition of $J^{(n,K)}$, that is
 \begin{equation}\label{eq:defoftnk}
 S^{(n,K)}(m)=A^{(n,K)}(m)-A^{(n,K)}(m-1), \qquad m\in\N.
 \end{equation}
 
  Let 
 \begin{equation}\label{eq:defofmm}
 M_m:=\sum_{i=1}^m \left(S^{(n,K)}(i)- E^{G_n}[S^{(n,K)}(i)\vert \mathcal{F}_{i-1}] \right).
 \end{equation}
  Define
 \[M^{(2)}_m:=M_m^2-\sum_{i=1}^m E^{G_n}[(S^{(n,K)}(i)-E^{G_n}[S^{(n,K)}(i)\vert \mathcal{F}_{i-1}])^2\vert \mathcal{F}_{i-1}].\]
 The first part of the proof consist in showing that the $M^{(2)}$ is a martingale to which we can apply the optional stopping theorem at $m(t)$ (defined at~\eqref{eq:defofmt})to obtain 
 \begin{equation}\label{eq:ostm2}
 E^{G_n}[M^{(2)}_{m(t)}]=0.
 \end{equation}
 In the second part of the proof we will use the display above to control the variance of $A^{(n,K)}(m(t))-\hat{A}^{(n,K)}(m(t))$ and deduce the lemma.
 
 \emph{Proof of display \eqref{eq:ostm2}}:\\
\noindent Our first claim is that $(M^{(2)}_m)_{m\in\N}$ is a martingale (with respect to the filtration $(\mathcal{F}_{m})_{m\in\N}$). That is, we want to show that
\[E^{G_n}[M_{m+1}^{(2)}-M_m^{(2)}\vert \mathcal{F}_m]=0.\]
Therefore, it suffices to establish that
\[E^{G_n}[M_{m+1}^2-M_m^2-E^{G_n}[(T_{m+1}-E^{G_n}[T_{m+1}\vert \mathcal{F}_m])^2\vert \mathcal{F}_m]\vert  \mathcal{F}_m ]=0,\]
or, equivalently
\begin{equation}\label{eq:kariola2}
E^{G_n}\left[M_{m+1}^2-M_m^2\vert \mathcal{F}_m\right]=E^{G_n}\left[(T_{m+1}-E^{G_n}[T_{m+1}\vert \mathcal{F}_m])^2\vert \mathcal{F}_m\right].
\end{equation}
Let us observe that
\begin{equation}\label{eq:kariola}
 \begin{aligned}
 &E^{G_n}[M_m(T_{m+1}-E^{G_n}[T_{m+1}\vert \mathcal{F}_m])\vert \mathcal{F}_m]\\=&M_mE^{G_n}[T_{m+1}-E^{G_n}[T_{m+1}\vert \mathcal{F}_m]\vert \mathcal{F}_m]\\=&0.
 \end{aligned}
 \end{equation}
where the first equality follows from the fact that $M_m$ is $\mathcal{F}_m$-measurable.
On the other hand
 \begin{align*}
 &E^{G_n}\left[M_{m+1}^2-M_m^2\vert \mathcal{F}_m\right]\\=&E^{G_n}\left[(M_m+(T_{m+1}-E^{G_n}[T_{m+1}\vert \mathcal{F}_m]))^2-M_m^2\vert \mathcal{F}_m \right]\\
 =& E^{G_n}\left[2M_m(T_{m+1}-E^{G_n}[T_{m+1}\vert \mathcal{F}_m]) + (T_{m+1}-E^{G_n}[T_{m+1}\vert \mathcal{F}_m])^2\vert \mathcal{F}_m\right]\\
 =&E^{G_n}\left[(T_{m+1}-E^{G_n}[T_{m+1}\vert \mathcal{F}_m])^2\vert \mathcal{F}_m\right],
 \end{align*}
 where the last equality follows from the identity \eqref{eq:kariola}. This establishes \eqref{eq:kariola2} and therefore  $(M^{(2)}_m)_{m\in\N}$ is a martingale with respect to $(\mathcal{F}_{m})_{m\in\N}$.  
Moreover $m(t)$ is a stopping time relative to $(\mathcal{F}_i)_{i\in\N}$.
Next, we will show that the pair $(M^{(2)},m(t))$ satisfies the hypotheses of the optional stopping theorem.

From \cite[Theorem 2.2, \S 7]{bookdoob1953stochastic}) it can be seen that the optional stopping theorem holds provided that the following three conditions hold
\begin{enumerate}
\item $m(t) < \infty$, $P^{G_n}$-almost surely,
\item $E^{G_n}[M^{(2)}_{m(t)}]<\infty$,
\item $E^{G_n}[M^{(2)}_m1_{\{m(t)>m\}}]\to0$ as $m\to\infty$.
\end{enumerate} 

To prove condition $(3)$ we use the Cauchy-Schwarz inequality to write
\begin{equation}\label{eq:C-S}
E^{G_n}\left[M^{(2)}_m1_{\{m(t)>m\}}\right]\leq E^{G_n}\left[ (M^{(2)}_m)^2\right]^{1/2}P^{G_n}[m(t)\geq m]^{1/2}.
\end{equation}
We have
\begin{equation}\label{eq:durango}
\begin{aligned}
&E^{G_n}\left[(M^{(2)}_m)^2 \right]=E^{G_n}\left[ \left(M_m^2-\sum_{i=1}^m E^{G_n}[(S^{(n,K)}(i)-E^{G_n}[S^{(n,K)}(i)\vert \mathcal{F}_{i-1}])^2\vert \mathcal{F}_{i-1}] \right)^2 \right]\\
&\leq 2E^{G_n}\left[M_m^4\right]+2E^{G_n}\left[ \left(\sum_{i=1}^m E^{G_n}[(S^{(n,K)}(i)-E^{G_n}[S^{(n,K)}(i)\vert \mathcal{F}_{i-1}])^2\vert \mathcal{F}_{i-1}]\right)^2 \right].
\end{aligned}
\end{equation}
We will start controlling $E^{G_n}[M_m^4]$.
\begin{equation}
E^{G_n}\left[M_m^4 \right]=E^{G_n}\left[\left(\sum_{i=1}^m \left(S^{(n,K)}(i)- E^{G_n}[S^{(n,K)}(i)\vert \mathcal{F}_{i-1}] \right)\right)^4 \right].
\end{equation}
The right hand side of the display above is composed by the sum of $m^4$ terms of the form
\begin{equation}
E^{G_n}\left[ \prod_{j=1}^4 \left(S^{(n,K)}(i_j)- E^{G_n}[S^{(n,K)}(i_j)\vert \mathcal{F}_{i_{j}-1}] \right)\right].
\end{equation}
By repeated use of the Cauchy-Schwarz inequality we get that
\begin{equation}
\begin{aligned}
&E^{G_n}\left[ \prod_{j=1}^4 \left(S^{(n,K)}(i_j)- E^{G_n}[S^{(n,K)}(i_j)\vert \mathcal{F}_{i_{j}-1}] \right)\right]\\
\leq&E^{G_n}\left[ \prod_{j=1}^2 \left(S^{(n,K)}(i_j)- E^{G_n}[S^{(n,K)}(i_j)\vert \mathcal{F}_{i_{j}-1}] \right)^2\right]^{1/2}E^{G_n}\left[ \prod_{j=3}^4 \left(S^{(n,K)}(i_j)- E^{G_n}[S^{(n,K)}(i_j)\vert \mathcal{F}_{i_{j}-1}] \right)^2\right]^{1/2}\\
\leq&\prod_{j=1}^4 E^{G_n}\left[ \left(S^{(n,K)}(i_j)- E^{G_n}[S^{(n,K)}(i_j)\vert \mathcal{F}_{i_{j}-1}] \right)^4\right]^{1/4}.
\end{aligned}
\end{equation} 
By virtue of Proposition \ref{prop:finitemoments}, all the terms appearing above are finite and moreover
\[\max_{i_1,i_2,i_3,i_4}E^{G_n}\left[ \prod_{j=1}^4 \left(S^{(n,K)}(i_j)- E^{G_n}[S^{(n,K)}(i_j)\vert \mathcal{F}_{i_{j}-1}] \right)\right]<\infty.
\] 
Hence, there exists $C$ such that
\begin{equation}\label{eq:loscuates}
E^{G_n}[M_m^4]<m^4C.
\end{equation}
Similarly, 
\[
E^{G_n}\left[ \left(\sum_{i=1}^m E^{G_n}[(S^{(n,K)}(i)-E^{G_n}[S^{(n,K)}(i)\vert \mathcal{F}_{i-1}])^2\vert \mathcal{F}_{i-1}]\right)^2 \right]
\]
is composed by the sum of $m^2$ terms of the form
\begin{equation}
 \prod_{j=1}^2E^{G_n}[(S^{(n,K)}(i_j)-E^{G_n}[S^{(n,K)}(i_j)\vert \mathcal{F}_{i_j-1}])^2\vert \mathcal{F}_{i_j-1}].
\end{equation}
By virtue of Proposition \ref{prop:finitemoments}, all those terms are finite  and moreover
\[
\max_{i_1,i_2}\prod_{j=1}^2E^{G_n}[(S^{(n,K)}(i_j)-E^{G_n}[S^{(n,K)}(i_j)\vert \mathcal{F}_{i_j-1}])^2\vert \mathcal{F}_{i_j-1}]<\infty.
\]
Therefore, there exists $C$ such that
\begin{equation}\label{eq:lostigres}
E^{G_n}\left[ \left(\sum_{i=1}^m E^{G_n}[(S^{(n,K)}(i)-E^{G_n}[S^{(n,K)}(i)\vert \mathcal{F}_{i-1}])^2\vert \mathcal{F}_{i-1}]\right)^2 \right]\leq Cm^2.
\end{equation}
Hence, from displays \eqref{eq:durango}, \eqref{eq:loscuates}, \eqref{eq:lostigres} we get that there exists $C$ with
\begin{equation}\label{eq:comounrayo}
E^{G_n}\left[(M^{(2)}_m)^2 \right]\leq Cm^4.
\end{equation}
Now we have to get an upper bound for $P^{G_n}[m(t)\geq m]$.
It follows from \eqref{eq:defofmt} that 
\begin{equation}
m(t)=\sum_{e\in E^\ast(\T^{(n,K)})} l^{(n,K)}_t(e).
\end{equation}
Hence
\begin{align*}
&P^{G_n}[ m(t)\geq m]\leq P^{G_n}\left [ \bigcup_{e\in E^\ast(\T^{(n,K)})} \left\{ l^{(n,K)}_t(e)\geq \frac{m}{\abs{E(\T^{(n,K)})}}\right\}   \right ]\\
&\leq \sum_{e\in E^\ast(\T^{(n,K)})} P^{G_n}\left[  l^{(n,K)}_t(e)\geq \frac{m}{\abs{E(\T^{(n,K)})}}\right].
\end{align*}
Therefore, since $\abs{E(\T^{(n,K)})}$ is finite for fixed $K,n$, it follows from Lemma~\ref{l:tailoflocaltime} and Lemma~\ref{l:tailoflocaltimetrifurcation} that there exists $c,C>0$ independent of $m$ such that
\begin{equation}\label{eq:martin}
P^{G_n}[ m(t)\geq m]\leq C \exp(-cm).
\end{equation}
From displays \eqref{eq:C-S}, \eqref{eq:comounrayo} and \eqref{eq:martin} we get that there exists $c',C'>0$ with
\begin{equation}\label{eq:lajauladeoro}
E^{G_n}\left[M^{(2)}_m 1_{\{m(t)\geq m\}}\right] \leq C'\exp(-c'm). 
\end{equation}
The display above implies condition $(3)$.
On the other hand 
\begin{equation}
E^{G_n}[M^{(2)}_{m(t)}]=\sum_{i=0}^{\infty} E^{G_n}[M^{(2)}_m 1_{\{m(t)=m\}} ].
\end{equation}
Hence, by \eqref{eq:lajauladeoro} we get condition $(2)$.
Finally, condition $(1)$ follows from \eqref{eq:martin}. We have shown that the optional stopping theorem holds for the pair $(M^{(2)}, m(t))$. Therefore, we have deduced display \eqref{eq:ostm2}.

\emph{Control of the variance of $A^{(n,K)}(m(t))-\hat{A}^{(n,K)}(m(t))$:}\\
 We start by noticing that~\eqref{eq:ostm2} implies
  \begin{equation}\label{eq:compensator}
 E^{G_n}\left[M^2_{m(t)}\right]=E^{G_n}\left[\sum_{i=1}^{m(t)} E^{G_n}[(S^{(n,K)}(i)-E^{G_n}[S^{(n,K)}(i)\vert\mathcal{F}_{i-1}])^2\vert \mathcal{F}_{i-1}]\right].
 \end{equation}
 Recall the definitions of $M_m$ and $S^{(n,K)}$ in \eqref{eq:defofmm} and \eqref{eq:defoftnk} respectively. Recall also the definition of $\tau(x)$ in the first paragraph of Section \ref{sect:averagedtimechange} and the definition of $l^{(n,K,\text{vert})}_m$ from \eqref{eq:defoflvert}. We have
  \begin{align*}
 M_{m}&=A^{(n,K)}(m)-\sum_{i=1}^{m} E^{G_n}[S^{(n,K)}(i)\vert \mathcal{F}_{i-1}]\\ 
&=A^{(n,K)}(m)-\sum_{x\in V^\ast(\T^{(n,K)})} l^{(n,K,\text{vert})}_m(x) E^{G_n}[\tau(x)]\\
&= A^{(n,K)}(m)-\hat{A}^{(n,K)}(m).
 \end{align*}
Hence,
 \begin{align}\label{end_of_the_world}
 &P^{G_n}\left[n^{-3/2}\abs{A^{(n,K)}(m(t)) -\hat{A}^{(n,K)}(m(t))}\geq\epsilon \right] \\ = & P^{G_n} \left[n^{-3/2}\abs{M_{m(t)}}\geq\epsilon\right]\\ \nonumber
 \leq & n^{-3} \epsilon^{-2}E^{G_n}\left[ M^{2}_{m(t)}\right]\\=&n^{-3}\epsilon^{-2}  E^{G_n}\left[\sum_{i=1}^{m(t)} E^{G_n}[(S^{(n,K)}(i)-E^{G_n}[S^{(n,K)}(i)\vert \mathcal{F}_{i-1}])^2\vert \mathcal{F}_{i-1}]\right],
 \end{align}
 where in the first inequality we have used Chebyshev's inequality and in the last equality we have used \eqref{eq:compensator}.
 The display above equals 
  \[\epsilon^{-2}n^{-3} E^{G_n}\left[\sum_{x\in V^\ast(\T^{(n,K)})} l^{(n,K,\text{vert})}_{m(t)}(x)E^{G_n}[(\tau(x)-E^{G_n}[\tau(x)])^2]\right].\]
 On the other hand, 
 \begin{equation}\label{eq:robertocarlos}
 E^{G_n}[(\tau(x)-E^{G_n}[\tau(x)])^2] \leq  2E^{G_n}[\tau(x)^2]+2 E^{G_n}[\tau(x)]^2 \leq 4 E^{G_n}[\tau(x)^2].
 \end{equation}
 
 If $x\in V^\circ(\T^{(n,K)})$, using Corollary \ref{variance_bound} (and recalling the definition of $e^+(v),e^-(v)$ from the paragraph above \eqref{eq:sopracontrariar}) we get that the right hand side of \eqref{eq:robertocarlos} is bounded above by
 \[C \abs{E(G_n(x))}^2 \text{diam}_{G_n}(G_n(x)) \min\{\reff^{G_n}(e^-(x)),\reff^{G_n}(e^+(x))\}
 \]
 for some constant $C$. Hence 
 \begin{equation}\label{eq:pericolospalotes}
 \begin{aligned}
 &E^{G_n} \left[\sum_{x\in V^\circ(\T^{(n,K)})}l^{(n,K,\text{vert})}_{m(t)}(x)E^{G_n}[(\tau(x)-E^{G_n}[\tau(x)])^2]\right]\\
 =&\sum_{x\in V^\circ(\T^{(n,K)})}E^{G_n} [l^{(n,K,\text{vert})}_{m(t)}(x)]E^{G_n}[(\tau(x)-E^{G_n}[\tau(x)])^2] \\
 =&\sum_{x\in V^\circ(\T^{(n,K)})}E^{G_n} [l^{(n,K,\rightarrow)}_{t}(e^-(x))+l^{(n,K,\leftarrow)}_{t}(e^+(x))]E^{G_n}[(\tau(x)-E^{G_n}[\tau(x)])^2]\\
  \leq &C \sum_{x\in V^\circ(\T^{(n,K)})} E^{G_n}[l^{(n,K,\rightarrow)}_{t}(e^-(x))+l^{(n,K,\leftarrow)}_{t}(e^+(x))]\\ &\abs{E(G_n(x))}^2 \text{diam}_{G_n}(G_n(x))\min\{\reff^{G_n}(e^-(x)),\reff^{G_n}(e^+(x))\},
\end{aligned}
\end{equation}
where the second equality follows from \eqref{eq:vertupdown}.
 By  virtue of Lemma \ref{l:tailoflocaltime} we have that, $\mathbf{P}^{(K)}$ a.s.,
 \[n^{-1/2}\sup_{x\in V^{\circ}(\T^{(n,K)})} \{ \min\{\reff^{G_n}(e^-(x)),\reff^{G_n}(e^+(x))\}E^{G_n}[l^{(n,K,\rightarrow)}_t(e^-(x))+l^{(n,K,\leftarrow)}_{t}(e^+(x))]\}\] is bounded uniformly in $n$ and $K$.
  Therefore, $\mathbf{P}^{(K)}$-a.s., there exists a constant $C$ such that the right hand side of \eqref{eq:pericolospalotes} is bounded by
 \[C n^{1/2}\sum_{x\in V^\circ(\T^{(n,K)})} \abs{E(G_n(x))}^2 \text{diam}(G_n(x)).\]
 A similar analysis, which we omit, can be carried for vertices $v\in V^\ast(\T^{(n,K)})\backslash V^\circ(\T^{(n,K)})$ (using the same ideas in the proof of display (1) of Lemma \ref{lem:vcircle}). This, together with the display above allows to get that
 \begin{equation}
 \begin{aligned}
 &E^{G_n} \left[\sum_{x\in V^\ast(\T^{(n,K)})}l^{(n,K,\text{vert})}_{m(t)}(x)E^{G_n}[(\tau(x)-E^{G_n}[\tau(x)])^2]\right]\\
 \leq & C n^{1/2}\sum_{x\in V^\ast(\T^{(n,K)})} \abs{E(G_n(x))}^2 \text{diam}(G_n(x))
 \end{aligned}
 \end{equation}
 
  Moreover, since $\sup_{x\in V^*(\T^{(n,K)})}\text{diam}(G_n(x))\leq 2\Delta^{(n,K)}_{G_n}$, (where $\Delta^{(n,K)}_{G_n}$ is as in \eqref{eq:defofdeltaintr} ) the display above is bounded by
\begin{equation}\label{eq:inmenso}
 C n^{1/2}\Delta^{(n,K)}_{G_n}\sum_{x\in V^\ast(\T^{(n,K)})} \abs{E(G_n(x))}^2 . 
 \end{equation}
On the other hand, since $\sum_{x\in V^\ast(\T^{(n,K)})} \abs{E(G_n(x))}\leq 2\abs{E(G_n)} $, it follows that 
\[\sum_{x\in V^\ast(\T^{(n,K)})} \abs{E(G_n(x))}^2 \leq 4\abs{E(G_n)}^2.\]
Moreover, by condition $(V)_{\nu}$ we have that,
\begin{equation}\label{eq:totaledgecardinality}
n^{-1}\abs{E(G_n)}\to\nu.
\end{equation}
 Hence, there exists $C'$ such that $\abs{E(G_n)}^2\leq C'n^2$. 
Therefore \eqref{eq:inmenso} is upper bounded by
\[C n^{5/2} \Delta^{(n,K)}_{G_n} \]
for some constant $C$. Summarizing our argument since~\eqref{end_of_the_world}, we find that
\begin{equation}\label{eq:pe}
P^{G_n}\left[n^{-3/2}\abs{ A^{(n,K)}(m(t))-\hat{A}^{(n,K)}(m(t))}\geq \epsilon\right]\leq C n^{-1/2}\Delta^{(n,K)}_{G_n}
\end{equation}
for some constant $C$.
But, since the $G_n$ verifies condition $(S)$ , 
\[\lim_{K\to\infty}\limsup_{n\to\infty}\mathbf{P}^{(K)}\left[n^{-1/2}\Delta^{(n,K)}_{G_n}\geq\epsilon\right]=0\] for all $\epsilon>0$. Therefore, by \eqref{eq:pe} we find that
\begin{equation}
\lim_{K\to\infty}\limsup_{n\to\infty}\mathbf{P}^{(K)}\left[n^{-3/2}\abs{A^{(n,K)}(m(t))-\hat{A}^{(n,K)}(m(t))}\geq \epsilon\right]=0,
\end{equation} 
and that finishes the proof of the lemma.
 \end{proof}

\subsubsection{Lemma~\ref{lem:commutetime}}
Recall the definition of $e^+(v),e^{-}(v)$ from the paragraph above \eqref{eq:sopracontrariar}.
For any $v\in V^\circ(\T^{(n,K)})$, let 
\[
p(v):=\frac{\reff^{G_n}(e^-(v))}{\reff^{G_n}(e^-(v))+\reff^{G_n}(e^+(v))} \quad \text{and}\quad q(v):=1-p(v).
\]
For the proof of Lemma \ref{lem:commutetime} we will make use of the following auxiliary result:
\begin{lemma}\label{lem:discretediscrete}
For all $t\geq0,\epsilon>0$, and $K>0$,
\begin{equation}\label{eq:fofitos1}
 \mathbf{P}^{(K)}\left[\sup_{v\in V^\circ(\T^{(n,K)})}d^{(n,K)}_{\text{res}}(e^+(v))\abs{ p(v)l^{(n,K,\text{vert})}_{m(t)}(v) -l^{(n,K,\leftrightarrow)}_t (e^+(v))}\geq \epsilon \right]\stackrel{n\to\infty}{\to}0,
\end{equation}
and
\begin{equation}\label{eq:fofitos2}
 \mathbf{P}^{(K)}\left[\sup_{v\in V^\circ(\T^{(n,K)})}d^{(n,K)}_{\text{res}}(e^-(v))\abs{ q(v)l^{(n,K,\text{vert})}_{m(t)}(v)-l^{(n,K,\leftrightarrow)}_t (e^-(v))}\geq \epsilon\right]\stackrel{n\to\infty}{\to}0.
\end{equation}
\end{lemma}

\begin{proof}[Proof of Lemma \ref{lem:discretediscrete}]
For all $v\in V^\circ(\T^{(n,K)})$, let $x_v^+$ be the midpoint of $e^+(v)$.
It follows from display \eqref{eq:convergenceofedgelocaltimes} in Proposition \ref{l:convergenceoflocaltimes}  that, for any $\epsilon>0$
\begin{equation}\label{eq:fresh2}
\lim_{n\to\infty}\mathbf{P}^{(K)}\left[\sup_{v\in V^\circ(\T^{(n,K)})}\abs{d^{(n,K)}_{\text{res}}(e^+(v))l^{(n,K)}_t(e^+(v))-L^{(n,K)}_t(x^+_v)}\geq \epsilon\right]=0.
\end{equation}
Also, from \eqref{eq:convergenceofvertexlocaltimes}, 
\begin{equation}\label{eq:fresh1}
\lim_{n\to\infty}\mathbf{P}^{(K)}\left[\sup_{v\in V^\circ(\T^{(n,K)})}\abs{2\reff^{(n,K,\text{vert})}(v)l_t^{(n,K,\text{vert})}(v)-L^{(n,K)}_t(v)}\geq \epsilon\right]=0,
\end{equation}
for all $\epsilon>0$.
By Lemma \ref{lem_no_macro_res} we get that 
\[
\lim_{n\to\infty}\mathbf{P}^{(K)}\left[\sup_{v\in V^\circ(\T^{(n,K)})} d^{(n,K)}_{\text{res}}(x^+_v,v)\geq \epsilon \right]=0.
\]
 Therefore, by the continuity of $L^{(n,K)}$,  
\begin{equation}\label{eq:fresh}
\lim_{n\to\infty}\mathbf{P}^{(K)}\left[\sup_{v\in V^\circ(\T^{(n,K)})}\abs{ L_t^{(n,K)}(x^+_v)-L_t^{(n,K)} (v)}\geq \epsilon\right]=0.
\end{equation}
By displays \eqref{eq:fresh2}, \eqref{eq:fresh1} and \eqref{eq:fresh}, we get that
\begin{equation}
\lim_{n\to\infty}\mathbf{P}^{(K)}\left[\sup_{v\in V^\circ(\T^{(n,K)})}\abs{2\reff^{(n,K,\text{vert})}(v)l_t^{(n,K,\text{vert})}(v)-d^{(n,K)}_{\text{res}}(e^+(v))l^{(n,K)}_t(e^+(v))}\geq \epsilon\right]=0,
\end{equation}
which is equivalent to
\begin{equation}\label{eq:fresh3}
\lim_{n\to\infty}\mathbf {P}^{(K)}\left[\sup_{v\in V^\circ(\T^{(n,K)})}\abs{d_{\text{res}}^{(n,K)}(e^+(v))\left(2p(v)l_t^{(n,K,\text{vert})}(v)-l^{(n,K)}_t(e^+(v))\right)}\geq \epsilon\right]=0.
\end{equation}
From the equation above and the fact that $l^{(n,K)}_t(e)$ equals either $2l^{(n,K,\leftrightarrow)}_t(e)$ or $2l^{(n,K,\leftrightarrow)}_t(e)-1$, we get  display \eqref{eq:fofitos1}. Display \eqref{eq:fofitos2} is proved similarly.
\end{proof}

\begin{proof}[Proof of Lemma \ref{lem:commutetime}]
For any $v\in V^\circ(\T^{(n,K)})$, let $X^{G_n,v,+}$ be a random walk on $G_n(e^+(v))$ started at $v$ (where $G_n(e^+(v))$ is defined in \eqref{eq:defofgn1} and \eqref{eq:defofgn2}). Let $v^+\in V^*(\T^{(n,K)})$ be the unique vertex such that $(v,v+1)\in E^\ast(\T^{(n,K)})$ and $v\prec v^+$. The uniqueness is guaranteed since $v\in V^\circ(\T^{(n,K)})$. Also, set $v^-$ as the unique vertex in $V^*(\T^{(n,K)})$ with $(v^-,v)\in \T^{(n,K)}$ and $v^-\prec v$.  Set
\begin{equation}
\tau^+(v):=\min\{m> 0: X^{G_n,v,+}_m=v^+ \}.
\end{equation}
Also, let $X^{G_n,v,-}$ be a random walk on $G_n(e^-(v))$ started at $v$ and
\begin{equation}
\tau^-(v):=\min\{m> 0: X^{G_n,v,-}_m=v^- \}.
\end{equation}

The first part of the proof consists in the proving that for $v\in V^\circ(\T^{(n,K)})$
\begin{equation}\label{eq:fkng}
E^{G_n}[\tau(v)]= p(v) E^{G_n} [\tau^+(v)] + q(v) E^{G_n}[\tau^-(v)]. 
\end{equation}

\emph{First part, proof of \eqref{eq:fkng}}:\\
In order to prove the previous equation, we will describe mathematically a decomposition of the excursions from $v$ which are depicted in Figure 8.
\begin{figure}
  \includegraphics[width=\linewidth]{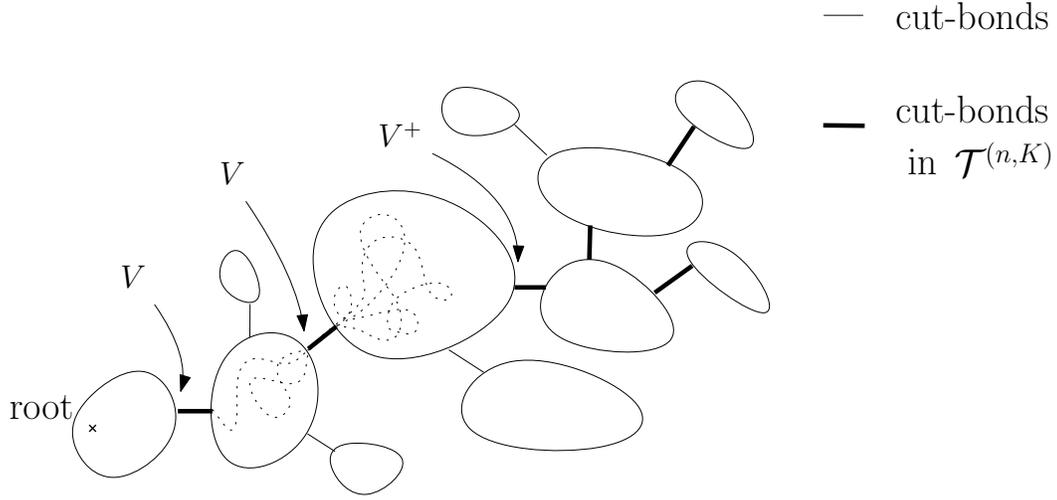}
  \caption{Excursions from $v$ before exiting the sausages adjacent to $v$}
\end{figure}

Let $G(v)$ be the number of times that $X^{G_n}$ returns to $v$ before hitting $\{v^-,v^+\}$. By the strong Markov property for $X^{G_n}$, it follows that $G(v)$ has a geometric distribution. Moreover, the strong Markov property allows to decompose $\tau(v)$ as 
\begin{equation}\label{eq:decompositionoffkng}
\tau(v)=\tau^{\text{out}}(v)+\sum_{i=1}^{G(v)} \tau_i^{\circlearrowleft}(v),
\end{equation}
where $(\tau_i^{\circlearrowleft}(v))_{i\in\N}$ is an i.i.d.~sequence of random variables having the distribution of $T^+(v)$ (defined in \eqref{eq:defofTT+}) conditioned on 
\[\{X^{G_n}_0=v\}\cap\left\{T^+(v)<T(\{v^+,v^-\})\right\}\]
and $\tau^{\text{out}}(v)$ has the distribution of $T( \{v^-,v^+\})$ conditioned on 
\[\{X^{G_n}_0=v\}\cap\left\{T(\{v^-,v^+\})<T^+(v)\right\}.\]
From \eqref{eq:decompositionoffkng} we get
\begin{equation}\label{eq:fkng1}
E^{G_n}[\tau(v)]=E^{G_n}[G(v)]E^{G_n}[\tau^{\circlearrowleft}_1(v)]+E^{G_n}[\tau^{\text{out}}(v)]. 
\end{equation}
We can write 
\begin{equation}\label{eq:fkng2}
\begin{aligned}
E^{G_n}[\tau^{\text{out}}(v)]=&p(v)E^{G_n}[\tau^{\text{out}}(v)\vert T(v^+)<T^+(v)]\\+&q(v)E^{G_n}[\tau^{\text{out}}(v)\vert T(v^-)<T^+(v)],
\end{aligned}
\end{equation}
Since $G(v)$ has a geometric distribution of parameter 
\[p=P^{G_n}_v[T(\{v^-,v^+\}) <T^+(v)]\] we get that
\begin{equation}
E^{G_n}[G(v)]=\frac{1}{p}-1=\frac{\pi(v)}{C_{\text{eff}}(v,\{v^-,v^+\})}-1,
\end{equation}
where $\pi(v)$ is the degree of $v$ in $G_n$ and $C_{\text{eff}}(v,\{v^-,v^+\})$ is the electrical conductance between $v$ and $\{v^-,v^+\}$.
Furthermore, we can decompose $E^{G_n}[\tau^{\circlearrowleft}_1(v)]$ as 
\begin{equation}\label{eq:fkng3}
\begin{aligned}
E^{G_n}[\tau^{\circlearrowleft}_1(v)]=&P^{G_n}_v[X^{G_n}_1\in G_n(e^+(v))\vert T^+(v)<T(\{v^-,v^+\}) ] E^{G_n}[\tau^{e,+}(v) ]\\+&P^{G_n}_v[X^{G_n}_1\in G_n(e^-(v))\vert T^+(v)<T(\{v^-,v^+\}) ] E^{G_n}[\tau^{e,-}(v) ],
\end{aligned}
\end{equation}
where $\tau^{e,+}(v)$ has the distribution of $\tau^{\circlearrowleft}_1(v)$ conditioned on $\{X^{G_n}_1\in G_n(e^+(v))\}$
and $\tau^{e,-}(v)$ has the distribution of $\tau^{\circlearrowleft}_1(v)$ conditioned on $\{X^{G_n}_1\in G_n(e^-(v))\}$.

Moreover, since the degree of $v$ in $G_n(e^+(v))$ is $1$,
\begin{equation}\label{eq:fkng3.3}
P^{G_n}_v[X^{G_n}_1 \in G_n(e^+(v))]=\frac{1}{\pi(v)}.
\end{equation}
Moreover
\begin{equation}\label{eq:fkng3.6}
\begin{aligned}
&P^{G_n}_v[T^+(v)<T(\{v^-,v^+\})\vert X^{G_n}_1 \in G_n(e^+(v))]\\
=&P^{G_n}_v[X^{G_n,v,+} \text{ returns to }v \text{ before hitting }v^+\text{ in } G_n(e^+(v))]\\=&1-C_{\text{eff}}(v,v^-),
\end{aligned}\end{equation}
where in the last equality we have used that the degree of $v$ in $G_n(e^+(v))$ is $1$.

On the other hand, 
\begin{equation}\label{eq:fkng4}
\begin{aligned}
&P^{G_n}_v[X^{G_n}_1 \in G_n(e^+(v))\vert T^+(v)<T(\{v^-,v^+\}) ]\\
=&\frac{P^{G_n}_v[X^{G_n}_1 \in G_n(e^+(v));T^+(v)<T(\{v^-,v^+\}) ]} {P^{G_n}_v[ T^+(v)<T(\{v^-,v^+\})]}\\
=&\frac{P^{G_n}_v[X^{G_n}_1 \in G_n(e^+(v))]P^{G_n}_v[T^+(v)<T(\{v^-,v^+\})\vert X^{G_n}_1 \in G_n(e^+(v))]}{P^{G_n}_v[ T^+(v)<T(\{v^-,v^+\})]} \\
= &\frac{1}{\pi(v)}(1-C_{\text{eff}}(v,v^-))\left(1-\frac{C_{\text{eff}}(v,\{v^-,v^+\})}{\pi(v)}\right)^{-1}\\
= & \frac{1-C_{\text{eff}}(v,v^-)}{\pi(v)-C_{\text{eff}}(v,\{v^-,v^+\})},
\end{aligned}
\end{equation}
where, in the third equality we have used \eqref{eq:fkng3.3} and \eqref{eq:fkng3.6}.
Similarly, using that the degree of $v$ in $G_n(e^-(v))$ is $\pi(v)-1$, we have \begin{equation}\label{eq:fkng5}
\begin{aligned}
&P^{G_n}_v[X^{G_n}_1\in G_n(e^-(v))\vert T^+(v)<T(\{v^-,v^+\}) ]\\
= &\frac{\pi(v)-1}{\pi(v)}\left(1-\frac{C_{\text{eff}}(v,v^+)}{\pi(v)-1}\right)\left(1-\frac{C_{\text{eff}}(v,\{v^-,v^+\})}{\pi(v)}\right)^{-1}\\
= & \frac{\pi(v)-1-C_{\text{eff}}(v,v^+)}{\pi(v)-C_{\text{eff}}(v,\{v^-,v^+\})}.
\end{aligned}
\end{equation}
 An elementary computation using \eqref{eq:fkng1}, \eqref{eq:fkng2}, \eqref{eq:fkng3} and \eqref{eq:fkng5} yields
\begin{equation}\label{eq:fkng10}
\begin{aligned}
E^{G_n}[\tau(v)]&=p(v)E^{G_n}_v[\tau^{\text{out}}(v)\vert T(v^+)<T^+(v)]\\&+q(v)E^{G_n}[\tau^{\text{out}}(v)\vert T(v^-)<T^+(v)]\\
&+\frac{1-C_{\text{eff}}(v,v^-)}{C_{\text{eff}}(v,\{v^-,v^+\}))}E^{G_n}[\tau^{e,+}(v) ]\\&+\frac{\pi(v)-1-C_{\text{eff}}(v,v^+)}{C_{\text{eff}}(v,\{v^-,v^+\})} E^{G_n}[\tau^{e,-}(v) ]
\end{aligned}
\end{equation}

We can perform a decomposition similar to \eqref{eq:decompositionoffkng} with $\tau^+(v)$ in place of $\tau(v)$: Let $G^+(v)$ be the number of returns of $X^{G_n,v,+}$ to $v$ before hitting $v^+$. Then we have that
\begin{equation}
\tau^+(v)=\tau^{+,\text{out}}(v)+\sum_{i=1}^{G^+(v)}\tau^{+,\circlearrowleft}_i(v),
\end{equation}
where $(\tau_i^{+,\circlearrowleft}(v))_{i\in\N}$ is an i.i.d.~sequence of random variables having the distribution of $\min\{j>0: X^{G_n,v,+}_j=v\}$ conditioned on
\[\left\{\min\{j>0: X^{G_n,v,+}_j=v\}<\min\{j>0: X^{G_n,v,+}_j=v^+\}\right\}.\]
and $\tau^{+,\text{out}}(x)$ has the distribution of $\min\{j>0: X^{G_n,v,+}_j=v^+\}$ conditioned on
\[\left\{\min\{j>0: X^{G_n,v,+}_j=v^+\}<\min\{j>0: X^{G_n,v,+}_j=v\}\right\}.\]
Hence, we have
\begin{equation}
E^{G_n}[\tau^+(v)]=E^{G_n}[\tau^{+,\text{out}}(v)]+E^{G_n}[G^+(v)]E^{G_n}[\tau^{+,\circlearrowleft}_1(v)].
\end{equation}
 Moreover $\tau^{+,\text{out}}(v)$ is distributed as $\tau^{\text{out}}(v)$ conditioned on $T(v^+)<T(v^-)$ and $\tau^{+,\circlearrowleft}_1$ is distributed as $\tau^{e,+}(v)$. Therefore
\begin{equation}\label{eq:fkng6}
E^{G_n}[\tau^+(v)]=E^{G_n}[\tau^{\text{out}}(v)\vert T(v^+)<T(v^-)]+E^{G_n}[G^+(v)]E^{G_n}[\tau^{e,+}(v)].
\end{equation}
An analogous decomposition for $\tau^-(v)$ yields
\begin{equation}\label{eq:fkng7}
E^{G_n}[\tau^-(v)]=E^{G_n}[\tau^{\text{out}}(v)\vert T(v^-)<T(v^+)]+E^{G_n}[G^-(v)]E^{G_n}[\tau^{e,-}(v)]
\end{equation}
where $G^-(v)$ is a geometric random variable of parameter $C_{\text{eff}}(v,v^-)/(\pi(v)-1)$.
Moreover
\begin{equation}\label{eq:fkng8}
E^{G_n}[G^+(v)]=\frac{1}{C_{\text{eff}}(v,v^-)}-1 \quad \text{ and } \quad E^{G_n}[G^-(v)]=\frac{\pi(v)-1}{C_{\text{eff}}(v,v^-)}-1.
\end{equation}
Also
\begin{equation}\label{eq:fkng9}
p(v)=\frac{C_{\text{eff}}(v,v^+)}{C_{\text{eff}}(v,\{v^+,v^-\})} \quad \text{ and } \quad q(v)=\frac{C_{\text{eff}}(v,v^-)}{C_{\text{eff}}(v,\{v^+,v^-\})}.
\end{equation}
An elementary computation using \eqref{eq:fkng10}, \eqref{eq:fkng6}, \eqref{eq:fkng7}, \eqref{eq:fkng8} and \eqref{eq:fkng9} yields \eqref{eq:fkng}.

\emph{Second part, application of the commute time formula:}\\
Let 
\[
\begin{aligned}
\bar{A}^{(n,K)}(m)=&\sum_{v\in  V^\circ(\T^{(n,K)})} p(v)l^{(n,K,\text{vert})}_m(v) E^{G_n}[\tau^+(v)]\\
+&\sum_{v\in  V^\circ(\T^{(n,K)})} q(v)l^{(n,K,\text{vert})}_m(v) E^{G_n}[\tau^-(v)]
\end{aligned}
\]
Recall the definition of $\hat{A}^{(n,K)}$ from \eqref{eq:defofhatank}. Using \eqref{eq:fkng} and displays (1) and (2) of Lemma \ref{lem:vcircle} we get
\begin{equation}\label{eq:mm}
\begin{aligned}
\lim_{K\to\infty}\limsup_{n\to\infty}\mathbf{P}\Bigg[&
\Bigg|\hat{A}^{(n,K)}(m(t))-\bar{A}^{(n,K)}(m(t))\Bigg|\geq \epsilon\Bigg]=0,
\end{aligned}
\end{equation}
for all $\epsilon>0$.
On the other hand, by virtue of the commute time formula, for any $e=(e^-,e^+)\in E^*(\T^{(n,K)})$ with $e^-\prec e+$, we have that
\[E^{G_n}[\tau^+(e^-)]+ E^{G_n}[\tau^-(e^+)]=2\reff^{G_n}(e) \abs{E(G_n(e))},\]
which allow us to rewrite $\tilde{A}^{(n,K)}(m)$ (which was introduce above Lemma~\ref{lem:aproximationofank})
\begin{equation}\label{eq:mmm}
\begin{aligned}
\tilde{A}^{(n,K)}(m(t))=&\sum_{e\in E^*(\T^{(n,K)})} l^{(n,K,\leftrightarrow)}_t(e) E^{G_n}[\tau^+(e^-)]\\
+&\sum_{e\in E^*(\T^{(n,K)})} l^{(n,K,\leftrightarrow)}_t(e) E^{G_n}[\tau^-(e^+)].
\end{aligned}
\end{equation}
Let us denote
\begin{equation}
\begin{aligned}
R^{(n,K)}(m)=&\sum_{e\in E^*(\T^{(n,K)})} \left(p(e^-)l^{(n,K,\text{vert})}_m(e^-) -l^{(n,K,\leftrightarrow)}_t(e)\right) E^{G_n}[\tau^+(e^-)]\\
+&\sum_{e\in E^*(\T^{(n,K)})} \left(q(e^+)l^{(n,K,\text{vert})}_m(e^+) -l^{(n,K,\leftrightarrow)}_t(e)\right) E^{G_n}[\tau^-(e^+)].
\end{aligned}
\end{equation}
Therefore, (using display (2) of Lemma \ref{lem:vcircle} to neglect the error term) we have that 
\begin{equation}\label{eq:antman}
\tilde{A}^{(n,K)}(m(t))-\bar{A}^{(n,K)}(m(t))=R^{(n,K)}(m(t))+o(n,K),
\end{equation}
where $o(n,K)$ is a term that satisfy 
\[\lim_{K\to\infty} \limsup_{n\to\infty}\mathbf{P}^{(K)}[o(n,K)n^{-3/2}\geq \epsilon]=0\] for all $\epsilon>0$.
Let us write $E^{G_n}[\tau^{\leftrightarrow}(e)]:=E^{G_n}[\tau^+(e^-)]+E^{G_n}[\tau^-(e^+)]$.
 Lemma \ref{lem:discretediscrete} and \eqref{eq:mm} imply that, for all $\epsilon>0$,
\begin{equation}
\mathbf{P}^{(K)}\left[\abs{R^{(n,K)}(m(t))}\geq \epsilon\sum_{e\in E^*(\T^{(n,K)})} \frac{E^{G_n} [\tau^{\leftrightarrow}(e)]}{d^{(n,K)}_{\text{res}}(e)}\right]\stackrel{n\to\infty}{\to}0,
\end{equation}
which is equivalent to
\begin{equation}\label{eq:gimmetheloot}
\mathbf{P}^{(K)}\left[\abs{R^{(n,K)}(m(t))}\geq \epsilon n^{1/2}\sum_{e\in E^*(\T^{(n,K)})} \frac{E^{G_n} [\tau^{\leftrightarrow}(e)]}{\reff^{(n,K)}(e)}\right]\stackrel{n\to\infty}{\to}0,
\end{equation}
for all $\epsilon>0$.
By the commute time formula we have that 
\[\frac{E^{G_n} [\tau^{\leftrightarrow}(e)]}{\reff^{(n,K)}(e)}=2\abs{E(G_n(e))},\]
Therefore, from display \eqref{eq:gimmetheloot} we get that
\[
\mathbf{P}^{(K)}\left[\abs{R^{(n,K)}(m(t))}\geq \epsilon 2 n^{1/2}\sum_{e\in E^*(\T^{(n,K)})} \abs{E(G_n(e))}\right]\stackrel{n\to\infty}{\to}0,
\]
for all $\epsilon>0$.
Moreover, by display \eqref{eq:totaledgecardinality} we get that there exists $C>0$ such that
\begin{equation}
\mathbf{P}^{(K)}\left[\abs{R^{(n,K)}(m(t))}\geq \epsilon C2 n^{3/2}\right]\stackrel{n\to\infty}{\to}0,
\end{equation}
for all $\epsilon>0$. That, together with \eqref{eq:mm} and \eqref{eq:antman} finishes the proof of the lemma.
\end{proof}

\subsubsection{Lemma~\ref{lem:hatankislinear}}

Before presenting the proof of Lemma \ref{lem:hatankislinear}, we will need two more preparatory lemmas.
Let 
\begin{equation}
A^{(K)}(t):=\int_{\frak{T}} L_t(x) \lambda^{(K)}_{\text{res}}(dx)
\end{equation}
and 
\begin{equation}
\tau^{(K)}(t):=\inf\{s\geq0: A^{(K)}(s)\geq t\}.
\end{equation}
It is not hard to verify that $(B^{(K)}_t)_{t\geq0}$ is distributed as $(B_{\tau^{(K)}(t)})_{t\geq0}$. This observation implies the following result whose proof is left to the reader.
\begin{lemma}\label{lem:klocaltime} For all $K>0$,
\[
L^{(K)}_t(x)=L_{\tau^{(K)}(t)}(x).
\] 
\end{lemma}

The proof of the first display of the following lemma is display $(26)$ in \cite{Croydon_crt}. The second display follows directly from the first one.
\begin{lemma}\label{lem:conttimechange} For all $\epsilon>0$, we have that
\[
\lim_{K\to\infty} \mathbf{P}^{(K)}\left[\abs{A^{(K)}(t)-t}\geq \epsilon \right]=0
\]
and
\[
\lim_{K\to\infty} \mathbf{P}^{(K)} \left[\abs{\tau^{(K)}(t)-t}\geq\epsilon\right]=0.
\]
\end{lemma}

Let
\[D(n,K,\delta):=\{(x,e)\in \frak{T}^{(n,K)}\times E(\T^{(n,K)}): d^{(n,K)}_{\text{res}}(x,e)\leq \delta \},\]
where
\[
d^{(n,K)}_{\text{res}}(x,e):=d^{(n,K)}_{\text{res}}(x,e_+)\vee d^{(n,K)}_{\text{res}}(x,e_-) 
\]
and $e_+,e_-$ are the endpoints of $e$ $(e=(e_-,e_+))$.

\begin{lemma}\label{lem:gluethetimes}
\[
\lim_{\delta\to 0}\limsup_{K\to\infty}\limsup_{n\to\infty}\mathbf{P}^{(K)}\left[\sup_{(x,e)\in D(n,K,\delta)} \abs{d_{\text{res}}^{(n,K)}(e)l^{(n,K,\leftrightarrow)}_t(e)-L^{(n,K)}_t(x)} \geq \epsilon \right]=0
\]
\begin{proof}[Proof of Lemma \ref{lem:gluethetimes}]
By Lemma \ref{lem:klocaltime} and  Lemma~\ref{lem:conttimechange}, we get that
\begin{equation}\label{eq:titotitanium}
\lim_{K\to\infty}\mathbf{P}^{(K)}\left[\sup_{x\in\frak{T}}\abs{L_t(x)-L^{(K)}_t(x)} \geq \epsilon \right]=0,
\end{equation}
for all $\epsilon>0$.

Let \[\tilde{L}_t(x):=\sigma_d\rho L_{(\sigma_d\rho)^{-1}t}(x).\]
Then we have
\begin{equation}
\begin{aligned}
&\mathbf{P}^{(K)}\left[\sup_{(x,e)\in D(n,K,\delta)} \abs{d^{(n,K)}_{\text{res}}(e)l^{(n,K,\leftrightarrow)}_t(e)-L^{(n,K)}_t(x)} \geq \epsilon \right]\\
\leq &\mathbf{P}^{(K)}\left[\sup_{e\in E(\T^{(n,K)})} \abs{d^{(n,K)}_{\text{res}}(e)l^{(n,K,\leftrightarrow)}_t(e)-\tilde{L}_t(\Upsilon^{-1}_{n,K}(x_e))} \geq \epsilon/3 \right]\\
+&\mathbf{P}^{(K)}\left[\sup_{x,y:d^{(n,K)}_{\text{res}}(x,y)\leq \delta} \abs{\tilde{L}_t(\Upsilon_{n,K}^{-1}(x))-\tilde{L}_t(\Upsilon_{n,K}^{-1}(y))} \geq \epsilon/3 \right]\\
+&\mathbf{P}^{(K)}\left[\sup_{x\in \frak{T}^{(n,K)}} \abs{\tilde{L}_t(\Upsilon_{n,K}^{-1}(x))-L^{(n,K)}_t(x)} \geq \epsilon/3 \right],
\end{aligned} 
\end{equation}
where we have used that if $(x,e)\in D(n,K,\delta)$, then $d^{(n,K)}_{\text{res}}(x_e,x)\leq\delta$.
By the third display in Proposition \ref{l:convergenceoflocaltimes}, \begin{equation}\label{eq:tito1}
\lim_{n\to\infty}\mathbf{P}^{(K)}\left[\sup_{e\in E(\T^{(n,K)})} \abs{d^{(n,K)}_{\text{res}}(e)l^{(n,K,\leftrightarrow)}_t(e)-\tilde{L}_t(\Upsilon^{-1}_{n,K}(x_e))} \geq \epsilon/3 \right]=0.
\end{equation}
Moreover, since $\tilde{L}$ is continuous in the space variable and the Lipschitz norm of $\Upsilon^{-1}_{n,K}$ converges to $(\sigma_d\rho)^{-1}$ as $n\to\infty$ (see Lemma \ref{lem:graphspatialresistancetreeconvergence}), we have that
\begin{equation}\label{eq:tito2}
\lim_{\delta\to0}\mathbf{P}^{(K)}\left[\sup_{x,y:d^{(n,K)}_{\text{res}}(x,y)\leq \delta} \abs{\tilde{L}_t(\Upsilon^{-1}_{n,K}(x))-\tilde{L}_t(\Upsilon^{-1}_{n,K}(y))} \geq \epsilon/3 \right]=0.
\end{equation}
Finally, by Lemma \ref{lem:abstractcoupling} and \eqref{eq:titotitanium}  we get that
 \begin{equation}\label{eq:tito3}
\lim_{K\to\infty}\limsup_{n\to\infty}\mathbf{P}^{(K)}\left[\sup_{x\in \frak{T}^{(n,K)}} \abs{\tilde{L}_t(\Upsilon^{-1}_{n,K}(x))-L^{(n,K)}_t(x)} \geq \epsilon/3 \right]=0.
 \end{equation}
 The proof follows from displays \eqref{eq:titotitanium}, \eqref{eq:tito1}, \eqref{eq:tito2} and \eqref{eq:tito3}.
\end{proof}
\end{lemma}
\begin{proof}[Proof of Lemma \ref{lem:hatankislinear}]

By the proof of Lemma \ref{epsnet0},  for each $\delta>0$ there exists a finite covering of $\frak{T}$ by balls of radius $\delta$. Moreover, by the tree-geometry of $\frak{T}$, it is easy to note that each closed connected set in $\frak{T}$ has a point which is the closest to the root. Let $D_i^{\delta}, i=1,\dots,n(\delta)$ be a delta covering of $\frak{T}$ and $x^{\delta}_i$ be the point in the closure of $D_i^{\delta}$ which is the closest to the root. Obviously $D_i^{\delta}, i=1,\dots,n(\delta)$ will also be a $\delta$-covering of $\frak{T}^{(K)}$, for all $K\geq0$ and, furthermore, if $D_i^\delta\cap \frak{T}^{(K)}\neq \emptyset$, then $x^\delta_i\in\frak{T}^{(K)}$. Let $I_i\subset \frak{T}^{(n,K)}$ be the image of $D_i^\delta \cap \frak{T}^{(K)}$ under the homeomorphism $\Upsilon_{n,K}$, where $\Upsilon_{n,K}$ is as in \eqref{eq:defofupsilon}. Since the Lipschitz norm of $\Upsilon_{n,K}$ converges to $\sigma_d\rho$ as $n\to\infty$ (see Lemma \ref{lem:graphspatialresistancetreeconvergence}), we have that there exists $C$ such that
\begin{equation}\label{eq:nomama}
\sup_{i=1,\dots n(\delta)} \sup_{x,y\in I_i} d^{(n,K)}_{\text{res}}(x,y)\leq C\delta.
\end{equation}
Moreover, by Lemma \ref{lem_no_macro_res} we can redefine the sets $I_i$ such that each edge $e\in E^\ast(\T^{(n,K)})$ is completely contained in one of the $I_i$ (i.e., each $I_i$ will be a union of edges of $\T^{(n,K)}$), $\cup_{i=1}^{n(\delta)} I_i=\frak{T}^{(n,K)}\backslash \tilde{I}$ (where $\tilde{I_i}$ is as in display (3) of Lemma \ref{lem:vcircle}) and display \eqref{eq:nomama} holds for $n$ large enough.
 To start the proof we recall that, since $L^{(n,K)}$ is the local time of $B^{(n,K)}$ with respect to $\lambda^{(n,K)}_{\text{res}}$, we have
\begin{equation}\label{eq:robertocarlos2}
\int_{\T^{(n,K)}} L^{(n,K)}_t(x) \lambda^{(n,K)}_{\text{res}}(dx)=t.
\end{equation}
Recalling display (3) of Lemma \ref{lem:vcircle} we get that
\begin{equation}
\lim_{n\to\infty}\mathbf{P}^{(K)}\left[ \abs{\int_{\T^{(n,K)}\backslash\tilde{I}} L^{(n,K)}_t(x) \lambda^{(n,K)}_{\text{res}}(dx)-t}\geq \epsilon\right]=0.
\end{equation}
We can regard the measure $\mu^{(G_n,K)}$ as defined over the edges of $\T^{(n,K)}$ instead of the vertices. For any $e=(v_1,v_2)\in E^\ast(\T^{(n,K)})$ with $v_1\prec v_2$, we let $\mu^{(n,K)}(e):=\mu^{(n,K)}(v_1)$.  
Therefore, to prove the lemma it is enough to show that
\begin{equation}
\lim_{K\to\infty}\limsup_{n\to\infty} \mathbf{P}^{(K)}\left[\abs{n^{-3/2}\tilde{A}^{n,K}(t)-\nu \int_{\T^{(n,K)}\backslash\tilde{I}} L^{(n,K)}_t(x) \lambda^{(n,K)}_{\text{res}}(dx)} \geq \epsilon \right]=0,
\end{equation}
for all $\epsilon>0$.
Using \eqref{eq:robertocarlos2} and recalling that $\mu^{(n,K)}$ is the edge cardinality scaled by a factor $n^{-1}$, we can write
\begin{align*}
&\abs{n^{-3/2}\tilde{A}^{n,K}(t)-\nu \int_{\T^{(n,K)}\backslash\tilde{I}} L^{(n,K)}_t(x) \lambda^{(n,K)}_{\text{res}}(dx) }\\
\leq&\sum_{i=1}^{n(\delta)}\abs{\int_{I_i}n^{-1/2}\reff^{G_n}(e)l^{(n,K,\leftrightarrow)}_t(e)\mu^{(n,K)}(de)-\nu\int_{I_i}L^{(n,K)}_t(x)\lambda^{(n,K)}_{\text{res}}(dx)}\\
=&\sum_{i=1}^{n(\delta)}\abs{\left[\int_{I_i}n^{-1/2}\reff^{G_n}(e)l^{(n,K,\leftrightarrow)}_t(e)\mu^{(n,K)}(de)\right]-\nu L^{(n,K)}_t(x_i)\lambda^{(n,K)}_{\text{res}}(I_i)},
\end{align*}
where in the last equality we have used the mean value theorem for integrals and $x_i$ is some point in $I_i$. The last display is bounded above by
\begin{equation}\label{eq:plop}
\begin{aligned}
 &\sum_{i=1}^{n(\delta)} \int_{I_i}\abs{ n^{-1/2}\reff^{G_n}(e)l^{(n,K,\leftrightarrow)}_t(e)- L_t^{(n,K)}(x_i) }\mu^{(n,K)}(de)\\
+&\sum_{i=1}^{n(\delta)}L^{(n,K)}_t(x_i)\abs{\nu \lambda^{(n,K)}_{\text{res}}(I_i)-\mu^{(n,K)}(I_i)}.
\end{aligned}
\end{equation}
By Lemma \ref{lem:gluethetimes}, and \eqref{eq:nomama}  for all $\eta> 0$, we can choose $\delta$ such that
\begin{equation}
\lim_{K\to\infty}\limsup_{n\to\infty}\mathbf{P}^{(K)}\left[\sup_{i=1,\dots ,n(\delta)}\sup_{e\in I_i}\abs{n^{-1/2}\reff^{G_n}(e)l^{(n,K,\leftrightarrow)}_t(e)- L_t^{(n,K)}(x_i)}\leq \epsilon/2\right]\leq \eta.
\end{equation}

 Therefore, using that the total mass of $\mu^{(n,K)}$ is smaller than one, we get that 
 \begin{equation}\label{eq:chaneca}
 \lim_{K\to\infty}\limsup_{n\to\infty}\mathbf{P}^{(K)}\left[\sum_{i=1}^{n(\delta)} \int_{I_i}\abs{ n^{-1/2}\reff^{G_n}(e)l^{(n,K,\leftrightarrow)}_t(e)- L_t^{(n,K)}(x_i) }\mu^{(n,K)}(de)\leq \epsilon/2\right]\leq \eta.
 \end{equation}
On the other hand, by condition $(V)_{\nu}$ we have that, letting $z^{\delta}_i$ be the image of $x^\delta_i$ under $\Upsilon_{n,K}$, we have that 
\[
\lim_{K\to\infty}\limsup_{n\to\infty}\mathbf{P}^{(K)}\left[\abs{\mu^{(n,K)}\left(\overrightarrow{\T^{(n,K)}_{z^{\delta}_i} } \right) - \nu \lambda^{(n,K)}\left(\overrightarrow{\T^{(n,K)}_{z^{\delta}_i} } \right)}\geq \epsilon/4 \right] =0
\]
for all $i=1,\dots,n(\delta)$, $\epsilon>0$.   
Moreover, it can be deduced from \eqref{eq:couplinggeometry} and Lemma \ref{lem:graphspatialresistancetreeconvergence} that
\[
\lim_{K\to\infty}\limsup_{n\to\infty}\mathbf{P}^{(K)}\left[\abs{\lambda^{(n,K)}\left(\overrightarrow{\T^{(n,K)}_{z^{\delta}_i} } \right) - \lambda^{(n,K)}_{\text{res}}\left(\overrightarrow{\T^{(n,K)}_{z^{\delta}_i} } \right)}\geq \epsilon/4 \right] =0
\]
for all $i=1,\dots,n(\delta)$, $\epsilon>0$, where we used condition $(R)$ to say that the distance and the resistance distance are proportional which implies that the normalized measures $\lambda^{(n,K)}$ and $\lambda^{(n,K)}_{\text{res}}$.

Therefore, from the two displays above,
\begin{equation}\label{eq:nomorelabelsplease}
\lim_{K\to\infty}\limsup_{n\to\infty}\mathbf{P}^{(K)}\left[\abs{\mu^{(n,K)}\left(\overrightarrow{\T^{(n,K)}_{z^{\delta}_i} } \right) - \nu \lambda^{(n,K)}_{\text{res}}\left(\overrightarrow{\T^{(n,K)}_{z^{\delta}_i} } \right)}\geq \epsilon/2 \right] =0
\end{equation}
for all $i=1,\dots,n(\delta)$, $\epsilon>0$.   
Note that each $I_i$ can be expressed as
\begin{equation}\label{eq:morelabels!}
I_i=\overrightarrow{\T^{(n,K)}_{x} }\backslash\left(\cap_{i=1}^k \overrightarrow{\T^{(n,K)}_{x_i} } \right)
\end{equation}
for some $k$ and $x,x_i\in \frak{T}^{(n,K)}$. Since, from \eqref{eq:nomorelabelsplease}, we can control the difference between $\mu^{(n,K)}$ and $\lambda^{(n,K)}_{\text{res}}$-measures for the sets of the form $\overrightarrow{\T^{(n,K)}_{x} }, x\in\frak{T}^{(n,K)}$ and by display \eqref{eq:morelabels!} the sets $I_i$ can be expressed as (finite) differences and intersections of sets of the said type, it follows that for each $\epsilon>0$,
\[\lim_{K\to\infty}\limsup_{n\to\infty}\mathbf{P}^{(K)}\left[\abs{\mu^{(n,K)}(I_i) - \nu \lambda^{(n,K)}_{\text{res}}(I_i))} \geq \epsilon /2\right] =0\]
for all $i=1,\dots,n(\delta)$. This, and the fact that $L^{(n,K)}_t$ is uniformly bounded in $K$ and $n$ gives that
\[\lim_{K\to\infty}\limsup_{n\to\infty}\mathbf{P}^{(K)}\left[\sum_{i=1}^{n(\delta)}L^{(n,K)}_t(x_i)\abs{\nu \lambda^{(n,K)}_{\text{res}}(I_i)-\mu^{(n,K)}(I_i)}\geq \epsilon/2 \right]=0.\]
The display above, together \eqref{eq:plop} and the fact that $\eta$ is arbitrary in \eqref{eq:chaneca}, proves the lemma. 
\end{proof}

\appendix

\section{Variance estimate on the commute time}

In this section, we shall prove a formula which leads to an upper-bound the second moment of the commute time in a general electrical network (see~\cite{lyons2005probability} for an introduction to this field).

Recall that in a finite electrical network $(G,c(\cdot))$ the invariant measure is given by $\pi(x)=\sum_{x\sim y} c(x,y)$. 

In this section we denote $T_x:=\min\{n \geq 0, X_n=x\}$, the hitting time of $x\in G$ for the random walk $(X_n)_{n\geq 0}$ naturally associated to the aforementioned electrical network.

\begin{proposition}\label{var_bound}
Let $(G,c(\cdot))$ be a finite connected electrical network. Fix two distinct vertices $x,y\in G$, then
\[
\frac{E_x[T_y^2]+E_y[T_x^2]}{(E_x[T_y]+E_y[T_x])^2}\leq 2 \frac{E_{\pi}\bigl[R_{\text{eff}}(x, \cdot)+R_{\text{eff}}(\cdot, y)\bigr]}{R_{\text{eff}}(x, y)},
\]
where
\[
E_{\pi}\bigl[R_{\text{eff}}(x, \cdot)+R_{\text{eff}}(\cdot, y)\bigr]:=\sum_{z\in G} \Bigl(\frac{\pi(z)}{\sum_{z\in G} \pi(z)}\Bigr)\bigl(R_{\text{eff}}(x, z)+R_{\text{eff}}(z, y)\bigr).
\]
\end{proposition}

By using the commute time formula $E_x[T_y]+E_y[T_x]=\Bigl(\sum_{z\in G} \pi(z)\Bigr)R_{\text{eff}}(x, y)$ (see~\cite{chandra1996electrical}) and noticing that, by Rayleigh's monotonicity principle, for any $z,z'\in G$ we have $R_{\text{eff}}(z,z')\leq \text{diam}(G)$ in any graph with unit conductances, we obtain the following corollary.
\begin{corollary}\label{variance_bound}
Let us consider a simple random walk on a finite connected graph $G$. Fix two distinct vertices $x,y\in G$, then
\[
E_x[T_y^2]+E_y[T_x^2]\leq 16 \abs{E(G)}^2\text{diam}(G) R_{\text{eff}}(x, y).
\]
\end{corollary}

Let us prove Proposition~\ref{var_bound}
\begin{proof}
First, we will estimate $E_x[T_y^2]$. In order to do this, we introduce for $z\in G$ the random variable $N(z)=\text{card}\{i\leq T_y,\ X_i=z\}$. Obviously, we have
\begin{equation}\label{var_com0}
T_y=\sum_{z\in G} N(z).
\end{equation}

This means that
\begin{equation} \label{var_com1}
E_x[T_y^2]=E_x\Bigl[\bigl(\sum_{z\in G} N(z)\bigr)^2\Bigr] \leq \Bigl(\sum_{z\in G} E_x\bigl[ N(z)^2\bigr]^{1/2}\Bigr)^2,
                    \end{equation}
                    by the $L^2$-triangular inequality.

Now, let us notice that for any $k\geq 1$
\begin{equation}\label{eq:Nisgeometric}
P_x[N(z)=k]=P_x[T_z<T_y](1-P_z[T_y<T_z^+])^{k-1}P_z[T_y<T_z^+].
\end{equation}

Then, an elementary computation shows that
\[
E_x[ N(z)]=\frac{P_x[T_z<T_y]}{P_z[T_y<T_z^+]},
\]
and
\begin{align}\label{var_com2}
E_x\bigl[ N(z)^2\bigr]=P_x[T_z<T_y] \frac{1+P_z[T_y<T_z^+]}{(P_z[T_y<T_z^+])^2} & \leq 2\frac{P_x[T_z<T_y]}{(P_z[T_y<T_z^+])^2}\\\nonumber &=2\frac{E_x[ N(z)]}{P_z[T_y<T_z]}.
\end{align}

We rewrite $(P_z[T_y<T_z^+])^{-1}=\pi(z)R_{\text{eff}}(z, y)$ and use~\eqref{var_com2} in~\eqref{var_com1} to see that
\begin{align*}
E_x[T_y^2] &\leq 2 \Bigl(\sum_{z\in G} \bigr(E_x[ N(z)]\pi(z)R_{\text{eff}}(z, y)\bigr)^{1/2}\Bigr)^2 \\
                     &\leq 2  \Bigl(\sum_{z\in G} E_x[ N(z)]\Bigr)\Bigl(\sum_{z\in G} \pi(z)R_{\text{eff}}(z, y)\Bigr) \\
                     &\leq 2 E_x[T_y] \Bigl(\sum_{z\in G} \pi(z)R_{\text{eff}}(z, y)\Bigr),
                     \end{align*}
where we used the Cauchy-Schwarz inequality and also~\eqref{var_com0}.

A similar formula can be obtained with $E_y[T_x^2]$. Combining those two formulas and using the  commute time formula (see~\cite{chandra1996electrical}), we obtain the result.
\end{proof}
\begin{proposition}\label{prop:finitemoments} Let $G$ be a finite, connected graph. For any $x,y\in G$ we have that
\[E_x[T_y^4]<\infty\]
\end{proposition}
\begin{proof}[Proof of Proposition \ref{prop:finitemoments}]
Using the same reasoning and notation as in the proof of Proposition \ref{var_bound} we get that
\begin{align}
&E_x[T_y^4]=E_x\left[\left(\sum_{z\in G} N_z\right)^4\right].\\
&\leq \left(\sum_{z\in G} E_x\left[ N_z^4 \right]^{1/4}\right)^{4},
\end{align}
where we have used the $L^4$ triangular inequality.
Therefore, to prove the lemma it us enough to show that
\[ E_x\left[ N_z^4 \right]<\infty\]
for all $x,z\in G$. But the display above follows easily, since by equation \eqref{eq:Nisgeometric}, $N_z$ is stochastically dominated by a geometric random variable.
\end{proof}

\section{Fourth moment bound}
Let $(X_k)_{k\in\N}$ be a sequence of i.i.d.~sequence of random variables adapted to a filtration $(\mathcal{F}_k)_{k\in\N}$ and $\tau$ a stopping time with respect to  $(\mathcal{F}_k)_{k\in\N}$.  Denote $m_k:=E[X_1^k]$ and $\Psi(z):=E[\exp(z X_1)]$.
\begin{lemma}\label{lem:martingaleinequality} 
Assume that for all $z$ in a neighborhood of $0$ we have
 \begin{equation}\label{eq:optionalstopping}
 E\left[ \Psi(z)^{-\tau}\exp\left(z\sum_{i=1}^{\tau}X_i\right)\right]=1.  
 \end{equation}
 Then we have 
\begin{equation}\label{eq:firstmoment}
E\left[\sum_{i=1}^{\tau}X_i\right]=m_1E[\tau],
\end{equation}
If, in addition, $m_1=0$, then there exists $C$ not depending on the distribution of $(X_i)_{i\in\N}$ and $\tau$ such that 
\begin{equation}\label{eq:fourthmomentfinal}
E\left[\left(\sum_{i=1}^{\tau}X_i\right)^4\right] \leq C \left (m_2^2E[\tau^2] + m_4E[\tau]\right).
 \end{equation}
\end{lemma}
\noindent Note that the inequality in display \eqref{eq:fourthmomentfinal} is the same (modulo a constant) as the one we would have gotten if $\tau$ were independent of $(X_i)_{i\in\N}$. 
 \begin{proof}
Differentiating the moment generating function in $\eqref{eq:optionalstopping}$ with respect to $z$ (and evaluating at $z=0$) we get 
\[
E\left[\sum_{i=1}^{\tau}X_i\right]=m_1E[\tau],
\]
which proves the first claim of the Lemma. To prove the remaining claim we computer higher order derivatives, and using the relation $m_1=0$ we find
\begin{equation}\label{eq:secondmoment}
E\left[\left(\sum_{i=1}^{\tau}X_i\right)^2\right]=m_2E[\tau].
\end{equation}

\begin{align*}\label{eq:fourthmoment}
E\left[\left(\sum_{i=1}^{\tau}X_i\right)^4\right]&=6m_2E\left[\tau\left(\sum_{i=1}^{\tau}X_i\right)^2\right] +4m_3 E\left[\tau\sum_{i=1}^{\tau} X_i \right]\\
&\qquad +(m_4-3m_2^2)E[\tau]-3m_2^2E[\tau^2].
\end{align*}

Applying Holder's inequality to the right hand side of the last display above (and dropping the negative terms) we get a quadratic inequality in terms of $E^{\omega}[(\sum_{i=1}^{\tau}X_i)^4]^{1/2}$, of the form $x^2\leq \alpha x + \beta$ with
\[\alpha= 6m_2 E^\omega[\tau^2]^{1/2}\]
and
\[
\beta= 4m_3E[\tau^2]^{1/2}E\left[\left(\sum_{i=1}^{\tau} X_i\right)^2\right]^{1/2}+m_4E[\tau].
\]

We can see that $x^2\leq 4\alpha^2+2\beta$ (by looking at the cases $x\leq 2\alpha$ and $x\geq 2 \alpha$ which implies that $x\leq \beta/\alpha$). In our context, this inequality reads 
 \begin{equation}\label{eq:betterthanborodin}
E\left[\left(\sum_{i=1}^{\tau}X_i\right)^4\right] \leq 144m_2^2E[\tau^2] +8m_3E[\tau^2]^{1/2}E\left[\left(\sum_{i=1}^{\tau} X_i \right)^2\right]^{1/2}+2m_4E[\tau].
 \end{equation}
 
Using Holder's inequality we see that $m_3\leq m_2^{1/2}m_4^{1/2}$. That, plus formula \eqref{eq:secondmoment} yields that the second summand in the right hand side above is bounded by
\[
8m_2m_4^{1/2} E[\tau^2]^{1/2}E[\tau]^{1/2},
\]
which is smaller than the geometric mean between the first and third summands of the right hand side of \eqref{eq:betterthanborodin}. Finally, since geometric means are smaller than arithmetic means we get from \eqref{eq:betterthanborodin} that there exists $C$ such that
 \begin{equation}
E\left[\left(\sum_{i=1}^{\tau}X_i\right)^4\right] \leq C \left (m_2^2E[\tau^2] + m_4E[\tau]\right),
 \end{equation}
 which proves the lemma.

\end{proof}

\section*{Glossary of notations}

\vspace{0.5cm}

\def\qq{&}
\begin{center}
\halign{
#\quad\hfill&#\quad\hfill&\quad\hfill#\cr\
$\mathfrak{T}$ \qq The Continuum random tree (CRT) \fff{crt}
$\phi_{\mathfrak{T}}(\mathfrak{T})$ \qq  integrated super-Brownian excursion (ISE) \fff{ise}
$B^{ISE}$  \qq Brownian motion on the ISE \fff{bise}
$B^{CRT}$  \qq Brownian motion on the CRT \fff{bise}
$\mathfrak{T}^{(K)}$ \qq $K$-CRT \fff{kcrt}
$B^{(K)}$ \qq Brownian motion on $K$-CRT \fff{bcrtk}
$B^{K-ISE}$\qq Brownian motion on $K$-ISE \fff{bisek} 
$L^{(K)}$ \qq Local time of the $B^{K-ISE}$ \fff{ltk}
$(G_n,(V_i^n)_{i\in \N})_{n\in \N}$ \qq sequence of random augmented graphs \fff{gv}
$\T^{(n,K)}$ \qq  $K$-skeleton of our graphs $G_n$ \fff{mass}
$\text{root}^*$ \qq root of the skeleton of $\T^{(n,K)}$ \fff{rootstar}
$\mathfrak{T}^{(n,K)}$ \qq  version of $\T^{(n,K)}$ with fewer vertices \fff{mass}
$V^*(\T^{(n,K)})$ \qq  vertices of $\T^{(n,K)}$ corresponding to cut-points \fff{mass}
$V^\circ(\T^{(n,K)})$ \qq vertices of $\T^{(n,K)}$ that are not leaves or branching points\fff{vcirc}
$E^*(\T^{(n,K)})$ \qq  edges of $\T^{(n,K)}$  between vertices of $V^*(\T^{(n,K)})$ \fff{mass}
$\phi^{(n,K)}$ \qq random embedding of $\T^{(n,K)}$ \fff{mass}
$v^{(n,K)}(x)$ \qq volume of the sausage attached at $x\in V^*(\T^{(n,K)})$ \fff{mass}
$\reff^{(n,K)}(e)$ \qq resistance of the edge $e \in E(\T^{(n,K)})$ \fff{mass}
$\pi^{(n,K)}(x)$ \qq projection of $x\in G_n$ onto $\T^{(n,K)}$ \fff{mass}
$d^{(n,K)}(\cdot,\cdot)$ \qq rescaled intrinsic distance on $\T^{(n,K)}$ \fff{mass}
 $d^{(n,K)}_{\text{res}}(\cdot,\cdot)$ \qq rescaled resistance distance on $\T^{(n,K)}$ \fff{mass}
  $\mu^{(n,K)}$ \qq volume measure on $\T^{(n,K)}$ \fff{mass}
  $\lambda^{(n,K)}_{\text{res}}$ \qq the unit Lebesgue measure on $\T^{(n,K)}$ w.r.t.~resistance \fff{lambdares}
  $\lambda_{T}$  \qq renormalized Lebesgue measure on $T$ \fff{lambdaT}
$\Delta^{(n,K)}_{\Z^d}$ \qq maximal $\Z^d$-diameter of sausages \fff{delta1}
$\Delta^{(n,K)}_{\Z^d}$ \qq maximal $\omega_n$-diameter of sausages \fff{delta2}
$B^{(n,K)}$ \qq Brownian motion on $\mathfrak{T}^{(n,K)}$  with resistance metric \fff{bnk}
$h^{(n,K)}_m$ \qq successive visits of $B^{(n,K)}$ to $V^*(\T^{(n,K)})$ \fff{h}
$h^{(n,K)}_m$ \qq successive visits of $X^{G_n}$ to $V^*(\T^{(n,K)})$ \fff{ank}
$\overrightarrow{\T^{(n,K)}_{x}}$ \qq are the descendants of $x$ in $\T^{(n,K)}$ \fff{arrowt}

\cr}\end{center}

\bibliographystyle{plain}
\bibliography{RWBRW}

\end{document}